\RequirePackage{fix-cm}
\documentclass[smallextended]{svjour3}       
\smartqed  

\usepackage{graphicx}
\usepackage{amsmath}
\usepackage{amssymb}
\usepackage{amsfonts}
\usepackage{mathtools}
\usepackage{bbm}
\usepackage{algorithm}
\usepackage{algpseudocode}
\usepackage{tikz}
\usetikzlibrary{positioning,fit,shapes, arrows, positioning, arrows.meta}
\usepackage{circuitikz}
\usepackage{pgfplots}
\usepackage{subcaption}
\usepackage{booktabs}
\usepackage{multirow}
\usepackage{array}
\usepackage{hyperref}
\usepackage{textcomp}
\usepackage{appendix}
\usepackage{fancyvrb}
\usepackage{verbatimbox}
\usepackage{makeidx}

\newcommand{\bft}[1]{\mathbf{#1}}
\DeclareMathOperator*{\argmin}{argmin}

\newcommand{\R}{\mathbb{R}}

\pgfplotsset{compat=1.18}

\begin{document}

\title{Automating Reformulation for Parallel ADMM
}


\author{Kaizhao Sun  \and
        Baihao Wu    \and 
        Kun Yuan     \and
        Wotao Yin
}


\institute{K. Sun \at
DAMO Academy, Alibaba Group (U.S.) Inc. \\
\email{kaizhao.s@alibaba-inc.com}  
\and
B. Wu \at
Academy for Advanced Interdisciplinary Studies, Peking University \\
\email{baihaowu@pku.edu.cn}  
\and 
K. Yuan \at 
Center for Machine Learning Research, Peking University\\ 
\email{kunyuan@pku.edu.cn}  
\and 
W. Yin \at 
DAMO Academy, Alibaba Group (U.S.) Inc. \\
\email{wotao.yin@alibaba-inc.com}  
}

\date{Received: date / Accepted: date}

\maketitle

\begin{abstract}
Many real-world optimization models contain exploitable sparsity and block structure, but this structure is often obscured in algebraic form, limiting the effectiveness of modern parallel algorithms. We propose an automatic pipeline that converts a generic multiblock problem into a canonical two-block formulation suitable for parallel Alternating Direction Method of Multipliers (ADMM). The method constructs a coupling graph, applies an edge-subdivision–based bipartization to obtain a bipartite representation, and produces an ADMM-ready decomposition with independent subproblems. Fast graph-traversal heuristics, a new mixed-integer linear program (MILP), and a learning-based graph neural network (GNN) surrogate model are developed to guide edge subdivision. Numerical experiments demonstrate that the resulting reformulations yield strong parallel ADMM performance. The entire pipeline is implemented in the open-source Julia package \texttt{PDMO.jl}.
\end{abstract}

\section{Introduction}
Real-world optimization models rarely arrive in ``clean'' algebraic formats commonly used in textbooks and in many algorithm papers. In practice, variables naturally group into \emph{blocks}, which may have different sizes, while there are constraints \emph{coupling} two or more blocks through shared linear relations. Writing such models in a generic algebraic format, typically with only a few aggregated variable and constraint groups for convenience, can obscure the underlying sparsity and interaction structure that modern large-scale or parallel algorithms rely on.

\paragraph{The obstacle: structure for splitting methods is hard to craft by hand.}
For first-order splitting methods, the main difficulty is not merely specifying a model, but \emph{reformulating} the original problem into a form that yields simple, separable subproblems and a clear communication pattern. In this paper, we use ``decoupling'' to mean introducing auxiliary variables and consistency constraints so that the resulting formulation is \emph{equivalent} to the original model but exposes a more favorable block structure to the algorithm. This issue arises for classical two-block Alternating Direction Method of Multipliers (ADMM) and also for its more general variant Function-Linearized Proximal ADMM (FLiP-ADMM) \cite{ryu2022large}: while FLiP-ADMM offers additional flexibility (e.g., allowing gradient-based and/or linearized treatment inside subproblems), it still starts from the same backbone of two objective blocks connected by linear constraints, and therefore benefits from the same structural preparation. These considerations motivate an automated pipeline that turns a generic multiblock model into an ``ADMM/FLiP-ready'' two-block decomposition, rather than relying on ad hoc, expert-crafted reformulation.

\begin{figure}[ht]
\centering
\begin{minipage}{0.4\textwidth}
\centering
\begin{circuitikz}[american, thick]
\node[circle, fill=black, inner sep=2pt, label=right:{Node 1}] (N1) at (90:2) {};
\node[circle, fill=black, inner sep=2pt, label=below left:{Node 2}] (N2) at (210:2) {};
\node[circle, fill=black, inner sep=2pt, label=below right:{Node 3}] (N3) at (330:2) {};
\draw
  (N1) to[R=$R_{1}$, *-*, i>_=$I_{1}$] (N2)
  (N2) to[R=$R_{2}$, *-*, i>_=$I_{2}$] (N3)
  (N3) to[R=$R_{3}$, *-*, i>_=$I_{3}$] (N1);
\draw[->] (0,3) -- node[right] {$J_1$} (N1);
\draw[->] (-3,-1) -- node[above] {$J_2$} (N2);
\draw[->] (3,-1) -- node[above] {$J_3$} (N3);
\end{circuitikz}
\end{minipage}
\hfill
\begin{minipage}{0.45\textwidth}
\begin{align*}
\min_{I_{1}, I_{2}, I_{3}}\quad & R_{1} I_{1}^2 + R_{2}I_{2}^2 + R_{3}I_{3}^2 \\
\mathrm{s.t.}\quad & I_{1} - I_{3} = J_1 \\
& I_{2} - I_{1} = J_2 \\
& I_{3} - I_{2} = J_3
\end{align*}
\end{minipage}
\caption{A triangle resistive network and power dissipation minimization.}\label{fig: demo}
\end{figure}
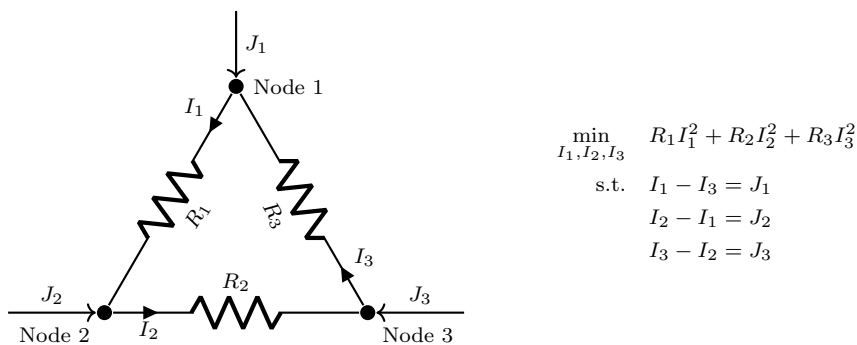

\paragraph{A small example: equivalent splittings, different behavior.}
We begin with a simple circuit model to illustrate (i) how equivalent decoupling transformations arise naturally and (ii) why their choice can matter algorithmically.

\begin{example}\label{example: demo}
Consider a resistive circuit consisting of three nodes connected in a triangle by resistors $R_{1}$, $R_{2}$, and $R_{3}$. Directed edges carry branch currents $I_{1}$, $I_{2}$, and $I_{3}$, respectively; each node $i\in \{1,2,3\}$ has an external injection current $J_i$ entering or leaving the network. The objective is to determine the branch currents that minimize total power dissipation subject to Kirchhoff's Current Law (KCL) at each node; see Fig. \ref{fig: demo}.

This formulation is already small, but it exhibits the key obstruction: each KCL constraint couples two currents, and the coupling pattern forms a triangle. To prepare such a model for two-block splitting methods, one can decouple a selected coupling by introducing an auxiliary variable. For example, breaking the coupling in the first constraint via an auxiliary variable $z=I_1$ yields
\begin{subequations}\label{eq: circuit demo}
\begin{align}
\min_{I_{1}, I_{2}, I_{3},z}\quad &  R_{1} I_{1}^2 + R_{2}I_{2}^2 + R_{3}I_{3}^2 \\
\mathrm{s.t.} \quad &
\begin{bmatrix}
1 &   \\
1 &   \\
  & 1 \\
  & 1 \\
\end{bmatrix}
\begin{bmatrix}
z \\ I_{2}
\end{bmatrix}
+
\begin{bmatrix}
-1 &   \\
   & -1   \\
-1 &    \\
   & -1\\
\end{bmatrix}
\begin{bmatrix}
I_{1} \\ I_{3}
\end{bmatrix}
=\begin{bmatrix}
0\\ J_1 \\J_2 \\-J_3
\end{bmatrix}.
\end{align}
\end{subequations}
This is an \emph{equivalent} reformulation in the usual splitting sense: the added variable $z$ is tied to the original variable $I_1$ by consistency constraints ($z-I_1=0, z-I_3 = J_1$) so that eliminating $z$ recovers the original KCL system, but this explicit split changes the way couplings are presented to an algorithm. Instead of this split, different ones can be made to decouple the second ($z - I_2=0, z-I_1 = J_2$) or the third constraint ($z - I_3=0, z-I_2=J_3$).

We apply ADMM\footnote{See Section \ref{subsec: admm} for algorithmic background.} with different penalty parameters $\rho$ to each of the three resulting reformulations, and plot the trajectories of the sum of primal and dual residuals in Fig.~\ref{figure: demo_trajectory}. (Here the primal residual measures constraint violation, and the dual residual measures successive changes in the dual/consensus quantities, as in standard ADMM diagnostics.) The trajectories show that different but equivalent decouplings can lead to markedly different convergence behavior. This motivates automating the choice of decoupling structure, rather than relying on manual trial-and-error.
\begin{figure}[ht]
\centering
\begin{subfigure}{0.32\textwidth}
\centering
\includegraphics[width=\linewidth]{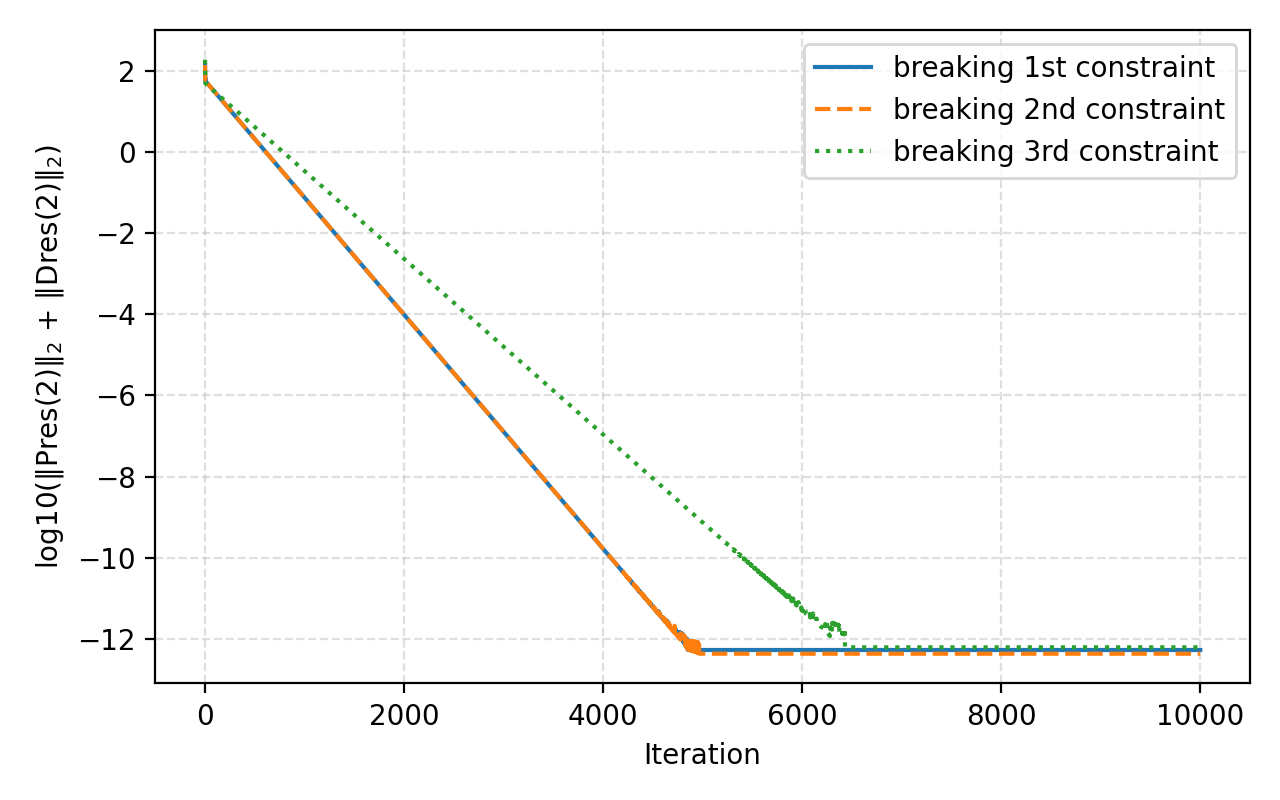}
\caption{$\rho=1$}
\end{subfigure}
\hfill
\begin{subfigure}{0.32\textwidth}
\centering
\includegraphics[width=\linewidth]{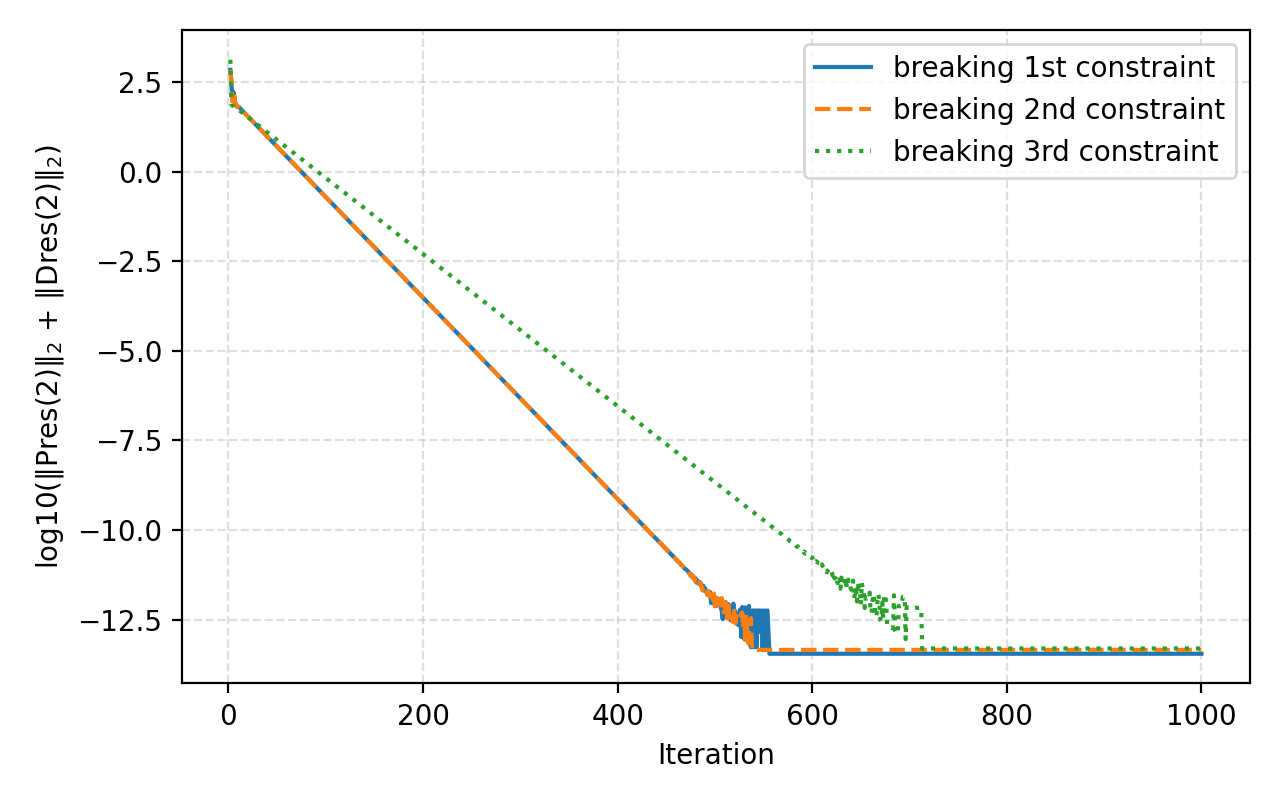}
\caption{$\rho=10$}
\end{subfigure}
\hfill
\begin{subfigure}{0.32\textwidth}
\centering
\includegraphics[width=\linewidth]{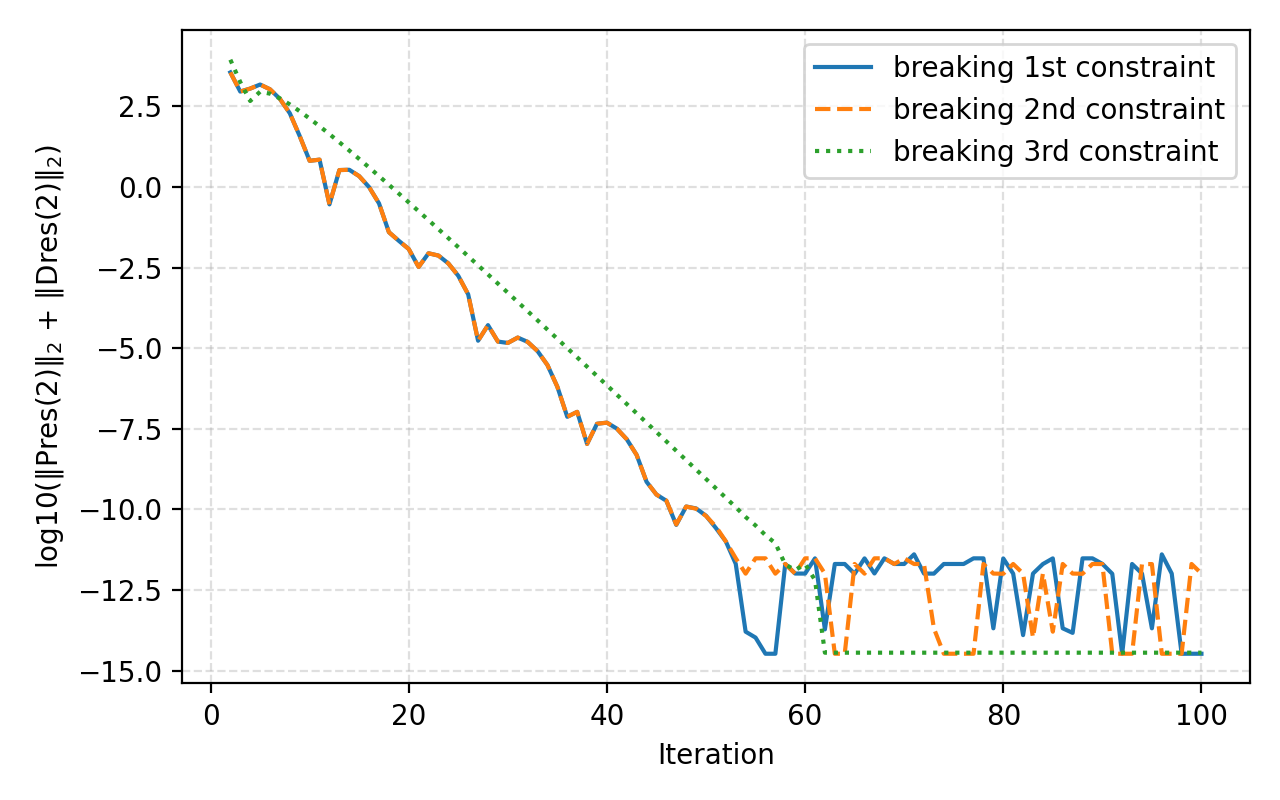}
\caption{$\rho=100$}
\end{subfigure}
\caption{Residual trajectories of ADMM with $(R_{1}, R_{2}, R_{3})$ = (1.0e-6, 1.0e2, 1.0e8) and $(J_1, J_2, J_3) = (-50,100,-50)$.}\label{figure: demo_trajectory}
\end{figure}
\end{example}

The example suggests a broader principle: for optimization problems whose variable blocks are coupled by linear constraints, the key design choice is the \emph{pattern of block couplings} induced by linear constraints. We capture this pattern using a \emph{coupling hypergraph}, whose vertices represent variable blocks and whose edges represent linear couplings between blocks. We then apply an \emph{edge-subdivision} procedure: subdividing a selected edge inserts an intermediate vertex, corresponding to an auxiliary block/variable, and replaces a direct coupling by two simpler couplings through that auxiliary, in the same spirit as introducing $z$ in Example~\ref{example: demo}. By subdividing an appropriate set of edges, we obtain a \emph{bipartite} coupling graph, which is naturally aligned with two-block splitting: the two parts define the two variable groups, while the bipartite edges define the linear couplings between them. This bipartite structure is precisely what enables efficient parallel updates, e.g., $(z, I_2)$ or $(I_1,I_3)$ in \eqref{eq: circuit demo}, in classic two-block ADMM, and it also provides the structural backbone needed by more flexible ADMM-type templates such as FLiP-ADMM.

Building on these ideas, we propose an automatic reformulation pipeline that converts a generic multiblock model into a canonical two-block form suitable for parallel ADMM and FLiP-ADMM. As illustrated in Fig.~\ref{fig: diagram}, the process extracts block interactions, constructs a coupling graph, enforces bipartiteness via selective edge subdivision, and produces a parallel-ready splitting. Conceptually, this serves as a \emph{compiler-like pipeline} that bridges modeling and algorithmic execution by translating structural information into a convergence-compatible formulation. Rather than introducing a new ADMM variant, our goal is to automate the structural reformulation that practitioners typically craft by hand when preparing a complex model for ADMM-type methods.

\begin{figure}[ht]
\centering
\begin{tikzpicture}[
    box/.style={draw, thick, inner sep=0pt}
]
\node[box] (B1) at (0,0)
    {\begin{tabular}{c}
        Multiblock \\
        model
     \end{tabular}
    };
\node[box] (B2) at (3.0,0)
    {\begin{tabular}{c}
        Coupling-graph \\
        construction
     \end{tabular}
    };
\node[box] (B3) at (6.3,0)
    {\begin{tabular}{c}
        Bipartization via \\
        edge subdivision
     \end{tabular}
    };
\node[box] (B4) at (9.5,0)
    {\begin{tabular}{c}
        Run ADMM \\
        in parallel
     \end{tabular}
     };
\draw[->, thick] (B1) -- (B2);
\draw[->, thick] (B2) -- (B3);
\draw[->, thick] (B3) -- (B4);
\end{tikzpicture}
\caption{Proposed bipartization pipeline.}\label{fig: diagram}
\end{figure}
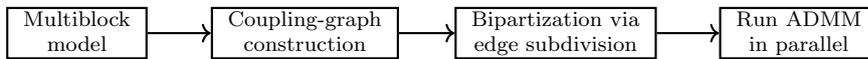
The term ``bipartization” has no universally accepted formal definition in graph theory. It has been used in the literature with different meanings depending on the context, such as edge deletion \cite{pilipczuk2019edge}, vertex deletion \cite{choi1989graph}, or structural transformations aimed at enforcing bipartiteness \cite{miotk2019bipartization}. In this paper, we adopt a specific meaning: bipartization refers to the process of converting a graph into a bipartite graph via edge subdivision, i.e., by inserting auxiliary vertices along selected edges to eliminate odd cycles.

\paragraph{Contributions preview.}
Concretely, this paper (i) formalizes a general multiblock abstraction to capture block-structured models, (ii) develops algorithmic bipartization strategies (heuristic, optimization-based, and learning-based) to produce bipartite/two-block decompositions, and (iii) validates the impact of these reformulations through numerical experiments and an open-source Julia implementation \texttt{PDMO.jl}.

In the next two subsections, we first provide algorithmic background on ADMM to clarify the target two-block structure, and then review related work on problem reformulation and decomposition to position our contributions.

\subsection{Background on ADMM and FLiP-ADMM} \label{subsec: admm}
ADMM~\cite{gabay1976dual,glowinski1975approximation} has emerged as a versatile operator splitting method supported by a mature theoretical foundation~\cite{boyd2011distributed,ryu2022large,wang2019global} and strong empirical performance~\cite{douglas1956numerical,peaceman1955numerical,shi_linear_2014}. Its appeal in our context is structural: when a model is cast into a two-block form with suitable separability, ADMM updates can decompose into many independent subproblems that are solved in parallel, with coordination handled through simple linear couplings.

Consider an optimization problem of the form
\begin{align}\label{eq: two_block_problem}
    \min_{(\bft{x}, \bft{z})}
    \{ f(\bft{x}) + h(\bft{z})~|~ \bft{A}\bft{x} + \bft{B}\bft{z} = \bft{b}, ~\bft{x}\in X, ~\bft{z} \in Z \}.
\end{align}
The associated augmented Lagrangian is
\begin{align}
L_{\rho}(\bft{x}, \bft{z}, \bft{\lambda})
:= f(\bft{x}) + h(\bft{z})
+ \langle \bft{\lambda}, \bft{A}\bft{x} + \bft{B}\bft{z} - \bft{b}\rangle
+ \frac{\rho}{2}\|\bft{A}\bft{x} + \bft{B}\bft{z} - \bft{b}\|^{2},
\end{align}
where $\rho>0$ is the penalty parameter and $\bft{\lambda}$ is the Lagrange multiplier.
A standard ADMM iteration consists of the following updates:
\begin{subequations}\label{eq: admm}
\begin{align}
    \bft{x}^{k+1} &= \argmin_{\bft{x}\in X}
    L_{\rho}(\bft{x}, \bft{z}^{k}, \bft{\lambda}^{k}), \label{eq: admm_x}\\
    \bft{z}^{k+1} &= \argmin_{\bft{z}\in Z}
    L_{\rho}(\bft{x}^{k+1}, \bft{z}, \bft{\lambda}^{k}), \label{eq: admm_z}\\
    \bft{\lambda}^{k+1}
    &= \bft{\lambda}^{k} + \rho (\bft{A}\bft{x}^{k+1} + \bft{B}\bft{z}^{k+1} - \bft{b}).
\end{align}
\end{subequations}
When $\bft{A}$ (respectively $\bft{B}$) admits a block-diagonal structure and $f$ (respectively $h$) is separable across the corresponding subblocks, the update \eqref{eq: admm_x} (respectively \eqref{eq: admm_z}) decomposes into multiple independent subproblems that can be solved fully in parallel. This is the structural property our bipartization pipeline is designed to induce from a generic multiblock input model.

However, most practical optimization problems do not naturally appear in the two-block form \eqref{eq: two_block_problem}, nor do they exhibit the block-diagonal/separable structure needed for efficient parallel ADMM. Although prior work on ADMM~\cite{hong2016convergence,wang2019global}, primal--dual methods~\cite{condat2013primal,vu2013splitting}, and block coordinate descent~\cite{tseng1997alternating,xu2013block} characterizes structural patterns that admit parallel updates, these methods generally assume that such structure is already present or can be crafted by experts. The main challenge is thus to reformulate a generic multiblock model into a representation that inherits the properties required by two-block splitting.

We also note that ADMM is only one instance of a broader family. In particular, FLiP-ADMM (see Section 8 of \cite{ryu2022large}) provides a more general template that can incorporate gradient-based steps and/or linearizations of terms inside the ADMM subproblems, offering additional modeling and computational flexibility. Well-known methods such as linearized ADMM and primal-dual hybrid gradient (PDHG) can be viewed as special cases of this FLiP-ADMM template. Our reformulations target the shared structural backbone---two objective blocks connected by linear constraints---so that practitioners can benefit from either classical ADMM updates or more flexible FLiP-ADMM-style subproblems after bipartization.

Finally, though many ADMM extensions to multiblock settings \cite{chen2016direct,davis2017three,gonccalves2016extending,sun2024dual} have been proposed, their convergence typically requires additional assumptions that may fail for general formulations. In contrast, the standard two-block ADMM \eqref{eq: admm} enjoys a more mature theoretical foundation, and a wide range of practical enhancements has been developed for this setting \cite{pollock2023filtering,mccann2024robust}. These considerations further motivate targeting a canonical two-block structure.

\subsection{Related Work on Problem Reformulation and Decomposition}
A substantial literature studies how to restructure models to improve scalability, decomposability, or solver performance. Our work is algorithm-aware in a specific sense: it aims to automate reformulations that produce bipartite/two-block coupling patterns tailored to parallel ADMM-type methods.

Several works exploit graph structure when it is intrinsic or explicitly provided. Shin et al.~\cite{shin2020decentralized} study decentralized solution schemes for graph-structured problems, e.g., optimal power flow (OPF), using overlapping domain decomposition and coordination, while \texttt{Plasmo.jl}~\cite{jalving2022graph} provides a graph-based modeling environment where the user explicitly encodes subsystems and constraints couplings. In contrast, our pipeline starts from a generic algebraic model and automatically extracts and manipulates the coupling structure to obtain a bipartite form suitable for ADMM.

Automated structure detection has also been developed for classical decomposition techniques. Generic Column Generation (\texttt{GCG})~\cite{gamrath2010experiments} identifies decompositions suitable for Dantzig--Wolfe~\cite{vanderbeck2006generic} and Benders~\cite{benders1962partitioning} reformulations by analyzing polyhedral structure and constraint linking patterns. 
A recent work \cite{naik2025variable} investigates automatic variable aggregation as a presolve procedure for nonlinear programs;
using a graph-based incidence representation, the authors design aggregation strategies that trade off dimensionality reduction against increased nonlinear coupling. 
While these methods demonstrate the value of automated decomposition, they target different decomposition paradigms than the bipartite/two-block structure sought by ADMM-type methods.

Another line of work exploits block structure at the numerical linear-algebra level. For example, \cite{rodriguez2020scalable} applies ADMM-based ideas to decompose and precondition block-structured KKT or saddle-point systems, and \cite{kang2014interior} discusses implicit Schur-complement decompositions and hybrid solvers within interior-point methods. These approaches operate after a formulation is fixed and focus on accelerating linear algebra kernels, whereas our work reformulates the model itself to expose parallel subproblems to first-order methods.

Disciplined convex programming (DCP) frameworks such as \texttt{CVX}~\cite{grant2009cvx} and \texttt{Convex.jl}~\cite{udell2014convex} transform high-level convex expressions into canonical conic forms, focusing on convexity verification and solver compatibility rather than decomposition quality for parallel splitting. General modeling languages such as \texttt{Pyomo}~\cite{hart2011pyomo}, \texttt{JuMP}~\cite{Lubin2023}, and \texttt{AMPL}~\cite{fourer1990modeling} provide expressive interfaces and solver interoperability, but do not automatically craft ADMM/FLiP-ready decompositions. Our work aims to fill this gap by automating bipartite/two-block reformulations from generic multiblock input models.

\subsection{Contributions}
Our main contributions are as follows.
\begin{itemize}
    \item \textbf{Multiblock abstraction.} We introduce a general Multiblock Problem (MBP) formulation that captures block-structured patterns commonly found in large-scale optimization models.
    \item \textbf{Automatic reformulation pipeline.} Building on MBP, we develop a complete pipeline that transforms an arbitrary multiblock model into a canonical two-block structure suitable for parallel solution by ADMM. The pipeline exposes latent sparsity, reorganizes block interactions, and produces an ADMM-ready decomposition under low structural assumptions and with controllable preprocessing cost. Because the transformation depends only on the problem’s sparsity pattern, it can be applied to individual instances or repeatedly across families of related problems. This pipeline addresses a key bottleneck in large-scale applications: constructing a splitting that preserves convergence guarantees without introducing excessive auxiliary variables. Rather than subdividing every edge, our bipartization strategies break odd cycles wisely, offering  explicitly control over the trade-off between improved decoupling and reformulation overhead. For sparse graphs, such as those arising in power and network systems, this trade-off is typically very favorable.

    \item \textbf{Bipartization strategies (BFS, MILP, and GNN).} A central component of our pipeline is bipartization via edge subdivision. We develop (i) fast graph-traversal strategies, such as breadth-first search (BFS), to obtain a valid bipartition through edge subdivision, and (ii) a novel mixed-integer linear programming (MILP) formulation that enforces bipartiteness while optimizing structural properties of the resulting reformulation that impact ADMM behavior. To address the scalability limitations of MILP on large graphs, we further introduce (iii) a Graph Neural Network (GNN) surrogate that learns to emulate optimal MILP partition decisions. This learning-based approach leverages the ability of GNNs to capture graph-structured dependencies and generalize across instances, enabling scalable partitioning for large-scale or recurring problem families. 
\item \textbf{Empirical validation and software.} We validate the effectiveness of the proposed approach through extensive numerical experiments across a diverse set of structured and synthetic models, demonstrating that the resulting bipartizations consistently improve parallel ADMM performance. All components are implemented and released in the open-source Julia package \texttt{PDMO.jl} available at 
\begin{center}
\url{https://github.com/alibaba-damo-academy/PDMO.jl}.
\end{center}
\end{itemize}

\subsection{Notations}
We denote the set of positive integers up to $p$ by $[p]$, real numbers by $\R$, and the $n$-dimensional real Euclidean space by $\R^n$. For vectors $x$ and $y$ of the same shape, the inner product of $x$ and $y$ is denoted by $\langle x, y \rangle$, and the Euclidean norm of $x$ is denoted by $\|x\|$. For $A \in \R^{m\times n}$, $\|A\|$ and $\sigma_{\min}(A)$ denote the largest and smallest singular value of $A$, respectively. For a set $X$, we use $\delta_X$ to denote the $0/\infty$-indicator function of $X$, i.e., $\delta_X(x) = 0$ if $x\in X$ and $+\infty$ otherwise. We use $\bft{0}$ for the zero array and $\mathrm{Id}$ for the identity mapping; their dimensions should be clear from the context.

When a graph $G$ with vertex set $V$ and edge set $E$ serves as the structural backbone of an optimization problem with variables $\bft{x}_i$'s, we use $\bft{x}_i$ and $i$ interchangeably to index vertices in $V$ or denote endpoints of edges in $E$. We use $I \sqcup J$ to denote the union of two \textit{disjoint} sets $I$ and $J$. Other notations will be explained upon their first appearance.

\section{Automatic Reformulation Pipeline}\label{sec: pipeline}

\subsection{Multiblock Problem Formulation}\label{sec: mbp}
Practitioners working with large-scale models often know informally which variables form natural blocks and how these blocks interact through localized constraints. Yet turning this intuition into a formulation suitable for scalable algorithms such as parallel ADMM is difficult. As demonstrated in Example~\ref{example: demo}, implementing ADMM typically requires reorganizing blocks manually, introducing auxiliary variables, and decoupling constraints to expose parallel subproblems. This process remains ad hoc, expert-driven, and poorly supported by existing modeling tools. This gap motivates the need for an explicit structural representation that captures sparsity and local interactions while providing a clean foundation for automatic reformulation. 

Example~\ref{example: demo} illustrates, in miniature, the type of structural patterns that arise in larger optimization models: the three branch currents naturally form three variable blocks, and each KCL equation couples a subset of these blocks. To formalize the class of models targeted by our reformulation pipeline, we introduce the following Multiblock Problem (MBP) formulation:
\begin{subequations}\label{eq: multiblock_nlp}
\begin{align}
\min_{(\bft{x}_1, ..., \bft{x}_n)} \quad  & \sum_{i=1}^n f_i(\bft{x}_i) + g_i(\bft{x}_i) \\
\mathrm{s.t.} \quad & \sum_{i \in S_k} A^k_i(\bft{x}_i) = b^k, ~ S_k \subseteq [n],~ k \in [m]. \label{eq: multiblock_constriants}
\end{align}
\end{subequations}
We explain different components shortly. In view of Example~\ref{example: demo}, we can identify $\bft{x}_i = I_i$, set $f_i(\bft{x}_i) = R_i \bft{x}_i^2$, and take $g_i(\bft{x}_i)=0$ for $i\in [3]$. The problem has $m=3$ linear constraints. For instance, the first KCL equation corresponds to $k=1$ with $S_1 = \{1,3\}$, where $A^1_1$ and $A^1_3$ denote $\mathrm{Id}$ and $-\mathrm{Id}$, respectively, and $b^1 = J_1$. The remaining two constraints are defined analogously.

This formulation reflects the structural patterns commonly found in large-scale optimization problems where the decision variables naturally decompose into $n$ \emph{blocks}, denoted $(\bft{x}_1,\dots,\bft{x}_n)$. Each block may represent a scalar, a vector, a matrix, or any application-specific array, and the overall model often becomes more interpretable and manageable when expressed in this blockwise form. By working with block variables explicitly, the MBP captures the latent separability present in many practical problems while remaining general enough to model a wide range of applications.

Each objective term $f_i$ is a smooth function with a Lipschitz continuous gradient, and the term $g_i$ is allowed to be nonsmooth but is assumed to be \emph{proximal-friendly}, meaning that its proximal operator
\begin{align}\label{eq: prox_def}
    \text{prox}_{\gamma g_i}(z) := \argmin_{x} g_i(x) + \frac{1}{2\gamma} \|x - z\|^2
\end{align}
can be computed efficiently for $\gamma >0$. This decomposition into smooth and proximal-friendly parts reflects common modeling practice. The smooth components $f_i$ typically encode losses, cost functions, or other differentiable terms, while the nonsmooth components $g_i$ often represent regularization terms or indicators that enforce domain constraints. Such composite objectives arise frequently in large-scale applications and are well suited for first-order methods. In FLiP-ADMM subproblems, the smooth parts can be handled through gradient-based updates, and the nonsmooth parts through proximal mappings. When a partial minimization of the augmented Lagrangian with respect to $\bft{x}$ or $\bft{z}$ is computationally easy, performing the exact update is preferable; otherwise, inexpensive first-order steps provide a scalable alternative. 

The constraints in \eqref{eq: multiblock_constriants} further highlight the multiblock structure. Each of the $m$ \emph{block constraints} couples only a subset of variable blocks indexed by $S_k\subseteq[n]$, through linear mappings $A^k_i$ for $i\in S_k$. The sets $S_k$ may differ across constraints, allowing the model to encode heterogeneous and localized interactions such as conservation laws, network flows, or consistency conditions. Importantly, these couplings are typically sparse: each constraint involves only a few blocks, while the total number of blocks could be larger. Throughout this paper, we assume $|S_k|\geq 2$: if some affine constraints involve only one $\bft{x}_i$, one can try to incorporate these constraints into the domain of $g_i$, or equip another $\bft{x}_j$ with the zero mapping to fake the constraints into a 2-block form. 

Expressing a problem explicitly in terms of block variables, composite objectives, and sparsely coupled block constraints, formulation~\eqref{eq: multiblock_nlp} provides an intermediate representation between high-level modeling languages and algorithm-specific reformulations. It captures the inherent sparsity and localized interactions present in large-scale applications while offering a unified starting point from which our reformulation pipeline can automatically detect, manipulate, and refine block interactions to produce ADMM-ready decompositions. We present a generalization of Example \ref{example: demo} next and further illustrate how MBP representations arise in specific applications in Section~\ref{sec: experiment}. 

\begin{example}\label{example: network_flow}
    Example \ref{example: demo} can be generalized to the classic Network Flow Problem. Consider a connected directed graph $G(\mathcal{N}, \mathcal{A})$ representing a flow network. Each node $i \in \mathcal{N}$ is associated with a parameter $b_i$, where $b_i > 0$ denotes supply and $b_i < 0$ denotes demand. Each arc $(i,j) \in \mathcal{A}$ carries a flow with transportation cost $c_{ij}$. The classic network flow problem seeks to determine the flow on each arc so as to satisfy the supply–demand balance at every node while minimizing the total transportation cost over the network. The problem can be formulated as
\begin{subequations}\label{eq: network_flow}
\begin{alignat}{2}
    \min_{\{f_{ij}\}_{(i,j)\in \mathcal{A}}} \quad & \sum_{(i,j)\in \mathcal{A}}c_{ij}( f_{ij}) \\
    \mathrm{s.t.}\quad &
    \sum_{j\in O(i)} f_{ij} - \sum_{j\in I(i)} f_{ji}= b_i,\quad && \forall i\in \mathcal{N},\\
    & 0 \leq f_{ij} \leq u_{ij}, \quad && \forall (i,j) \in \mathcal{A}, 
\end{alignat}
\end{subequations}
where $I(i)$ and $O(i)$ denote sets of nodes incoming to and outgoing from $i$, respectively, and 
$u_{ij} \geq 0$ represents the capacity of $(i,j)$. Important special cases of \eqref{eq: network_flow} include the Shortest Path Problem, the Maximum Flow Problem, the Transportation Problem, and the Assignment Problem \cite{bertsimas1997introduction}. 

The network flow problem \eqref{eq: network_flow} naturally admits a multiblock structure. Each arc flow $f_{ij}$ is treated as a block variable with a smooth transportation cost $c_{ij}(\cdot)$ and a proximal-friendly constraint $\delta_{[0,u_{ij}]}(\cdot)$. Each nodal balance equation induces an affine coupling among all flow variables incident to the corresponding node. 
In fact, this structure is more clearly revealed on the \textit{dual hypergraph} of the network, where each node $v(i,j)$ corresponds to an arc $(i,j)\in \mathcal{A}$, and each hyper-edge $h(i)$ is the set of nodes corresponding to edges connected to $i\in \mathcal{N}$, i.e., $h(i) = \{v(i,j)\}_{j\in O(i)}\sqcup \{v(j,i)\}_{j\in I(i)}$. This resulting dual hypergraph inherently aligns with our MBP formulations: nodes represent block variables and hyper-edges correspond to nodal balance constraints.
\end{example}
\begin{remark}
While the MBP formulation is flexible enough to encode nonconvex and even mixed-integer structures, we do not consider these extensions in this work. Our attention is restricted to convex problems, smooth and proximal-friendly block objectives, and constraints described by linear mappings, the regime most compatible with first-order ADMM.
\end{remark}

\subsection{Graph-based Transformation}\label{sec: graph_transformuation}

Many optimization problems arise in multiblock form where individual constraints involve more than
two variable blocks. Although such formulations are flexible, a direct application of ADMM to a
general multiblock structure may exhibit slow convergence or may not converge at all. To leverage
the robust theoretical guarantees of the two-block ADMM framework, our pipeline automatically
transforms a multiblock problem into a bipartite form in which the final two blocks each consist of
one or more original sub-blocks.

To this end, we first construct a graph representation of the problem, which
serves as the foundation for this transformation. We associate a graph to the MBP \eqref{eq: multiblock_nlp} that captures the coupling pattern among block variables and block constraints. The construction is summarized in
Algorithm~\ref{alg:graph-construction}. To avoid notational mismatch, we use $\bft{x}_i$ to label vertices as well. 
Algorithm~\ref{alg:graph-construction} can be viewed as first forming a hypergraph, where nodes
represent block variables and hyper-edges represent block constraints, and then performing a \textit{star
expansion}\footnote{The operation that replaces each hyper-edge by a new auxiliary node connected to all vertices in that hyper-edge.} on every hyper-edge involving more than two nodes. Although a hypergraph is a more compact
representation of the MBP, we deliberately replace each hyper-edge by a constraint node connected to its incident variable nodes in order to facilitate subsequent transformations and computations. 

\begin{algorithm}[htbp]
\caption{Graph Construction for the Multiblock Problem}
\label{alg:graph-construction}
\begin{algorithmic}[1]
\State \textbf{Input:} A multiblock problem \eqref{eq: multiblock_nlp}
\State initialize an empty graph $G = (V,E)$;
\For{each block variable $\bft{x}_i$}
    \State add a node (with label) $\bft{x}_i$ to $V$;
\EndFor
\For{each block constraint $k\in [m]$ with involved block variables set $S_k$}
    \If{$|S_k| = 2$}
        \State suppose $S_k = \{i,j\}$ and add an undirected edge $(\bft{x}_i, \bft{x}_j)$ to $E$;
    \Else
        \State introduce a new constraint node $C_k$ and add it to $V$;
        \For{each $i \in S_k$}
            \State add an edge $(C_k, \bft{x}_i)$ to $E$;
        \EndFor
    \EndIf
\EndFor
\State\textbf{return} $G$.
\end{algorithmic}
\end{algorithm}


\begin{example}\label{eq: example}
Consider a slight generalization of Example \ref{example: demo}, where the first block constraint ($C_1$) couples 3 block variables, together with its transformed graph, as shown in Fig. \ref{fig: graph}. Blue nodes represent original block variables, and orange node $C_1$ represents the first block constraint that couples all three block variables. 
Edges $(\bft{x}_1, \bft{x}_2)$ and $(\bft{x}_2, \bft{x}_3)$ represent the second and third constraints. 
\begin{figure}[ht]
\begin{minipage}{0.5\textwidth}
\begin{align*}
\min_{(\bft{x}_1,\bft{x}_2,\bft{x}_3)} \quad 
    & f_1(\bft{x}_1) + f_2(\bft{x}_2) + f_3(\bft{x}_3) \\
\mathrm{s.t.}\quad 
    & 
    \begin{bmatrix}
        A_1 & A_2 & A_3 \\
        B_1 & B_2 & \\
            & C_2 & C_3
    \end{bmatrix}
    \begin{bmatrix}
        \bft{x}_1 \\ \bft{x}_2 \\ \bft{x}_3
    \end{bmatrix}
    =
    \begin{bmatrix}
    a \\ b \\ c
    \end{bmatrix}
    \quad 
    \begin{matrix}
    \leftarrow C_1 \\ \text{} \\ \text{}
    \end{matrix}
\end{align*}   
\end{minipage}
\hfill
\begin{minipage}{0.4\textwidth}
\centering
\begin{tikzpicture}[
    var/.style={circle, draw, fill=cyan!30, minimum size=18pt},
    constr/.style={circle, draw, fill=orange!70, minimum size=18pt}
]
\node[var] (v1) at (0,0.0) {$\bft{x}_1$};
\node[var] (v2) at (1.5,0) {$\bft{x}_2$};
\node[var] (v3) at (0,-1.5) {$\bft{x}_3$};
\node[constr] (u1) at (1.5,-1.5) {$C_1$};
\draw (v1) -- (v2);
\draw (v1) -- (u1);
\draw (v2) -- (v3);
\draw (v2) -- (u1);
\draw (v3) -- (u1);
\end{tikzpicture}
\end{minipage}
\caption{A 3-block example and its graph representation.}\label{fig: graph}
\end{figure}
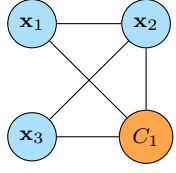
The graph representation can be understood algebraically as follows. We introduce $\bft{y} = (\bft{y}_1, \bft{y}_2, \bft{y_3})$ as the nodal variable for the newly introduced vertex $C_1$, and replace the constraint $A_1\bft{x}_1 + A_2\bft{x}_2+A_3\bft{x}_3 = a$ by
\begin{subequations}\label{eq: substitution}
    \begin{align}
    & A_i\bft{x_i} - \bft{y}_i = \bft{0}, ~\forall i\in [3],\\
    & \bft{y} = (\bft{y}_1, \bft{y}_2, \bft{y}_3)  \in Y:=\{(\bft{y}_1, \bft{y}_2, \bft{y}_3)~|~ \bft{y}_1 + \bft{y}_2 + \bft{y}_3 = a\}.
\end{align}
\end{subequations}
The variable $\bft{y}$ is assigned a proximal-friendly objective component $\delta_Y$, whose proximal operation is the projection onto $Y$ and can be computed in a closed form. Other ways of rewriting this constraint are possible, as long as each newly introduced constraint couples only two block variables (might include an auxiliary variable block). 
\end{example}

As shown in Example \ref{eq: example}, upon termination of this transformation stage, the original MBP \eqref{eq: multiblock_nlp} can be equivalently written as the following graph-based reformulation: 
\begin{subequations}\label{eq: graph_problem_formulation}
\begin{align}
    \min_{\{\bft{x}_i\}_{i\in V}} \quad & \sum_{i\in V} f_i(\bft{x}_i) + g_i(\bft{x}_i) \\
    \mathrm{s.t.} \quad & Q^e_i\bft{x}_i + Q^e_j\bft{x}_j = b^e,~\forall e = (i,j)\in E, \label{eq: graph_problem_formulation_constr}
\end{align}
\end{subequations}
where $G(V,E)$ is the transformed graph. We write $\bft{x}_i$ for the variable associated with each node $i\in V$, and we use each edge $e=(i,j)\in E$ to denote the coupling constraint generated from the block constraints involving blocks $i$ and $j$ with corresponding mappings $(Q^e_i,Q^e_j)$ (either some $A^k_i$ from \eqref{eq: multiblock_constriants} or $\pm\mathrm{Id}$) together with a right-hand side array $b^e$. In this representation, the entire
algebraic structure of the problem is encoded by the graph: vertices collect local
variables and local objective terms, while edges encode all pairwise interactions arising from the
original constraint sets. This graph-based reformulation is therefore an intermediate representation of the MBP, one in which the sparsity pattern and coupling structure are made explicit. However, formulation \eqref{eq: graph_problem_formulation} is not necessarily ready for ADMM: if $G$ is not bipartite, then one cannot divide block variables into two groups and perform parallel block updates alternately.  We formalize this observation in the next proposition. 

\begin{proposition}\label{prop: coupled_subproblem}
Consider the graph-based formulation \eqref{eq: graph_problem_formulation}
defined on a graph $G(V,E)$. 
There exists a partition of variable blocks into two disjoint groups $V = \texttt{L}\sqcup \texttt{R}$
such that every coupling block constraint \eqref{eq: graph_problem_formulation_constr} involves exactly one variable from each
group if and only if the graph
$G$ is bipartite.
\end{proposition}

\begin{proof}
A bipartition assigns each node of $G$ to one of two sets \texttt{L} and  \texttt{R}
so that every edge connects nodes in different sets. This is precisely a
2-coloring of $G$, which exists if and only if $G$ is bipartite.
\end{proof}
When $G$ is not bipartite, some odd cycle contains nodes residing on the same side when vertices are partitioned into two groups, forcing updates of their corresponding variables to remain coupled in ADMM. In order to obtain a formulation amenable to parallel two-block ADMM, we must further transform the graph so
that it becomes bipartite. 

\subsection{Bipartization via Edge Subdivision}\label{sec: bipartization}
The coupling graph constructed in the previous subsection is, in general, not bipartite. In particular,
the presence of odd cycles prevents a clean two-way separation of nodes, which is required for
producing a canonical two-block formulation suitable for ADMM. To obtain such a structure, we apply a classic graph-based transformation known as \emph{edge subdivision}, which replaces an edge by a path of length 2. In particular, subdividing at least one edge in every odd cycle converts the cycle into an even cycle. Repeating this operation sufficiently many times yields a bipartite graph.

\addtocounter{example}{-1}
\begin{example}[Continued: Bipartization]
The graph in Fig. \ref{fig: graph} contains odd cycles (e.g., $\bft{x}_1 - \bft{x_2} -C_1-\bft{x}_1$) and hence is not bipartite. According to Proposition \ref{prop: coupled_subproblem}, the algebraic model encoded by this graph is not ADMM-ready, no matter how block variables are grouped together. In order to run parallel ADMM, we can replace edge $(\bft{x}_2, C_1)$ by two edges $(\bft{x}_2, \bft{w})$ and $(\bft{w}, C_1)$ with a new node $\bft{w}$. The resulting graph is clearly bipartite, as shown in Fig. \ref{fig: bipartite}. 

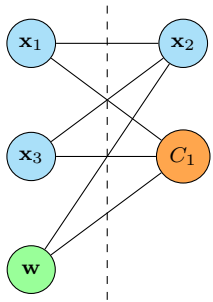
\begin{figure}[ht]
\centering
\begin{tikzpicture}[
    var/.style={circle, draw, fill=cyan!30, minimum size=18pt},
    constr/.style={circle, draw, fill=orange!70, minimum size=18pt}, 
    aux/.style={circle, draw, fill=green!40, minimum size=18pt}
]
\node[var] (v1) at (0,0.0) {$\bft{x}_1$};
\node[var] (v2) at (2.0,0) {$\bft{x}_2$};
\node[var] (v3) at (0,-1.5) {$\bft{x}_3$};
\node[constr] (u1) at (2.0,-1.5) {$C_1$};
\node[aux] (w) at (0, -3.0) {$\bft{w}$};
\draw (v1) -- (v2);
\draw (v1) -- (u1);
\draw (v2) -- (v3);
\draw (v3) -- (u1);
\draw (w) -- (v2); 
\draw (w) -- (u1);
\draw[dashed] (1,0.5) -- (1,-3.5);
\end{tikzpicture}
\caption{Bipartite graph obtained by subdividing $(\bft{x}_2, C_1)$ in Fig. \ref{fig: graph}.}\label{fig: bipartite}
\end{figure}
From an algebraic perspective, the original problem in Example \ref{eq: example} can be further reformulated as 
\begin{subequations}\label{eq: two_block_reformulation}
\begin{align}
    \min_{\stackrel{(\bft{x}_1, \bft{x}_3, \bft{w})}{(\bft{y}_1, \bft{y}_2,\bft{y}_3, \bft{x}_2)}} \quad & \Big(f_1(\bft{x}_1) + f_3(\bft{x}_3) \Big) + \Big( \delta_Y(\bft{y}) + f_2(\bft{x_2})\Big) \\
    \mathrm{s.t.}\quad & 
    \underbrace{
    \begin{bmatrix}
        A_1 &     & \\
        B_1 &     & \\
            & A_3 &  \\
            & C_3 &  \\
            &     & -\mathrm{Id} \\
            &     & -\mathrm{Id} 
    \end{bmatrix}
    }_{\bft{A}}
    \begin{bmatrix}
        \bft{x}_1 \\ \bft{x}_3 \\ \bft{w}
    \end{bmatrix}
    + 
    \underbrace{
    \begin{bmatrix}
        -\mathrm{Id} & & & \\
                     & & & B_2 \\
            &&-\mathrm{Id}& \\ 
            &&& C_2 \\
            &&& A_2 \\
            &\mathrm{Id}&&
    \end{bmatrix}
    }_{\bft{B}}
    \begin{bmatrix}
        \bft{y}_1 \\ \bft{y}_2 \\ \bft{y}_3 \\ \bft{x}_2
    \end{bmatrix}
    = 
    \underbrace{
    \begin{bmatrix}
        \bft{0} \\ b \\ \bft{0} \\ c \\ \bft{0} \\ \bft{0}
    \end{bmatrix}
    }_{\bft{b}}.\label{eq: two_block_reformulation_constraint}
\end{align}
\end{subequations}
Here we reuse $\bft{w}$ for the variable of the newly introduced node: we replace $A_2\bft{x}_2 - \bft{y}_2 = 0$ by 
\begin{align}\label{eq: substitute_2block_edge}
    A_2\bft{x}_2 - \bft{w} = 0,~ \bft{y}_2 - \bft{w} = 0.
\end{align}
Clearly, the reformulated problem \eqref{eq: two_block_reformulation} is now in the canonical
two-block form required by ADMM. All variables on the left partition $(\bft{x}_1,\bft{x}_3,\bft{w})$ appear only through the linear operator $\bft{A}$, which is
block-diagonal. Likewise, the variables on the right partition $(\bft{y}_1,\bft{y}_2,\bft{y}_3,\bft{x}_2)$ interact only through the operator $\bft{B}$, which can
be permuted into a block-diagonal structure as well.
\end{example}

Denote by $\widehat{G} = (\widehat{V}, \widehat{E})$ the bipartite graph obtained after the
subdivision procedure, with partitions $\widehat{V} = \texttt{L} \sqcup \texttt{R}$ and edges
$\widehat{E} \subseteq \texttt{L} \times \texttt{R}$. This graph encodes the optimization problem
\begin{subequations}\label{eq: bipartite_formulation}
\begin{align}
    \min_{\stackrel{\{\bft{x}_i\}_{i\in \texttt{L}}}{\{\bft{z}_j\}_{j\in \texttt{R}}}} \quad 
        & \Bigg(\sum_{i\in \texttt{L}} f_i(\bft{x}_i) + g_i(\bft{x}_i)\Bigg)
        +
        \Bigg(\sum_{j\in \texttt{R}} f_j(\bft{z}_j) + g_j(\bft{z}_j)\Bigg), \\
    \mathrm{s.t.}\quad 
        & \bft{A}_i^e\,\bft{x}_i + \bft{B}_j^e\,\bft{z}_j = \bft{b}^e, ~\forall\, e=(i,j)\in \widehat{E}. 
        \label{eq: bipartite_formulation_constraint}
\end{align}
\end{subequations}
This formulation is obtained by applying the graph construction and subdivision steps to the
original MBP~\eqref{eq: multiblock_nlp}, resulting in the intermediate representation
\eqref{eq: graph_problem_formulation} and ultimately in the bipartite structure
\eqref{eq: bipartite_formulation}. Once the graph is bipartite, all variables on the $\texttt{L}$(eft)-side
$\{\bft{x}_i\}_{i\in \texttt{L}}$ are coupled only with variables on the $\texttt{R}$(ight)-side $\{\bft{z}_j\}_{j\in \texttt{R}}$,
and never within their own side. Consequently, both blocks admit fully parallel updates in ADMM,
with each subproblem depending only on local terms and its incident edges. The cross-partition
constraints~\eqref{eq: bipartite_formulation_constraint} are handled entirely through the augmented
Lagrangian, while the primal updates remain separable across all nodes within each partition. Thus, the bipartization step converts the MBP into a canonical two-block form in which ADMM naturally decomposes into parallelizable subproblems and lightweight linear-algebra operations, ensuring scalability and structural compatibility with modern first-order methods.

It is important to emphasize that the bipartite representation obtained through edge subdivision is
not unique. Different choices of edges to subdivide correspond to different ways of breaking odd
cycles, and these choices depend on the order in which the graph is explored. Distinct traversal
strategies, such as breadth-first search (BFS) or depth-first search (DFS), may therefore yield
partitions with different numbers of subdivision nodes and different adjacency patterns. In the
next three subsections, we present a BFS-based bipartization algorithm, a MILP formulation
that seeks an optimal bipartition under criteria tailored to ADMM convergence performance, and a learning-based GNN surrogate that aims to emulate optimal MILP decisions.

\subsubsection{BFS-based Bipartization Algorithm}
We present a BFS-based bipartization procedure in Algorithm~\ref{alg:BFS-bipartization}, which traverses the graph in the breadth-first order and splits edges whenever a parity conflict is detected. 

\begin{algorithm}[ht]
\caption{BFS-based Bipartization Algorithm}
\label{alg:BFS-bipartization}
\begin{algorithmic}[1]
\State \textbf{Input:} undirected graph $G=(V,E)$;
\State initialize node partition $c:V\to\{-1, 0,1\}$ and edge decisions $\sigma:E\to\{ 0,1\}\times\{0,1\}$ with $c(v) =-1$ and $\sigma(e)=(0,0)$ for all $v\in V$ and all $e\in E$;
\State construct neighbor lists $N(v)$ for all $v\in V$, set $s\gets 0$; 
\For{each $v\in V$}
    \If{$c(v)= -1$} \Comment new connected component
        \State $c(v) \gets s$;
        \State initialize queue $Q \gets \{v\}$;
        \While{$Q$ is not empty}
            \State remove front node $u$ from $Q$;
            \State $p \gets c(u)$;
            \For{each edge $e = \{u,w\} \in E$ incident to $u$}
                \If{$c(w)=-1$}
                    \State $c(w) \gets 1-p$; \Comment assign opposite color to a new neighbor
                    \State insert $w$ into $Q$;
                \ElsIf{$c(w) = p$ and $\sigma(e) = (0,0)$} \Comment new color conflict detected
                    \State $\sigma(e) \gets (1, 1-p)$ ;\Comment split, assign to side $1-p$
                \EndIf
            \EndFor
        \EndWhile
        \State $s \gets 1-s$; \Comment alternate starting partition for next component
    \EndIf
\EndFor
\State \textbf{Return:} $(c, \sigma)$.
\end{algorithmic}
\end{algorithm}
The mappings $c$ and $\sigma$ encode sufficient information to construct the bipartite graph $\hat{G}(I\sqcup J, \widehat{E})$ using the following procedures in Algorithm \ref{alg:bipartite construction}. In principle, $c(v)$ assigns each vertex $v\in V$ to either the left partition ($c(v)=0$) or the right partition ($c(v)=1$). For each edge $e\in E$, the first component of $\sigma(e)$ indicates if the edge should be subdivided; if yes ($\sigma(e)[1]=1$), then the second component indicates the assignment of the newly introduced node. As we will see next, Algorithm \ref{alg:bipartite construction} is also used to construct bipartite graphs from outputs of MILP and GNN. 

\begin{algorithm}[ht]
\caption{Bipartite Graph Construction Procedures}
\label{alg:bipartite construction}
\begin{algorithmic}[1]
\State \textbf{Input:} undirected graph $G=(V,E)$, vertex mapping $c$, and edge mapping $\sigma$;
\State initialize empty $\texttt{L}$, $\texttt{R}$, and $\widehat{E}$;
\For{each $v\in V$}
    \If{$c(v)=0$}
        \State append $v$ to $\texttt{L}$; \Comment{assign to left partition}
    \Else 
        \State append $v$ to $\texttt{R}$; \Comment{assign to right partition}
    \EndIf
\EndFor
\For{each $e\in E$}
    \If{$\sigma(e)[1] = 1$}
        \State introduce a new node $w$; \Comment{subdivide this edge with a new node}
        \If{$\sigma(e)[2] = 0$} 
            \State append $w$ to $\texttt{L}$; \Comment{assign the new node to left partition}
        \Else 
            \State append $w$ to $\texttt{R}$; \Comment{assign the new node to right partition}
        \EndIf 
    \Else 
        \State append $e$ to $\widehat{E}$;  \Comment subdivision not required
    \EndIf
\EndFor

\State \textbf{Return:} $\widehat{G}(\widehat{V} := \texttt{L} \sqcup\texttt{R}, \widehat{E})$.
\end{algorithmic}
\end{algorithm}

\subsubsection{An MILP Formulation for Optimal Bipartization}
In this subsection, we present a different approach to convert $G$ to $\widehat{G}$ based on a mixed-integer linear program (MILP). The goal is to
partition the nodes of the coupling graph into two groups and decide which edges to subdivide so as
to obtain a bipartite graph, while simultaneously controlling certain formulation-specific quantities
that influence the convergence rate of ADMM, such as the norm of $\|\bft{A}\|$ or $\|\bft{B}\|$ in the final
two-block formulation. We assume $\bft{A}$ and $\bft{B}$ are matrices in this section, while the underlying idea extends to general linear mappings. 

\paragraph{Variables and constraints.}
The purpose of the bipartization step is to convert the original graph $G(V,E)$ into a
bipartite graph $\widehat{G}(\widehat{V},\widehat{E})$ by allowing certain edges to be subdivided. We introduce the following binary variables.
\begin{itemize}
    \item For each node $i \in V$, introduce two binary variables $(x_i^\texttt{L}, x_i^\texttt{R}) \in \{0,1\}^2$. Node $i$
    is assigned to \texttt{L} (respectively \texttt{R}) if and only if $x_i^\texttt{L} = 1$ (respectively $x_i^\texttt{R} = 1$). Each node
    must belong to exactly one side, so we impose
    \begin{align}\label{constr: node}
        x_i^\texttt{L} + x_i^\texttt{R} = 1, \quad \forall i\in V.
    \end{align}

    \item For each edge $e=(i,j)\in E$, introduce three binary variables
    $(z_e, x_e^\texttt{L}, x_e^\texttt{R}) \in \{0,1\}^3$. The variable $z_e$ indicates whether edge $e$ is subdivided:
    $z_e = 0$ means that $e$ is kept as is, while $z_e = 1$ means that $e$ is subdivided by
    introducing a new node. If $z_e = 0$, the edge is not split and we set $x_e^\texttt{L} = x_e^\texttt{R} = 0$. In
    that case the endpoints cannot lie on the same side of the bipartition, so they must satisfy
    $x_i^\texttt{L} + x_j^\texttt{L} = 1$ (equivalently, $x_i^\texttt{R} + x_j^\texttt{R} = 1$).

    When $z_e = 1$, we introduce a new node associated with edge $e$, add edges $(i,e)$ and $(j,e)$, and delete the edge $(i,j)$. Here we continue to use $e$ as the new node to avoid additional notations.  The new node must be assigned to either \texttt{L} or \texttt{R} via $x_e^\texttt{L}$ and $x_e^\texttt{R}$, and it cannot be placed on the same side as $i$ or $j$.
    These logical requirements can be encoded by the following linear constraints: for all
    $e=(i,j)\in E$,
    \begin{subequations}\label{constr: edge}
        \begin{align}
            & x_e^\texttt{L} + x_e^\texttt{R} = z_e, \\
            &  1 - z_e \leq x_i^\texttt{L} + x_j^\texttt{L} \leq 1 + z_e,\\
            & z_e \leq x_i^\texttt{L} + x_e^\texttt{L} \leq 2 - z_e, \\
            & z_e \leq x_j^\texttt{L} + x_e^\texttt{L} \leq 2 - z_e.
        \end{align}
    \end{subequations}
    When $z_e = 0$, the second inequality enforces $x_i^\texttt{L} + x_j^\texttt{L} = 1$, that is, $i$ and $j$ lie on
    opposite sides and no subdivision node is used. When $z_e = 1$, the last two inequalities ensure
    that the subdivision node and each endpoint cannot both be assigned to L, and by symmetry, the
    same holds for R, thereby avoiding odd cycles through the split edge.
\end{itemize}
As a consequence, any feasible solution to \eqref{constr: node} and \eqref{constr: edge} naturally provides valid mappings $c(i) = x_i^\texttt{R}$ for $i\in V$ and $\sigma(e) = (z_e, x^\texttt{R}_e)$ for $e\in E$, which can be used to construct the bipartite graph $\widehat{G}$ using Algorithm \ref{alg:bipartite construction}. 

\paragraph{Objective function.}
The convergence behavior of ADMM depends not only on the abstract bipartite structure, but also on
quantitative properties of the resulting linear operators, such as the norms of the matrices $\bft{A}$ and $\bft{B}$
in the coupling constraint $\bft{A}\bft{x} + \bft{B}\bft{z} = \bft{b}$. These norms frequently appear in iteration-complexity bounds and influence how quickly the dual residuals of ADMM iterates approach zero. Empirically speaking, it is known that ADMM converges faster with fewer coupling constraints. It is therefore
natural to incorporate such quantities into the MILP objective, for example by minimizing
$\|\bft{A}\|$ or $\|\bft{B}\|$ subject to the constraints above.

Due to the construction of the bipartite graph, the final matrices $\bft{A}$ and $\bft{B}$ can be
permuted into block-diagonal form, since no variables within the same partition are coupled. Without
loss of generality, we assume that $\bft{A}$ has already been permuted into such a form as shown in \eqref{eq: two_block_reformulation_constraint} of the illustrative example. We consider $\|\bft{A}\|$ in our derivation; this entails no restriction, as matrix norms are invariant under permutations. 

Suppose diagonal blocks of $\bft{A}$ are $\bft{A}_i$'s for $i\in \texttt{L}$, i.e., $\bft{A}_i$ is the vertical concatenation of $\bft{A}^e_i$ over edges $e\in \widehat{E}$ where $i$ is an endpoint. It holds that $\|\bft{A}\| = \max_{i\in \texttt{L}} \{\|\bft{A}_i\|\}$. The norm $\|\bft{A}\|$ can be effectively controlled using additional continuous and integer variables. Suppose the nodes assigned to \texttt{L} correspond to the $\bft{x}$-block in view of \eqref{eq: two_block_problem}, and we wish to minimize $\|\bft{A}\|$. For each node $i \in \texttt{L}$, there are two cases.
\begin{enumerate}
    \item Node $i$ is a variable node from $V$, then in view of \eqref{eq: graph_problem_formulation}, the corresponding contribution of this node to $\|\bft{A} \|$, namely $\|\bft{A}_i\|$, can be computed as
    \begin{align}\label{eq: contribution}
        \text{contribution}(i):= \sqrt{\lambda_{\max} \left(\sum_{e \in E_i}  (Q^e_i)^\top Q^e_i \right)},
    \end{align}
    where $E_i$ is the set of edges in $E$ connected to $i$. E.g., $\bft{x}_1$ and $\bft{x}_3$ in \eqref{eq: two_block_reformulation_constraint}. 
    \item Node $i$ is a newly introduced node for a subdivided edge, then the contribution of this node to $\|\bft{A}_i \|$ is exactly $\sqrt{2}$. E.g., $\bft{w}$ in \eqref{eq: substitute_2block_edge}. 
\end{enumerate}
With an auxiliary scalar variable $t^{\texttt{L}}$, we can enforce the following constraints
\begin{subequations}\label{eq: contr: contribution}
\begin{align}
    t^{\texttt{L}}& \geq \text{contribution}(i) x_i^\texttt{L}, ~\forall i\in V, \\
    t^{\texttt{L}}& \geq  \sqrt{2} x^\texttt{L}_e, ~\forall e \in E. 
\end{align}    
\end{subequations}
Consequently, the following MILP
\begin{align}\label{eq: milp}
    \min \{ t^{\texttt{L}}| \eqref{constr: node}, \eqref{constr: edge},\eqref{eq: contr: contribution},~t^{\texttt{L}}\geq0,~& (x_i^\texttt{L}, x^\texttt{R}_i)\in \{0,1\}^2,~\forall i\in V, \notag \\
    & (z_e, x_e^\texttt{L}, x_e^\texttt{R}) \in \{0,1\}^3,~\forall e\in E\}
\end{align}
provides a valid bipartization for $G$ and at the same time effectively minimizes $\|\bft{A}\|$ in the final ADMM-ready formulation \eqref{eq: two_block_problem}. 


\begin{remark}
    We remark on some practical considerations. 
\begin{itemize}
\item \textbf{Surrogate for Block Contribution.} When the block contribution \eqref{eq: contribution} is hard to compute, we can use some surrogate metrics that are easier to compute. For example, since 
\begin{align}\label{eq: surrogate}
   \mathrm{contribution}(i) \leq  \sqrt{\sum_{e\in E_i} \|Q^e_i\|^2} \leq \sqrt{\sum_{e\in E_i} \|Q^e_i\|_F^2},
\end{align}
where $\|\cdot\|_F$ denotes the Frobenius norm, either of these two quantities can be used in \eqref{eq: contr: contribution}, and \eqref{eq: milp} still minimizes a valid upper bound of $\|\bft{A}_i\|$.

\item \textbf{Balanced Partitions.} We can enforce additional (soft) constraints so that the resulting bipartite graph is balanced:
\begin{align}\label{eq: node_approx_equai}
    \sum_{i\in V} x_i^\texttt{L}  + \sum_{e\in E} x_e^\texttt{L} \approx \sum_{i\in V} x_i^\texttt{R}  + \sum_{e \in E} x_e^\texttt{R} \approx \text{number of available cores.}
\end{align} 
Similarly, suppose the algebraic substitution for edge subdivision takes the form of \eqref{eq: substitution} and \eqref{eq: substitute_2block_edge}. The total numbers of columns and non-zeros entries in \texttt{L} and \texttt{R} can be designed to be approximately equal as follows:
\begin{align}
    & \sum_{i\in V} \mathrm{dim}(i) x_i^\texttt{L}   + \sum_{e\in E} (2\mathrm{dim}(e)) x_e^\texttt{L} \notag \\
    \approx  & \sum_{i\in V} \mathrm{dim}(i) x_i^\texttt{R}  + \sum_{e \in E} (2\mathrm{dim}(e))x_e^\texttt{R},\\
    & \sum_{i\in V} \left(\sum_{e\in E_i} \mathrm{nnz}(Q^e_i)\right)x_i^\texttt{L}   + \sum_{e\in E} (2\mathrm{dim}(e)) x_e^\texttt{L} \notag \\
    \approx & \sum_{i\in V} \left(\sum_{e\in E_i} \mathrm{nnz}(Q^e_i)\right)x_i^\texttt{R}   + \sum_{e\in E} (2\mathrm{dim}(e)) x_e^\texttt{R}, 
\end{align} 
where $\mathrm{dim}(i)$, $\mathrm{dim}(e)$, and $\mathrm{nnz}(Q^e_i)$ denote the dimensions of $\bft{x}_i$, $b^e$, and the number of non-zero entries of $Q^e_i$ in \eqref{eq: graph_problem_formulation_constr}, respectively.

The size of each nodal subproblem is not modeled in MILP during bipartization, because it is largely determined by the input MBP \eqref{eq: multiblock_nlp}, where scalar variables have already been grouped into blocks based on their couplings in the objective components. This aspect is highly application-dependent and related to detection of gradient and proximal operators. As a result, we leave this issue for future investigation on specific applications.

\item \textbf{Objective Components.} The objective function offers several degrees of freedom and can be tailored to a specific ADMM variant. For example, in addition to $\|\bft{A}\|$ and $\|\bft{B}\|$, one can add terms in \eqref{eq: node_approx_equai} in the objective function to minimize the total number of subproblems. Such terms promote MILP to find balanced partitions with fewer crossing edges. One can also incorporate bounds on Lipschitz constants of the smooth terms $\sum_{i\in \texttt{L}}f_i$ and $\sum_{i\in \texttt{R}} f_i$. Suppose for each $i\in V$, $f_i$ has a $L_i$-Lipschitz gradient. Then in the final formulation \eqref{eq: bipartite_formulation}, terms $\max_{i\in V} L_ix_i^\texttt{L}$ and $\max_{i\in V} L_ix_i^\texttt{R}$ are upper bounds of Lipschitz constants of the smooth components of the left and right partitions, respectively. These quantities can therefore influence how the final partition balances smoothness and coupling complication for faster convergence. 

\item \textbf{Computation of MILP.} The proposed MILP introduces $2|V| + 3|E|$ binary variables, and can therefore become challenging to solve when the number of block variables or block constraints is large. As a result, the bipartization approach is best suited to problems with a small number of dense blocks that are only loosely coupled. When solving the MILP  becomes a computational bottleneck, one can impose a time limit or optimality gap; any feasible solution already yields a valid bipartization. As shown in Section~\ref{sec: consensus_opt}, even a loose gap of 20\% is sufficient to produce a high-quality reformulation within seconds.
\end{itemize}
\end{remark}

\subsubsection{GNN-based Bipartization}
While MILP-based bipartization provides a theoretically grounded approach, its computational cost becomes prohibitive for large-scale graphs due to the superlinear growth in solving time with respect to graph size. To address this scalability challenge, we employ a Graph Neural Network (GNN) as a surrogate model that aims to emulate the decision patterns of the MILP solver. The GNN is trained to predict the bipartization decisions (i.e., node assignments) that the MILP would produce, thereby approximating the solution quality at a fraction of the computational time. This approach aligns with a growing body of literature that successfully employs GNNs to approximate combinatorial optimization solvers~\cite{chen2023on,NEURIPS2019_d14c2267,nair2021solvingmixedintegerprograms}, leveraging their ability to capture graph-structured dependencies and generalize across instances.

In this subsection, we briefly describe our approach to modeling a Graph Neural Network on the graph and define its learning objective. Specifically, after obtaining the graph representation $G(V,E)$ of the problem via Algorithm~\ref{alg:graph-construction}, we construct a Graph Convolutional Network (GCN) based on $G$. In the following, we describe how to derive neural network features from the graph and formulate the training objective. Further implementation details, including architectural specifications and training procedures, are provided in Appendix~\ref{section: gnn-details}.

\begin{figure}[htbp]\label{fig:Gnn-Pipeline}
\centering
\begin{tikzpicture}[
    node distance=0.3cm, 
    auto,
    box/.style={rectangle, draw=black, thick, fill=red!5, 
                minimum width=1.0cm, minimum height=1cm, align=center,
                rounded corners=4pt},
    box_1/.style = {rectangle, draw=black, thick, fill=green!5, 
                minimum width=1.0cm, minimum height=1cm, align=center,
                rounded corners=4pt},
    box_2/.style = {rectangle, draw=black, thick, fill=cyan!5, 
                minimum width=1.0cm, minimum height=1cm, align=center,
                rounded corners=4pt}
]

\node[box] (start) {MultiBlock \\Problem};
\node[box_1, right=of start] (process1) {Graph \\Representation};
\node[box_1, right=of process1] (right-mid) {Node Feature\\ $H^{(0)}$};
\node[box_1, above=of right-mid] (right-top) {Graph Topology\\$G=(V,E)$};
\node[box_1, below=of right-mid] (right-bottom) {Edge Feature\\$F^e$};
\node[box_1, right=of right-mid] (process2) {GNN};
\node[box_2, right=of process2] (decision) {Node \\Assignment};
\node[box_2, right=of decision] (end) {Bipartite \\Graph};

\draw[->, thick] (start) -- (process1);
\draw[->, thick] (process1) -- (right-mid);
\draw[->, thick] (process1.north) .. controls +(up:0.7cm) and +(left:1cm) 
    .. node[left] {} (right-top.west);
\draw[->, thick] (process1.south) .. controls +(down:0.7cm) and +(left:1cm) 
    .. node[left] {} (right-bottom.west);           
\draw[->, thick] (right-top.east) .. controls +(right:0.7cm) and +(up:0.7cm) 
    .. node[] {} (process2.north);
\draw[->, thick] (right-bottom.east) .. controls +(right:0.7cm) and +(down:0.7cm) 
    .. node[] {} (process2.south);
\draw[->, thick] (right-mid) -- (process2);
\draw[->, thick] (process2) -- (decision);
\draw[->, thick] (decision) -- node[above] {} (end);
\end{tikzpicture}
\caption{Illustration of bipartization approach based on GNN.}
\end{figure}
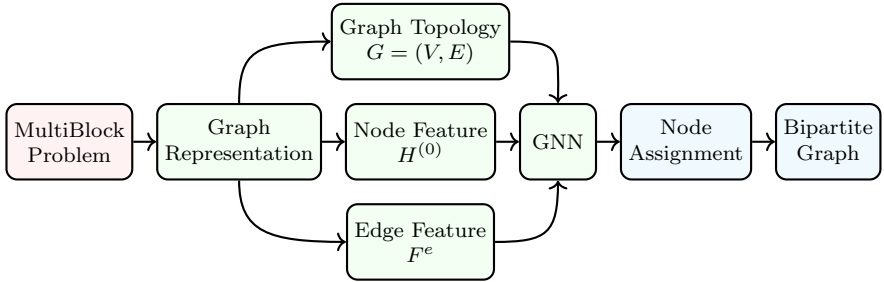

\paragraph{Node Feature Extraction.} We define the node feature matrix as $H^{(0)} \in \mathbb{R}^{|V| \times d_v}$, where $d_v$ is the input node feature dimension. To learn the behavior of the MILP solver accurately for a given class of MILP problems defined by a specific objective function, we extract the node-related information utilized by the MILP itself. For instance, for a class of MILP problems whose objective function involves certain node-specific parameters (e.g., the Lipschitz constants of the smooth terms $\sum_{i \in \texttt{L}} f_i$ and $\sum_{i \in \texttt{R}} f_i$ mentioned in the last section), we extract those parameters or related coefficients as node features. For MILP problems that do not utilize any node-specific information, we use an all-ones matrix as the node features. This implies that, during the GNN message-passing procedure, the node features primarily serve to receive edge features and aggregate information from neighboring nodes.

\paragraph{Edge Feature Extraction.} For a given edge with its associated constraint matrices $Q_i$ and $Q_j$, we extract a variety of features related to both matrices to form the edge feature vector. These include numerical statistical features of the matrices (e.g., dimension, mean, variance), structural features (e.g., rank, trace), and geometric features (e.g., the 2-norm and the infinity norm). Furthermore, we introduce a binary $(0/1)$ indicator feature based on the type of the edge. We denote the edge feature matrix as $F^e \in \mathbb{R}^{|E| \times d_e}$, where $d_e$ represents the edge feature dimension.

\paragraph{Learning Objective.} Given the node features $H^{(0)}$, edge features $F^e$, and the topology of the graph $G = (V, E)$, the goal of the GNN is to learn the following mapping:
$f \colon (H^{(0)}, F^e, G) \to c \in \{0,1\}^{|V|}$,
where $c \in \{0,1\}^{|V|}$ denotes the node assignment for the bipartite graph $\hat{G}$. For instance, $c(i) = 0$ corresponds to assigning node $i$ to the set \texttt{L}, and $c(i) = 1$ corresponds to assigning node $i$ to the set \texttt{R}.

During the training phase of the GNN, we use the MILP to annotate the training data. In the inference phase, after obtaining the node assignment vector $c$ for a specific problem, we can derive the edge mapping $\sigma$ according to the following rule. Given an edge $e = (i, j) \in E$, if $c(i) \neq c(j)$, then $\sigma(e) = (0, 0)$, which indicates that node $i$ and node $j$ belong to different partitions of the bipartite graph. This edge does not need to be split, and thus no additional variables are required. Otherwise, if $c(i) = c(j)$, we set $\sigma(e) = (1, c(i))$, which implies that the edge must be split by introducing a new variable, and the newly created node will be assigned to the same partition as node $i$. After that, Algorithm~\ref{alg:bipartite construction} is applied to generate a bipartite graph according to $c$ and $\sigma$.

\section{Experiments}\label{sec: experiment}
We present our numerical experiments in this section. The objective is to demonstrate that, compared to naive and standard reformulations, the proposed bipartization pipeline produces high–quality reformulations that accelerate parallel ADMM. All bipartization methods and ADMM variants are implemented in \texttt{PDMO.jl}. Additional codes and scripts used to generate results reported in this paper are available in the \texttt{test/reformulation} branch:
\begin{center}
\small
\url{https://github.com/alibaba-damo-academy/PDMO.jl/tree/test/reformulation}.
\end{center}
We provide implementation details and structure of experiments next. 

\paragraph{Bipartization Methods.}
We evaluate reformulations generated by BFS-based, MILP-based, and GNN-based bipartizations, all of which are implemented in the open-source package \texttt{PDMO.jl}.
\begin{itemize}
    \item \textbf{BFS.} We use Algorithm~\ref{alg:BFS-bipartization} for the BFS-based bipartization. A DFS-based bipartization is also available in \texttt{PDMO.jl} but omitted from comparison. 
    \item \textbf{MILP.}  For MILP-based bipartization, we solve the following formulation:
\begin{subequations}\label{eq: milp2}
\begin{align}
    \min \quad & t^\texttt{L} + t^\texttt{R} 
    + \left(\sum_{i\in V} x^\texttt{L}_i + \sum_{e\in E} x^\texttt{L}_e\right)
    + \left(\sum_{i\in V} x^\texttt{R}_i + \sum_{e\in E} x^\texttt{R}_e\right) \\
    \mathrm{s.t.}\quad 
    & \eqref{constr: node},\ \eqref{constr: edge},\ \eqref{eq: contr: contribution}, \\
    & t^{\texttt{R}} \ge \text{contribution}(i)\, x^{\texttt{R}}_i,~ \forall i\in V, \\
    & t^{\texttt{R}} \ge \sqrt{2} x^{\texttt{R}}_e,~ \forall e\in E, \\
    & (x^{\texttt{L}}_i, x^{\texttt{R}}_i)\in\{0,1\}^2,~ \forall i\in V, ~(z_e, x^{\texttt{L}}_e, x^{\texttt{R}}_e)\in\{0,1\}^3,~ \forall e\in E.
\end{align}
\end{subequations}
As suggested in \eqref{eq: surrogate}, we use $\|\cdot\|_F$ as estimates for $\text{contribution}(i)$ to avoid expensive computation of matrix norms.
The objective jointly penalizes (upper bounds of) norms $\|\bft{A}\|$ and $\|\bft{B}\|$ and the total number of nodes introduced during bipartization, encouraging compact and well-balanced decompositions. This MILP is solved by \texttt{HiGHS} \cite{huangfu2018parallelizing} interfaced through \texttt{JuMP} \cite{Lubin2023}. To encourage faster primal solutions, we set $\texttt{time\_limit} = 60$, $\texttt{mip\_heuristic\_effort}=0.2$, and $\texttt{mip\_rel\_gap}=0.01$ for \texttt{HiGHS} as default in \texttt{PDMO.jl}
\item \textbf{GNN.} 
For the experiments, we conducted inference tests using a GNN trained on a dataset of 10{,}000 graphs, each comprising 20 nodes, with all inference performed on CPU. Detailed training procedures and model specifications are provided in Appendix~\ref{section: gnn-details}. Our \texttt{PDMO.jl} package supports two inference backends: one is based on \texttt{PyCall.jl}~\cite{PyCall}, which executes inference through Python and provides faster bipartization performance, and the other one is based on \texttt{ONNXRunTime.jl}~\cite{ONNXRunTime_jl}, which enables Python-independent execution for improved accessibility. In the experiments reported in this paper, we use the \texttt{PyCall.jl} backend. Although the ONNX-based approach is fully supported, it typically incurs longer model loading time in our current implementation.
\end{itemize}
Wall-clock time results reported in this paper \textit{do not exclude} Julia JIT compilation time; subsequent runs within the same Julia session are expected to incur lower overhead, especially for MILPs and GNN inference.

\paragraph{Algorithms.}
We compare reformulations by two algorithms from \texttt{PDMO.jl}. 
\begin{itemize}
\item \textbf{Original ADMM}, where each partial augmented Lagrangian subproblem is solved using either \texttt{IPOPT}~\cite{wachter2006implementation} with default linear solvers when subproblems are constrained nonlinear programs, or customized solvers when subproblems have special forms, e.g., proximal mappings, linear systems, and one-dimensional minimization. 
\item A variant of \textbf{FLiP-ADMM} known as \textit{doubly linearized ADMM} (see Section~8 of~\cite{ryu2022large}) that performs only proximal gradient updates, i.e., it queries gradient oracles of $f_i$ and proximal oracles of $g_i$. 
\end{itemize}
Our focus is on isolating the effect of different reformulations under a fixed ADMM variant. In particular, we do not tune ADMM parameters or compare among ADMM algorithms. Performance differences should therefore be attributed to reformulations rather than solver engineering.

\paragraph{Structure of Experiments.}
Sections~\ref{section: lp} and \ref{section: network} illustrate how the MBP structure can be exposed in linear program instances and examine how BFS- and MILP-based bipartizations influence the convergence behavior of ADMM. Sections~\ref{section: dcopf} and \ref{sec: consensus_opt} further evaluate the effectiveness of the proposed pipeline on more structured applications with GNN-based bipartization. In particular, Section~\ref{section: dcopf} studies the distributed DC optimal power flow problem, while Section~\ref{sec: consensus_opt} considers decentralized consensus optimization on random graphs, highlighting both scalability and generality. All experiments are implemented in \texttt{PDMO.jl}, executed with 16 Julia threads. Computations are carried out on a cluster of nodes, each equipped with 8 CPU cores and 30\,GB RAM.  
\subsection{Linear Program: A Case Study}\label{section: lp}
Consider a generic Linear Program (LP) of the form
\begin{align}\label{eq: lp}
    \min_{x}\{ c^\top x~|~ \underline{b} \leq Ax \leq \overline{b}, ~ l \leq x \leq u\},
\end{align}
where $x\in \R^n$ denotes the decision vector, $A\in \R^{m\times n}$ is the constraint matrix, $c\in \R^n$ is the cost vector, and $l,u\in \R^n$ and $\underline b, \overline b \in \R^m$ specify variable and constraint bounds. In order to run ADMM, a classic way to fit \eqref{eq: lp} into the two-block form \eqref{eq: two_block_problem} is to introduce a slack vector $z\in\R^m$ and rewrite the problem as
\begin{align}\label{eq: naive}
    \min_{x,z}\{c^\top x~|~ Ax - z = \bft{0},~ x\in[l,u],~ z\in[\underline b,\overline b]\}.
\end{align}
Although \eqref{eq: naive} is immediately compatible with ADMM, it ignores the structural sparsity of $A$ and provides no guidance on how to group variables and constraints in a manner conducive to parallelism. 

A more effective approach is to reveal latent block structure in $A$ and then apply the bipartization pipeline to obtain an ADMM-friendly reformulation, i.e., rewrite the generic LP into a multiblock form \eqref{eq: multiblock_nlp}. 
For this purpose, we adopt a lightweight co-clustering heuristic that reorders the rows and columns of $A$ to expose block patterns:
\begin{enumerate}
    \item Initialize $k$ empty column clusters and $k$ row clusters, where $k >0$ is specified by the user. Assign columns cyclically so that each cluster receives columns in turn, and place all rows in a single row cluster.
    \item Run a fixed number of alternating passes: each iteration assigns rows to the row cluster most represented among their incident column clusters, then columns to the column cluster most supported by their incident row clusters, always breaking ties toward smaller cluster indices.
    \item Upon termination of the above loop, columns in the same cluster form a block variable, and rows in the same cluster form a block constraint. Add slacks for block constraints containing inequality constraints.
\end{enumerate}

As a case study,  we apply this procedure to the instance \texttt{enlight\_hard} from MIPLIB~\cite{gleixner2021miplib}, and investigate the performance of FLiP-ADMM on different reformulations.   Fig.~\ref{figure: matrix_spy} shows the original matrix and its co-clustered permutation after choosing $k=4$ and running 5 alternating passes. The heuristic finds four block variables $\bft{x}_i$'s for  $i\in [4]$ and four block constraints $C_j$'s for $j\in [4]$. Fig.~\ref{figure: block view} shows the resulting block view and the induced coupling graph. The node ``$C_2$" and ``$C_4$" in Fig. \ref{figure: block view}(b) refer to the second and fourth block constraints coupling $(\bft{x}_1, \bft{x}_2, \bft{x}_3)$ and $(\bft{x}_1, \bft{x}_3, \bft{x}_4)$, respectively. The other two block constraints involve exactly two blocks with edges labeled as $C_1$ and $C_3$.
\begin{figure}[ht]
    \centering
    \begin{subfigure}{0.48\textwidth}
        \centering
        \includegraphics[width=\linewidth]{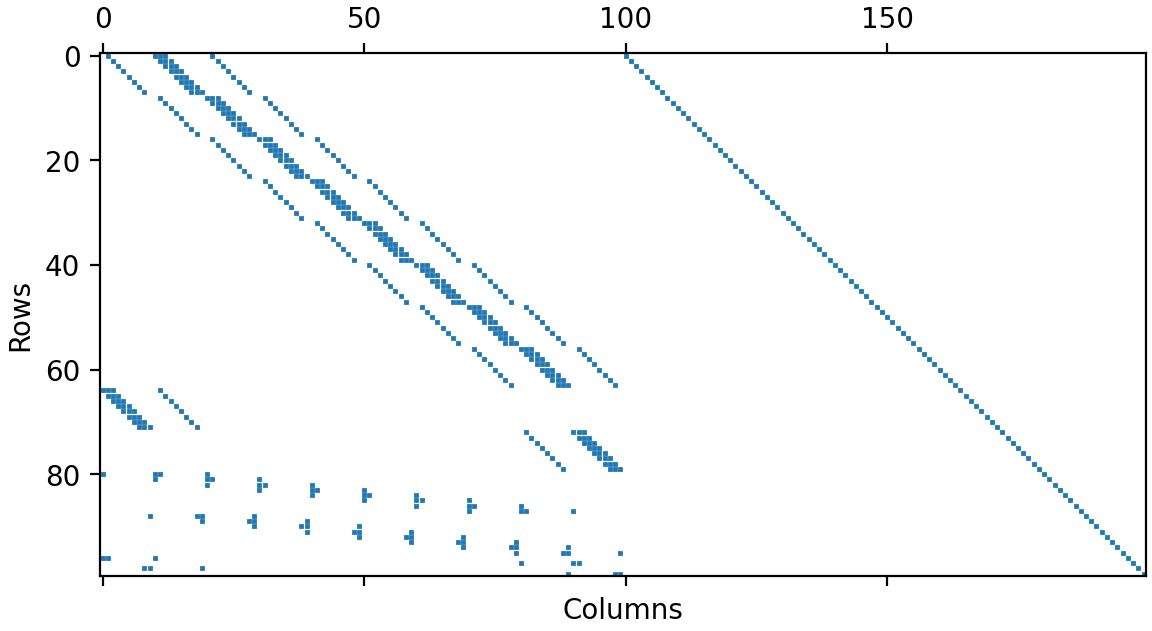}
        \caption{Original ${A}$}
    \end{subfigure}
    \hfill
    \begin{subfigure}{0.48\textwidth}
        \centering
        \includegraphics[width=\linewidth]{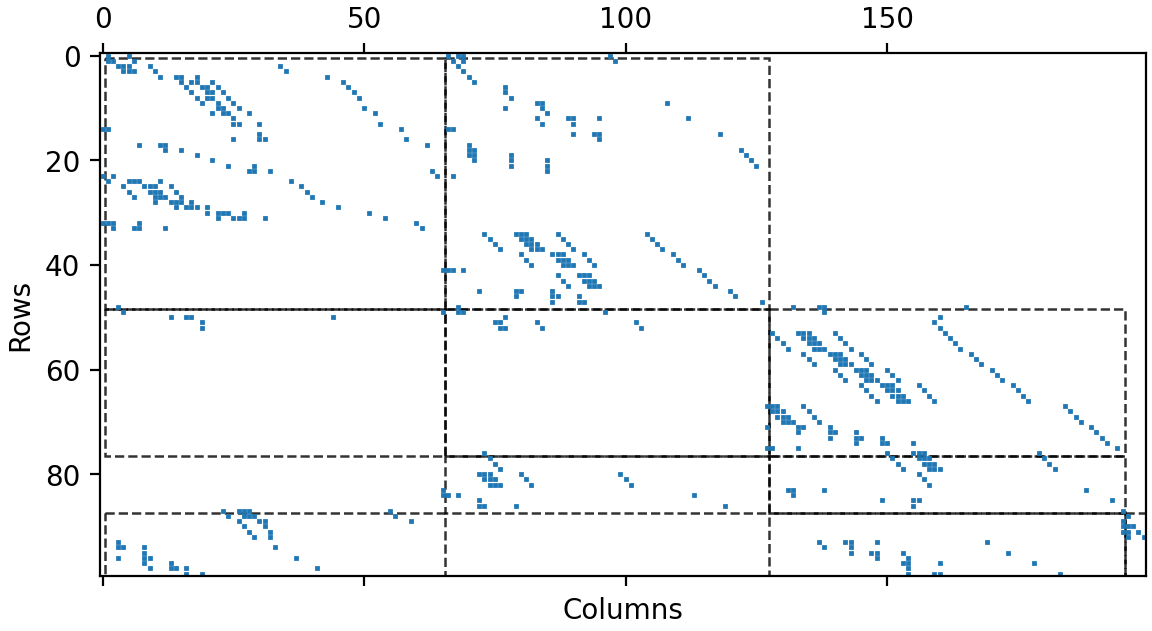}
        \caption{Permuted ${A}$}
    \end{subfigure}
    \caption{Original and permuted ${A}$ of \texttt{enlight\_hard}. Blocks of permuted ${A}$ are divided by dotted lines.} 
    \label{figure: matrix_spy}
\end{figure}

\begin{figure}[ht]
    \centering
    \begin{subfigure}{0.48\textwidth}
        \centering
        \includegraphics[width=\linewidth]{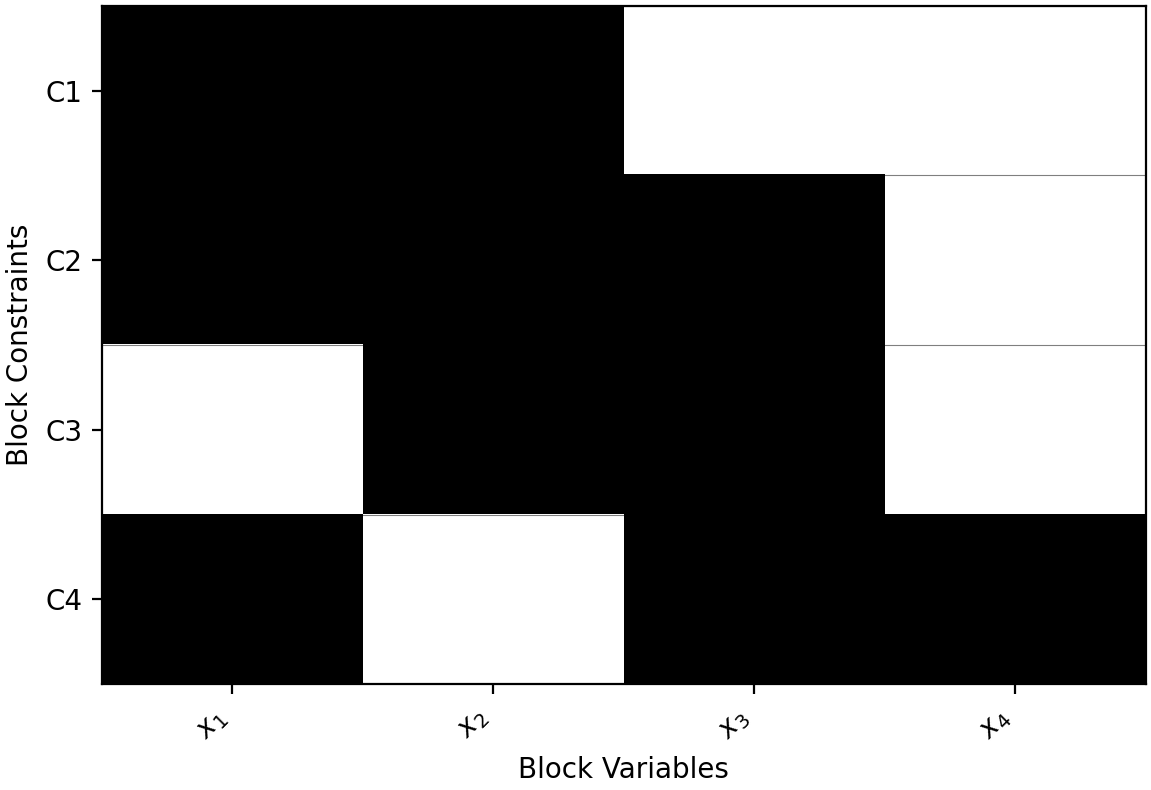}
        \caption{Block view of permuted ${A}$}
    \end{subfigure}
    \hfill
     \begin{subfigure}{0.48\textwidth}
        \centering
        \includegraphics[width=\linewidth]{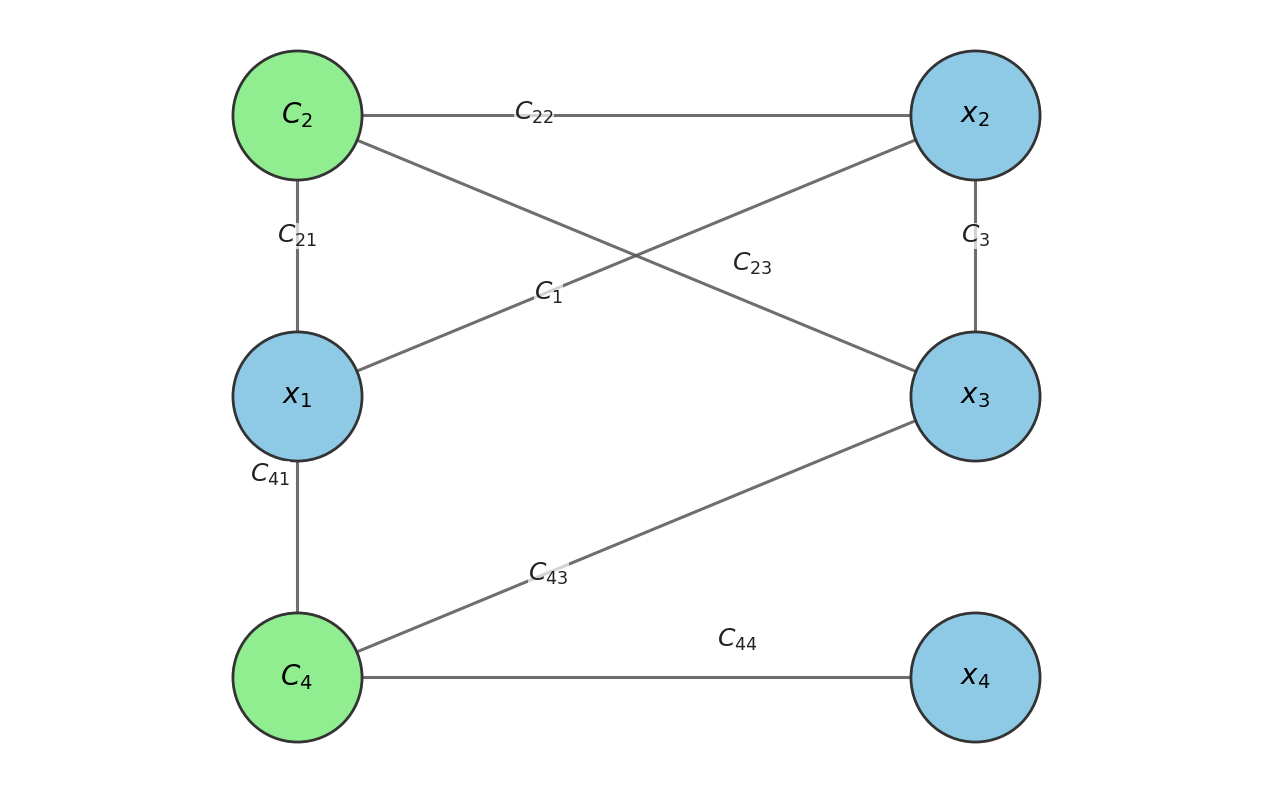}
        \caption{Graph representation}
    \end{subfigure}
    \caption{Block view of permuted ${A}$ and its graph representation.}
    \label{figure: block view}
\end{figure}

Fig.~\ref{figure: bipartite graph} compares the bipartite graphs produced by the MILP-based bipartization and the BFS-based bipartization.  
The MILP formulation yields a clean separation in which all original block variables reside on the left side, while all constraint nodes are placed on the right. In contrast, the BFS heuristic produces a mixed assignment in which variables, constraints, and subdivision nodes may appear on either side. Both constructions are valid bipartitions; the differences in structure reflect the
distinct mechanisms by which MILP and graph traversal assign nodes to partitions.

\begin{figure}[ht]
    \centering
    \begin{subfigure}{0.48\textwidth}
        \centering
        \includegraphics[width=\linewidth]{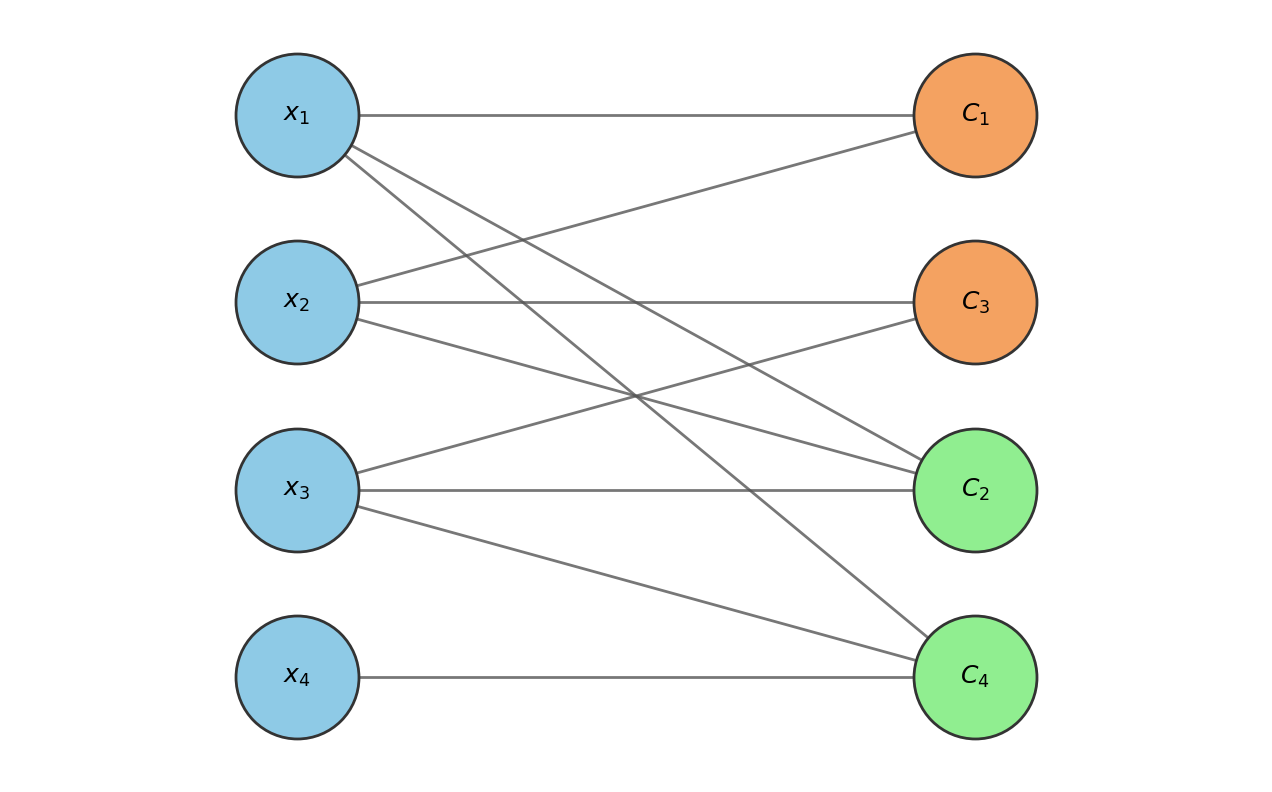}
        \caption{By MILP}
    \end{subfigure}
    \hfill
     \begin{subfigure}{0.48\textwidth}
        \centering
        \includegraphics[width=\linewidth]{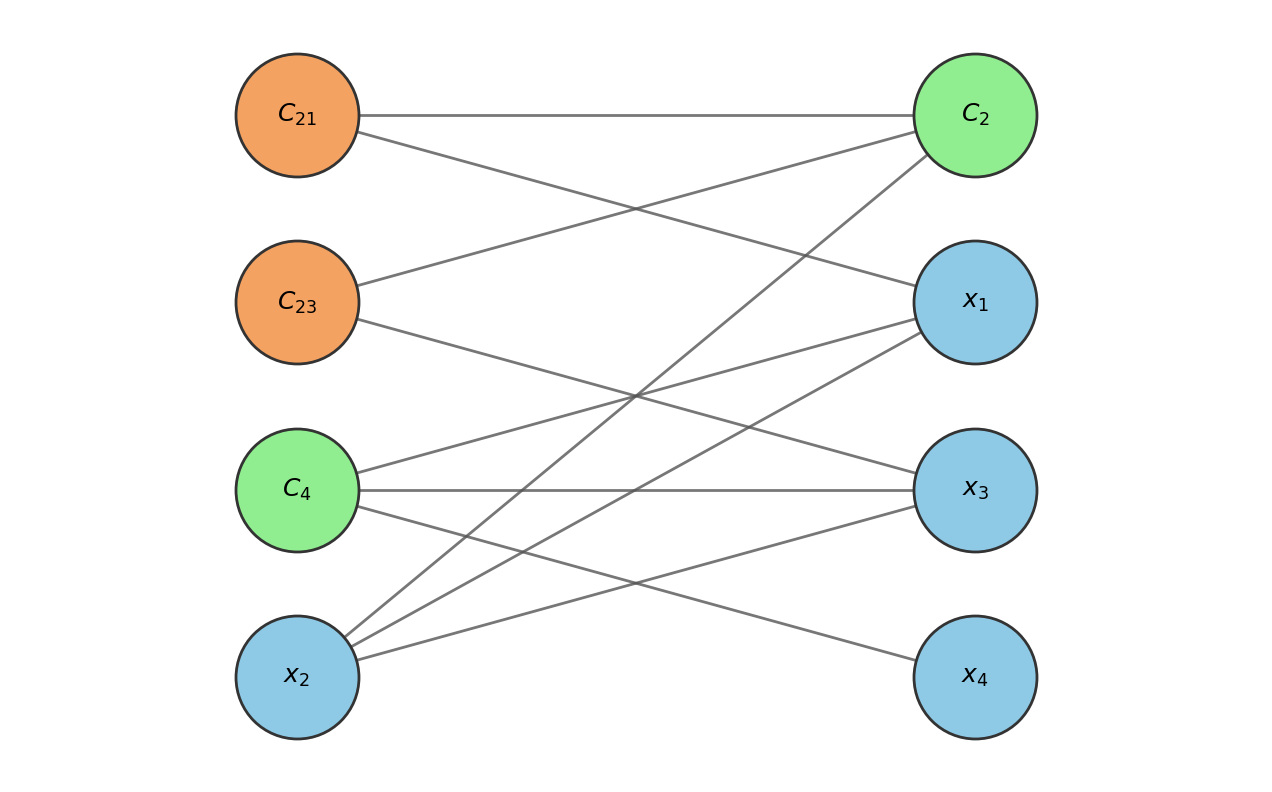}
        \caption{By BFS}
    \end{subfigure}
  \caption{Bipartite graph obtained by MILP and BFS.}\label{figure: bipartite graph}
\end{figure}   

Finally, Fig.~\ref{figure: primal_dual_residuals} reports the primal residual (plotted every 500 iterations for clarity) and dual residual trajectories of FLiP-ADMM ($\rho=1000$) applied to three formulations: the basic slack reformulation
\eqref{eq: naive}, the BFS-based bipartization, and the MILP-based bipartization. While all three formulations achieve small primal residuals quickly, two bipartization strategies lead to markedly faster convergence in dual residuals than the basic formulation, with the MILP-based decomposition exhibiting the steepest and most consistent residual decay. This case study demonstrates that even for an unstructured LP, revealing latent block patterns and applying bipartization can substantially improve the performance of ADMM. The advantages of MILP-based bipartization are already visible despite the weak inherent structure of the underlying model. We next examine a classic application of LP with an explicit multiblock structure.
\begin{figure}[ht]
    \centering
    \begin{subfigure}{0.48\textwidth}
        \centering
        \includegraphics[width=\linewidth]{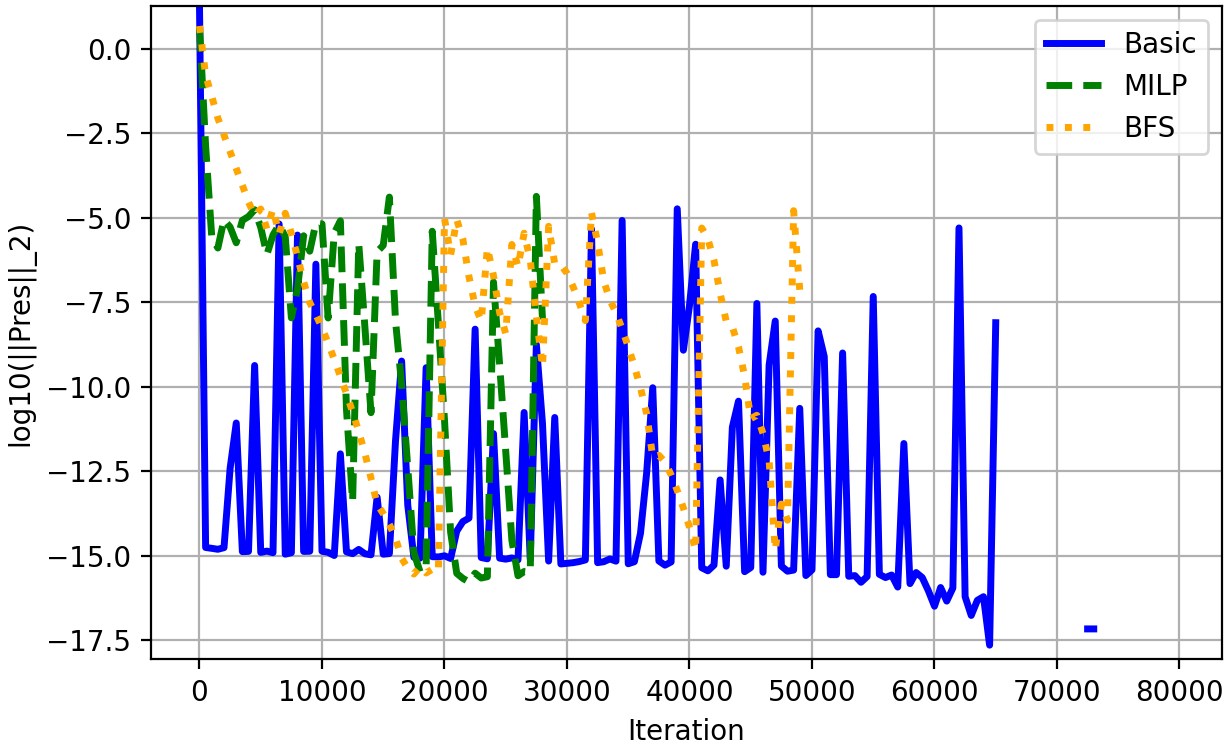}
        \caption{Primal Residuals}
    \end{subfigure}
    \hfill
    \begin{subfigure}{0.48\textwidth}
        \centering
        \includegraphics[width=\linewidth]{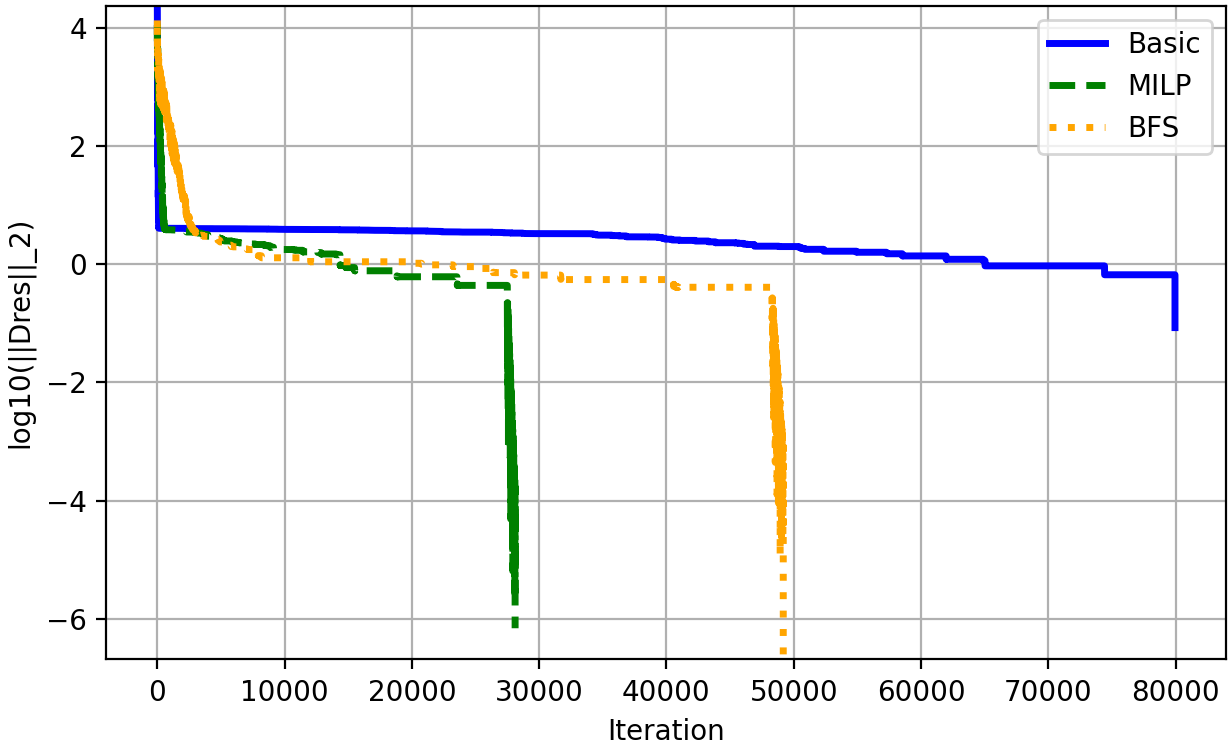}
        \caption{Dual Residuals}
    \end{subfigure}
    \caption{Trajectories of FLiP-ADMM on different reformulations.} 
    \label{figure: primal_dual_residuals}
\end{figure}

\subsection{Linear Program: The Network Flow Problem}\label{section: network}
We consider the classic network flow problem \eqref{eq: network_flow}. As discussed in Example \ref{example: network_flow}, this problem naturally conforms to the MBP format and is directly amenable to the proposed bipartization pipeline. The resulting ADMM subproblems are either one-dimensional quadratic minimizations over bounded intervals or projections onto linear subspaces, both of which admit closed-form solutions and are efficiently handled within \texttt{PDMO.jl}. Consequently, we focus exclusively on original ADMM in this subsection since first-order updates are unnecessary.

When we use linear arc costs $c_{ij}$, each network flow instance is an LP. Accordingly, we also apply ADMM to the classic reformulation \eqref{eq: naive} with variables $x$ and $z$. An interesting observation is that we have $\overline{b}=\underline{b}$ for network flow problems, implying that the $z$-block remains fixed throughout ADMM iterations. As a result, ADMM on \eqref{eq: naive} effectively reduces to the augmented Lagrangian method (ALM) and only needs to update $x$. The $x$-subproblem has a quadratic term whose Hessian is $\rho A^\top A$, where $A$ is the node–arc incidence matrix. Since $A^\top A$ is non-diagonal, leading to a coupled quadratic program, \texttt{PDMO.jl} will treat the $x$-subproblem as a general nonlinear program and solve it using IPOPT. While more specialized solvers could potentially exploit the structure of $A^\top A$ for improved efficiency, our comparisons in this subsection aim to provide a fair and handy baseline. We adopt network flow problems as representative test cases to highlight the general applicability of the multiblock formulation and the bipartization pipeline, rather than to benchmark against highly specialized and optimized solvers.

We fix $|\mathcal{A}| = 2000$, and for $|\mathcal{N}|\in \{200, 300,\cdots, 600\}$, we generate 10 random feasible directed network-flow instances as follows. First, we construct a connected simple graph on $|\mathcal{N}|$ nodes by creating a single cycle, ensuring every node has degree at least two; we then add additional edges uniformly at random, while preferably preserving a prescribed fraction of nodes with undirected degree two. 
Next, we assign a random orientation to each undirected edge to obtain a directed network. Arc costs and arc capacities are drawn uniformly from $[0,10]$ and $[0,40]$, respectively. Finally, to ensure feasibility, we sample a random feasible flow by drawing each arc flow uniformly from its capacity, and define supplies as the resulting net outflow at each node, followed by a scaling step to keep the supply magnitudes within the prescribed bound. Recall that the resulting MBP has $|\mathcal{A}|=2000$ block variables and the number of block constraints ranges from 200 to 600. ADMM with fixed $\rho=1$ terminates when the maximum violation of primal and dual residuals are less than 1e-4. The averaged iteration and total solution time is shown in Fig. \ref{figure: network_flow}. 
\begin{figure}[ht]
    \centering
    \begin{subfigure}{0.48\textwidth}
        \centering
        \includegraphics[width=\linewidth]{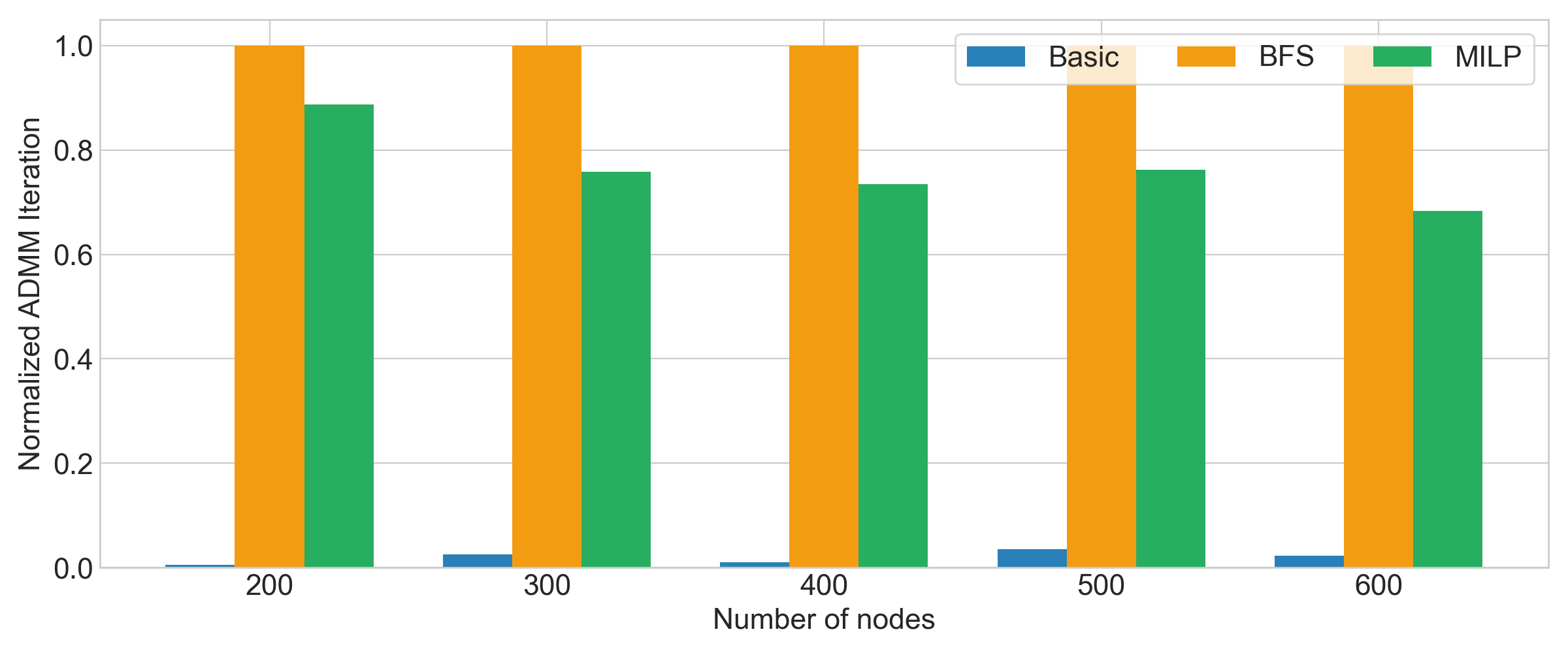}
        \caption{Iteration}\label{fig: network_iter}
    \end{subfigure}
    \hfill
    \begin{subfigure}{0.48\textwidth}
        \centering
        \includegraphics[width=\linewidth]{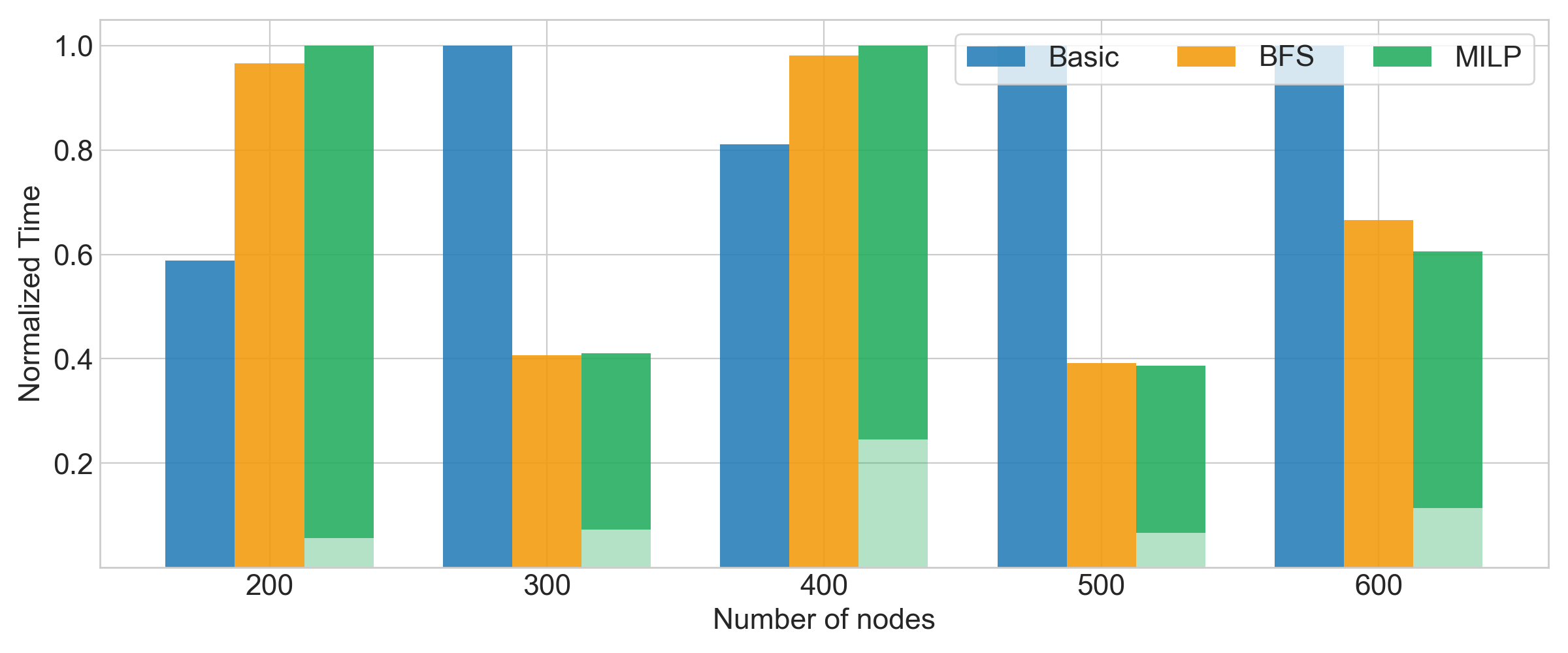}
        \caption{Total Time}
    \end{subfigure}
    \caption{Averaged iterations and total time comparison over 10 instances. Total time is the sum of partition time and ADMM time. Partition time of each method is shown in a lighter tone in the stacked bar. The max time/iteration is normalized to 1.0 for each group.} 
    \label{figure: network_flow}
\end{figure}

As shown in Fig. \ref{figure: network_flow}(a), ADMM applied to the classic reformulation, or equivalently, ALM applied to an equality constrained LP, converges in significantly fewer iterations. Recall that each ALM subproblem directly minimizes the augmented Lagrangian function with respect to the primal variable; in contrast, the bipartization-based ADMM decomposes the ALM subproblem with simpler computation, and hence requires more iterations to converge. Nevertheless, as the problem size grows, the proposed approach exhibits much better scalability and eventually achieves faster overall runtime\footnote{As noted earlier, the ALM subproblem is solved by IPOPT as a baseline; specialized solvers could be faster.} when accounting for both the partitioning cost and the ADMM iterations. Moreover, the MILP-based bipartization consistently yields fewer ADMM iterations than the BFS-based approach. While the MILP procedure may introduce overhead for small instances, its advantage in reducing iteration counts becomes increasingly pronounced for larger problems, eventually compensating for the additional partitioning time.

In the next two subsections, we turn to the Direct Current Optimal Power Flow (DC OPF) problem and Decentralized Consensus Optimization, where the graphical structures are explicit and domain-driven, and demonstrate more detailed numerical comparisons. 

\subsection{Distributed DC Optimal Power Flow}\label{section: dcopf}
In this subsection, we develop a distributed solution approach for the Direct Current Optimal Power
Flow (DC OPF) problem based on the proposed ADMM-based decomposition pipeline. DC OPF is a
fundamental optimization task in electric power systems: its objective is to determine the least-cost
dispatch of generators and the corresponding bus voltage angles that meet all demand while
satisfying nodal power balance and transmission flow limits under the DC power-flow approximation.

Consider a power network $G(N,E)$ with bus set $N$ and transmission line set $E$. The power network is undirected, but for ease of modeling, we assume each edge $(i,j)$ has a default direction, e.g., from $i$ to $j$, and only this representation is kept in $E$. Let $G_i$ denote
the set of generators located at bus $i$ ($G_i=\emptyset$ if bus $i$ has no generator), and
define $G := \cup_{i\in N} G_i$. Let $\delta_i$ denote the set of neighbors of bus $i$. The standard DC OPF model is
\begin{subequations} \label{eq: dc_opf}
\begin{align}
    \min_{\stackrel{\{P_g\}_{g\in G},\{\theta_i\}_{i \in N}}{\{f_{ij}\}_{(i,j)\in E}} }\quad 
        & \sum_{g \in G} c_g(P_g) \label{eq: dc_opf_obj}\\
    \mathrm{s.t.}\quad 
        & \sum_{g\in G_i}P_g - P^D_i = \sum_{j\in \delta_i} f_{ij}, ~\forall i \in N, \label{eq: dc_opf_nodal_balance}\\
        & f_{ij} = B_{ij}(\theta_i - \theta_j), ~\forall (i,j) \in E, \label{eq: dc_opf_flow_def}\\
        & \underline{P}_g \leq P_g \leq \overline{P}_g, ~\forall g \in G, \label{eq: dc_opf_gen_limit}\\ 
        & \underline{f}_{ij} \leq f_{ij} \leq \overline{f}_{ij}, ~ \forall (i,j) \in E. \label{eq: dc_opf_flow_limit}
\end{align}
\end{subequations}
The decision variables include generator outputs $P_g$, bus voltage angles $\theta_i$, and transmission line flows $f_{ij}$.
The objective \eqref{eq: dc_opf_obj} aggregates generation costs, where each $c_g(\cdot)$ is typically linear or quadratic. Constraint \eqref{eq: dc_opf_nodal_balance} enforces nodal power balance with $P^D_i$ denoting the demand at bus $i$. Under the DC approximation, line flows are modeled by \eqref{eq: dc_opf_flow_def}, where $B_{ij}$ is the line susceptance. Note that by \eqref{eq: dc_opf_flow_def} we have $f_{ji} = -f_{ij}$, and hence we only define $f_{ij}$ for $(i,j)\in E$. The remaining constraints impose generator and thermal limits. 

\paragraph{Motivation for Distributed Reformulation.} 
There is an increasing need to solve \eqref{eq: dc_opf} in a distributed manner. Large power systems often span multiple regions operated by different Independent System Operators (ISOs) that cannot freely exchange sensitive data. Moreover, the system size can make centralized computation or data
storage infeasible. To address these challenges, distributed numerical solutions are highly desirable. Suppose the bus set $N$ is partitioned into $P$ zones $Z_1,\dots,Z_P$, each operated by an agent indexed by $p \in [P]$. A transmission line $(i,j)\in E$ is called a \emph{tie line} if its endpoints lie in different zones
$z(i)\neq z(j)$. Such buses are \emph{boundary buses}. Let $\mathrm{NB}_p$ denote the neighboring
buses of zone $p$ and $\mathrm{TL}$ the set of all tie lines. See Fig. \ref{fig:zone_partition} for an illustrative example. 
\begin{figure}[ht]
\centering
\begin{tikzpicture}[
  bus/.style={circle, draw, thick, fill=white, inner sep=1.2pt, minimum size=16pt, font=\small},
  intline/.style={thick},
  tieline/.style={thick, dashed},
  zonebox/.style={rounded corners, thick, draw, fill opacity=0.08, inner sep=10pt},
  boundary/.style={bus, very thick, fill=yellow!30},
  zonelabel/.style={font=\small\bfseries}
]

\node[bus]      (1) at (0,1.2) {$1$};
\node[boundary] (2) at (0,0)   {$2$};
\node[boundary] (3) at (1.2,0.6) {$3$};

\node[boundary] (4) at (3.4,1.2) {$4$};
\node[boundary] (5) at (3.4,0)   {$5$};
\node[bus]      (6) at (4.6,0.6) {$6$};

\node[boundary] (7) at (1.7,-2.0) {$7$};
\node[boundary] (8) at (3.0,-2.0) {$8$};

\draw[intline] (1) -- (3);
\draw[intline] (2) -- (3);
\draw[intline] (4) -- (6);
\draw[intline] (5) -- (6);
\draw[intline] (7) -- (8);

\draw[tieline] (3) -- (4);
\draw[tieline] (2) -- (7);
\draw[tieline] (8) -- (5);

\node[zonebox, fill=blue!30,  fit=(1)(2)(3)] (Z1) {};
\node[zonebox, fill=green!30, fit=(4)(5)(6)] (Z2) {};
\node[zonebox, fill=orange!35,fit=(7)(8)]   (Z3) {};

\node[zonelabel, anchor=south west] at (Z1.north west) {$Z_1$};
\node[zonelabel, anchor=south west] at (Z2.north west) {$Z_2$};
\node[zonelabel, anchor=north west] at (Z3.south west) {$Z_3$};

\end{tikzpicture}
\caption{An illustrative 8-bus example with $P=3$ zones. 
The zones are $Z_1=\{1,2,3\}$, $Z_2=\{4,5,6\}$, and $Z_3=\{7,8\}$. 
Dashed edges denote tie lines, with $\mathrm{TL}=\{(2,7),(3,4),(5,8)\}$. 
Boundary buses are highlighted in yellow.
For each zone $p$, the set of neighboring buses is defined as 
$\mathrm{NB}_1=\{4,7\}$, $\mathrm{NB}_2=\{3,8\}$, and $\mathrm{NB}_3=\{2,5\}$.}\label{figure: opf_demo}
\label{fig:zone_partition}
\end{figure}
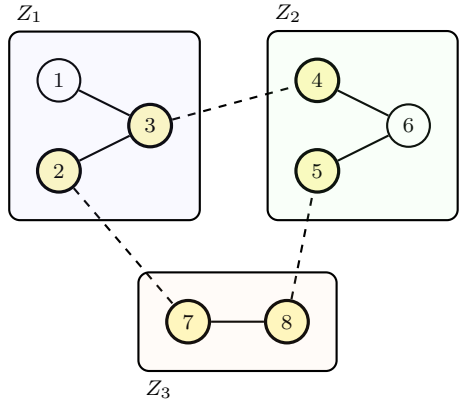

Each agent maintains local copies of the variables associated with its own buses and boundary buses. The local variable vector is
\begin{subequations}
\begin{align}
x_p := & \left(\{P_g^p\}_{g\in G_i, i\in Z_p},\;
            \{\theta^{p}_i\}_{i \in \overline{Z}_p},\;
            \{f^p_{ij}\}_{(i,j)\in E_p}\right),\\
\text{where~} E_p :=& E\cap((Z_p \times Z_p)\cup(\mathrm{NB}_p \times Z_p)\cup(Z_p\times \mathrm{NB}_p),   
\end{align}
\end{subequations}
and are subject to local constraints
\begin{align}
    X_p := \{x_p:\;&
    \text{\eqref{eq: dc_opf_nodal_balance} for } i\in Z_p,\;
    \text{\eqref{eq: dc_opf_gen_limit} for } g\in G_i,\; i\in Z_p, \notag \\
    &\text{\eqref{eq: dc_opf_flow_def} and \eqref{eq: dc_opf_flow_limit} for }
    (i,j)\in E\cap(Z_p\times (Z_p \cup \mathrm{NB}_p)) \}.
\end{align}
Using these local variables, the centralized DC OPF \eqref{eq: dc_opf} can be rewritten in
distributed form as:
\begin{subequations}\label{eq: distributed_dc_opf}
\begin{align}
    \min_{\{x_p\}_{p\in [P]}} 
    \quad & \sum_{p=1}^P \left(\sum_{g\in G_i,\, i\in Z_p} c_g(P_g^p)\right) 
        \label{eq: distributed_obj}\\
    \mathrm{s.t.}\quad 
        & \theta^{z(i)}_i = \theta^{z(j)}_i,\quad \forall (i,j)\in \mathrm{TL}, 
            \label{eq: theta_i_consensus} \\
        & \theta^{z(i)}_j = \theta^{z(j)}_j,\quad \forall (i,j)\in \mathrm{TL}, 
            \label{eq: theta_j_consensus}\\
        & f^{z(i)}_{ij}= f^{z(j)}_{ij},\quad \forall (i,j)\in \mathrm{TL}, 
            \label{eq: flow_consensus}\\
        & x_p \in X_p,\quad \forall p\in [P]. 
            \label{eq: dc_opf_local}
\end{align}
\end{subequations}
The objective \eqref{eq: distributed_obj} is separable across zones, and constraints \eqref{eq: theta_i_consensus}–\eqref{eq: flow_consensus} impose consensus on duplicated variables (voltage angles and tie-line flows) shared between adjacent regions. The remaining constraints \eqref{eq: dc_opf_local} enforce all DC power flow physics locally within each zone.
Importantly, the distributed formulation \eqref{eq: distributed_dc_opf} is precisely an instance of
the multiblock problem \eqref{eq: multiblock_nlp}: each zone corresponds to a block variable,
consensus constraints appear as pairwise couplings between blocks, and all other constraints and
cost terms remain local. We emphasize that the reformulation \eqref{eq: distributed_dc_opf} fits into the MBP format in a way that differs from the construction presented in the previous subsection, and still integrates smoothly into our bipartization pipeline. Subproblems of both original and FLiP-ADMM minimize a quadratic objective function subject to constraints $X_p$; we use \texttt{IPOPT} to solve these constrained problems. 

\paragraph{Performance of original ADMM over various networks.} We test our method on a collection of OPF networks of varying scales from MATPOWER~\cite{zimmerman2010matpower}. For each case, we use \texttt{Graphs.jl}~\cite{Graphs2021} to partition the underlying power grid into $P = 3,4,\cdots, 10$ zones to emulate multiple local system operators. For every such grid with given zones, we apply our pipeline (with BFS-, MILP, or GNN-based bipartization) to obtain a two-block ADMM-ready structure, and then run ADMM on the resulting distributed model until the $\ell_\infty$ norms of both primal and dual residuals are less than 1e-4.

Fig.~\ref{figure: opf_original_admm} reports the averaged ADMM iterations and total runtime (partition time plus ADMM time), where the average is taken over eight partition sizes for each test case. The BFS heuristic is extremely fast, requiring essentially no preprocessing time. In contrast, the MILP-based approach incurs higher partitioning cost but produces moderately better decompositions, leading to noticeably fewer ADMM iterations across all test cases. For small networks, the preprocessing overhead of the MILP approach dominates, resulting in longer total runtime. However, for larger networks, the improved decomposition quality often compensates for this overhead and yields faster overall convergence. The GNN-based approach achieves iteration counts and total runtime comparable to the MILP-based method in the three largest cases, while exhibiting some variability on smaller instances. This reflects a typical limitation of learning-based methods, whose performance depends on the representativeness of the training data and may deteriorate on structurally different instances. We also note that the MATPOWER cases are standard research benchmarks, partly anonymized or synthetic, which may further affect generalization behavior.
\begin{figure}[ht]
    \centering
    \begin{subfigure}{0.48\textwidth}
        \centering
        \includegraphics[width=\linewidth]{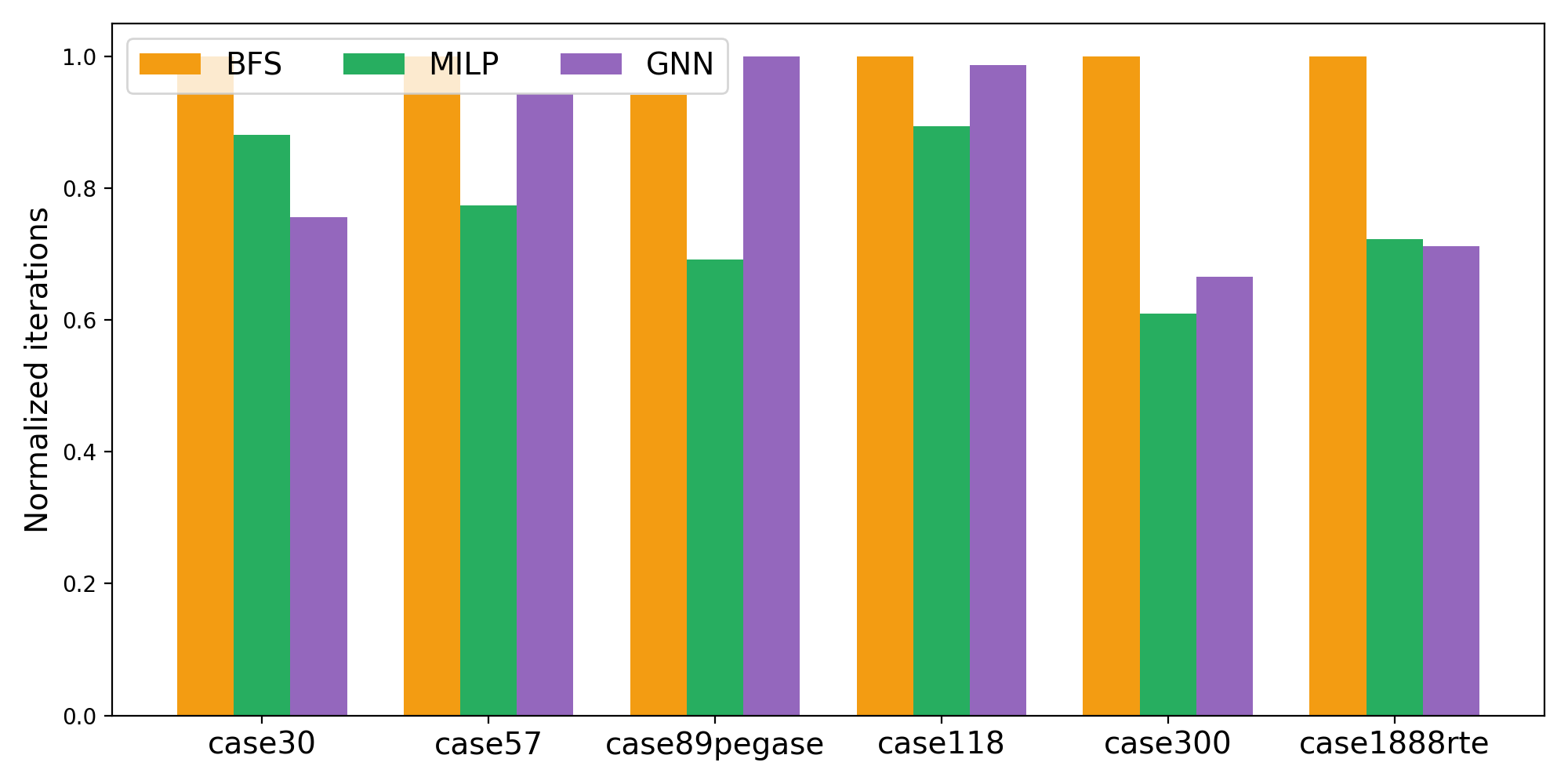}
        \caption{Iteration}
    \end{subfigure}
    \hfill
    \begin{subfigure}{0.48\textwidth}
        \centering
        \includegraphics[width=\linewidth]{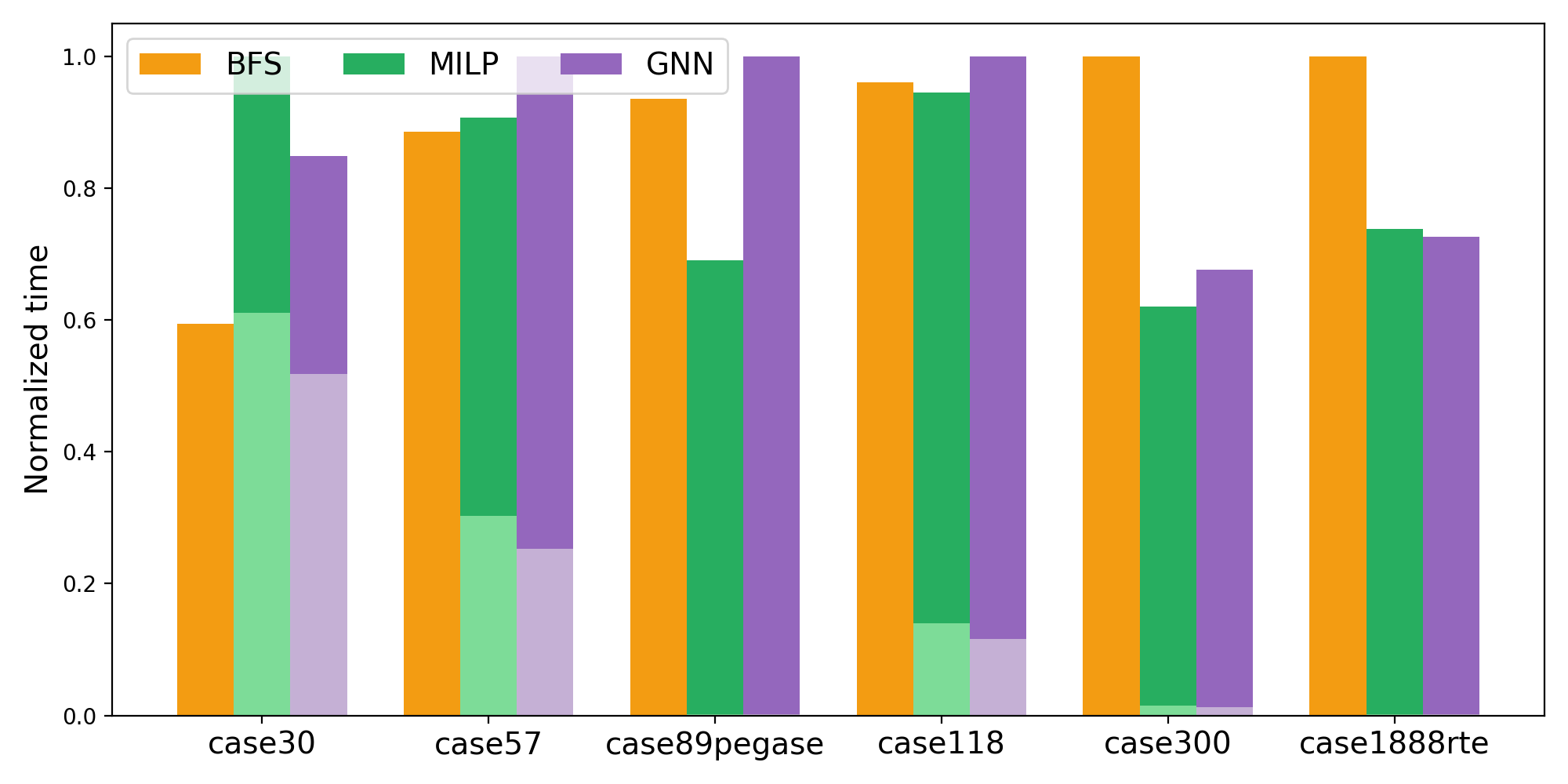}
        \caption{Total Time}
    \end{subfigure}
    \caption{Averaged iterations and total time comparison of original ADMM $(\rho=100)$. Partition time is shown in a lighter tone. The max time/iteration is normalized to 1.0 for each case.} 
    \label{figure: opf_original_admm}
\end{figure}

\paragraph{Performance of FLiP-ADMM with Many Zones.}
We further conduct a detailed case study on the \texttt{case57} network to examine how the behavior of FLiP-ADMM is influenced by different bipartization strategies under a larger number of zones. Specifically, we partition the grid into $P = 13,14,\dots,18$ regions and apply the bipartization pipeline using the same three partitioning methods. Fig.~\ref{figure: opf_flip_admm} summarizes the results.

In most cases, the MILP-generated bipartitions require fewer ADMM iterations than those obtained via BFS. The improved structural balance achieved by the MILP-based decomposition consistently translates into reduced ADMM runtime across nearly all partition sizes. The only notable exception occurs at $P=15$, where the BFS-based decomposition converges slightly faster, and the GNN-based approach achieves the shortest runtime. In addition, GNN also outperforms the other methods at $P=17$, suggesting that the learned model is capable not only of emulating MILP-quality bipartitions but, in some instances, of producing even more effective reformulations. Meanwhile, there are partition sizes for which the GNN-based method performs worse than both BFS and MILP, indicating room for improvement. Although the model captures important structural patterns, it does not always align with the partition characteristics defined by MILP. These observations suggest potential gains from refining the training objective and better aligning it with downstream ADMM performance.
\begin{figure}[ht]
    \centering
    \begin{subfigure}{0.48\textwidth}
        \centering
        \includegraphics[width=\linewidth]{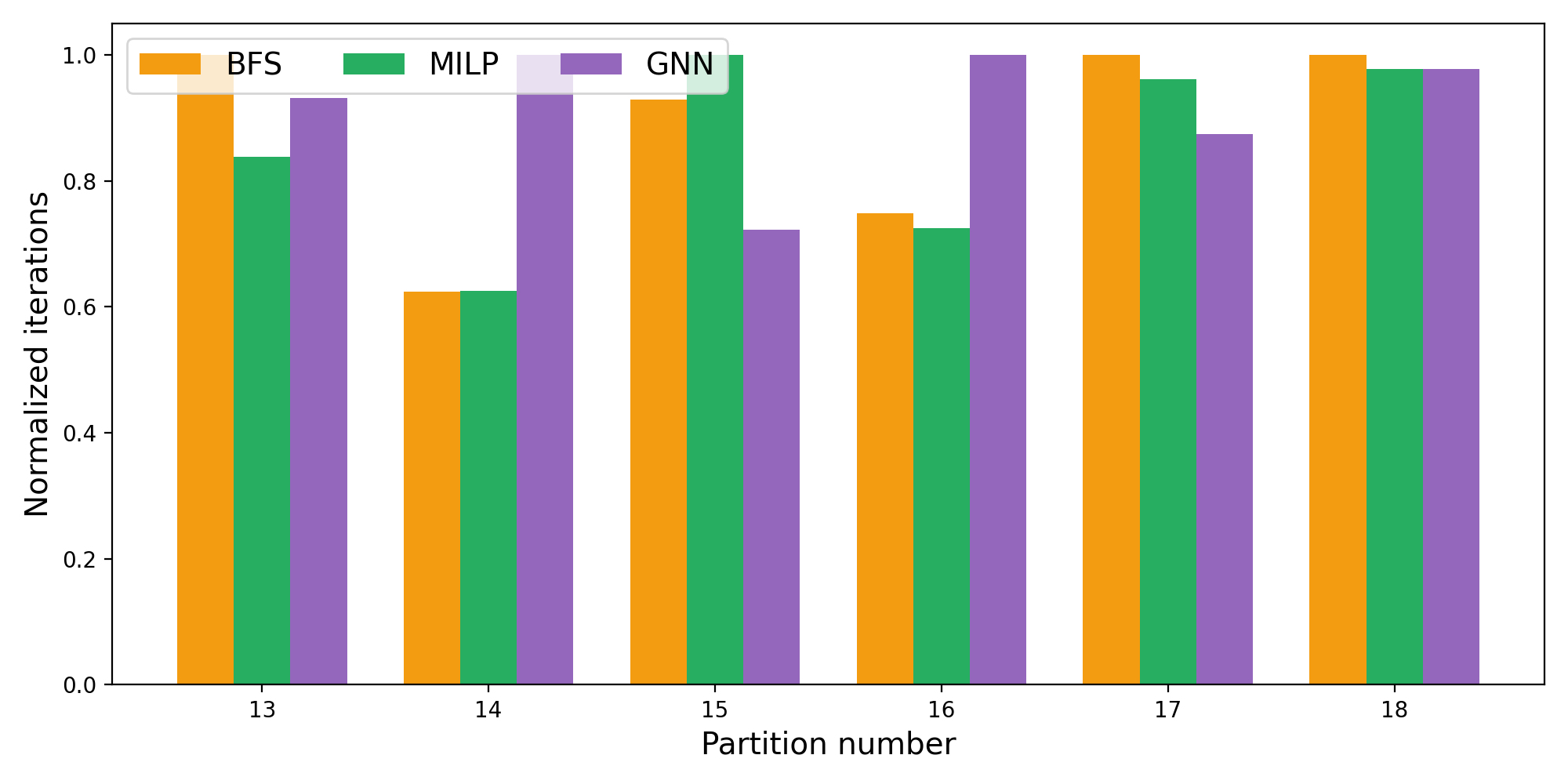}
        \caption{Iteration}
    \end{subfigure}
    \hfill
    \begin{subfigure}{0.48\textwidth}
        \centering
        \includegraphics[width=\linewidth]{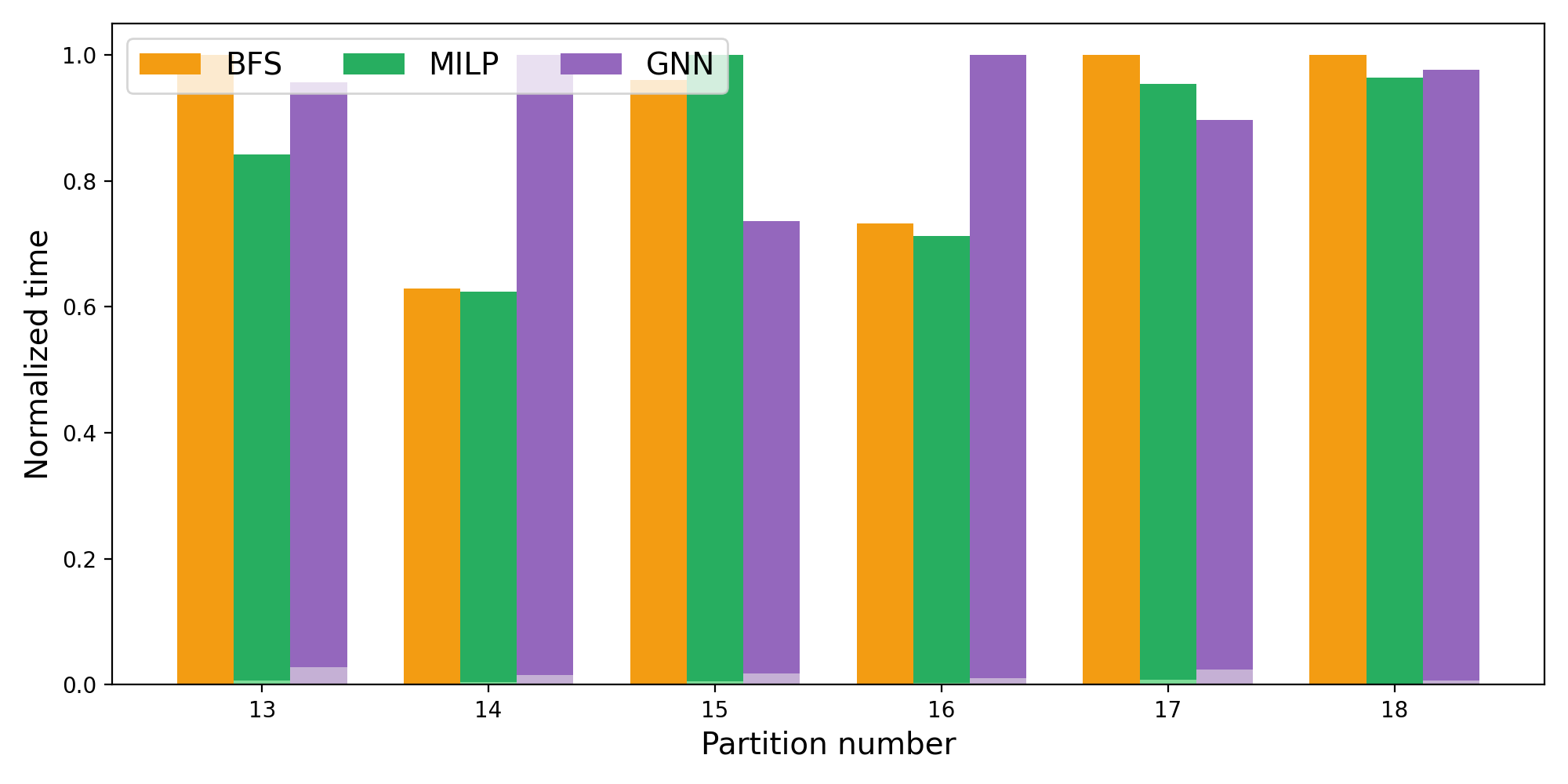}
        \caption{Total Time}
    \end{subfigure}
    \caption{Iterations and total time comparison of FLiP-ADMM $(\rho=1000)$ over different numbers of zones. Partition time is shown in a lighter tone. The max time/iteration is normalized to 1.0.} 
    \label{figure: opf_flip_admm}
\end{figure}
\subsection{Decentralized Consensus Optimization}\label{sec: consensus_opt}
We complement our experiments on structured problems with decentralized consensus optimization, where a connected network of agents collaboratively minimize the sum of local objective functions over a shared decision variable and information exchange is limited to neighboring agents. Formally speaking, given a connected network $G(V,E)$, the goal is to solve the problem 
\begin{align}\label{eq: decentralized_nlp}
    \min_{\bft{x}} \sum_{i\in V} f_i(\bft{x}),
\end{align}
and agent $i$ can only communicate with agent $j$ if $(i,j)$ or $(j,i)$ is in $E$.

Shi et al. \cite{shi_linear_2014} studied the convergence of ADMM over the following standard reformulation: 
\begin{subequations}\label{eq: standard_consensus}
\begin{align}
    \min_{\stackrel{\{\bft{x}_i\}_{i\in V}}{\{\bft{z}_{ij}\}_{(i,j)\in E}}} \quad & \sum_{i\in V} f_i(\bft{x}_i)\\
    \mathrm{s.t.} \quad & \bft{x}_i = \bft{z}_{ij}, ~\bft{x}_j = \bft{z}_{ij},~\forall(i,j) \in E.
\end{align}
\end{subequations}
This reformulation can be interpreted as a trivial bipartization of the communication graph: every edge $(i,j)\in E$ is subdivided by introducing an auxiliary variable $\bft{z}_{ij}$, so that all nodal variables $\{\bft{x}_i\}_{i\in V}$ lie in one partition and all auxiliary variables $\{\bft{z}_{ij}\}_{(i,j)\in E}$ lie in the other. While this construction guarantees that the transformed graph is bipartite and hence the successful execution of parallel ADMM, it enforces consensus through auxiliary variables on every original edge, leading to doubled couplings and increased communication in ADMM iterations. 

The proposed bipartization pipeline, in contrast, aims to create a more compact bipartite graph. In BFS-based bipartization, only edges that form odd cycles will be subdivided during graph traversal, and the MILP-based bipartization enforces direct control on the number of nodes in the final bipartite graph while maintaining balanced partitions. As a consequence, many neighboring agents remain directly connected across the two partitions and can perform alternating updates without introducing consensus variables. This reduced coupling not only lowers the problem dimension but also enables more efficient decentralized ADMM iterations, where updates propagate directly among (a subset of) adjacent agents rather than being mediated through edge-wise auxiliary variables. 

\paragraph{Problem Generation.} Given the number of nodes $|V|\in \{50, 100, 200\}$, we sample the graph connectivity ratio $p$ uniformly from $[2/|V|, 10/|V|]$, which in turn determines the number of edges as $|E| =p|V|(|V|-1)/2$. We then construct a random graph with $|V|$ nodes and $|E|$ edges while ensuring the graph stay connected, i.e., create a tree of $|V|$ nodes and then randomly add the remaining edges. We consider a decentralized consensus least squares problem in the form of \eqref{eq: decentralized_nlp}. We generate a ground truth vector $\bar{\bft{x}}\in \R^{500}$ with elements following the normal distribution $\mathcal{N}(0,1)$. For $i\in [|V|]$, we set $f_i(\bft{x}) = \|Q_i \bft{x} - q_i\|_2^2$, where $Q_i\in \R^{250 \times 500}$ is the linear measurement matrix whose elements follow the normal distribution $\mathcal{N}(0,1)$, and $q_i = Q_i \bar{\bft{x}} + \xi_i$ is the measurement vector whose elements are polluted by random noise $\xi_i$ following $\mathcal{N}(0, 0.1)$. 

\paragraph{Graph Features.}
For each $|V|\in \{50, 100, 200\}$, we generate 5 instances with different seeds, and apply the basic reformulation \eqref{eq: standard_consensus}, BFS-based bipartization, and MILP-based bipartization to obtain different bipartite graphs $\widehat{G}(\widehat{V}=\texttt{L}\sqcup \texttt{R} , \widehat{E})$. For MILPs solved by \texttt{HiGHS}, we
experiment with different $\texttt{mip\_rel\_gap}$ ranging from $\{1\%, 5\%, 10\%,20\%\}$. However, we note that a time limit of 60 seconds is set and hence the specified gap may not be achieved upon termination of MILP. Table \ref{table: graph_feature} summarizes the average features of original and bipartite graphs. We observe that BFS and MILP produce substantially more balanced bipartite graphs than the basic reformulation, reflected by the larger balance score, i.e., the ratio between $\min\{|\texttt{L}|,|\texttt{R}|\}$ and $\max\{|\texttt{L}|,|\texttt{R}|\}$. Moreover, the higher average degree $2|\widehat{E}|/|\widehat{V}|$ indicates a more compact distributed reformulation, with fewer auxiliary nodes and more direct couplings preserved across partitions, which is beneficial for parallel efficiency and communication cost. For MILP-based bipartization, increasing the relative optimality gap leads to a significant reduction in partitioning time. As we show next, the resulting partitions still exhibit good quality in practice.

The GNN-based bipartization consistently produces graphs that are substantially better than both the basic reformulation and the BFS heuristic in terms of structural metrics. In particular, the resulting balance scores are close to or even match those of the MILP-based approach, and the average degrees are similarly high, indicating compact reformulations with well-preserved couplings. The gaps between GNN and MILP metrics are generally small across all graph sizes, suggesting that the learned model successfully approximates high-quality MILP partitions. In terms of partition time, GNN achieves a favorable trade-off: it requires slightly more preprocessing time than BFS, but is significantly faster than MILP, especially under tight optimality gaps. As the graph size increases, the relative advantage in partition time becomes more pronounced while maintaining near-MILP structural quality. These results demonstrate that GNN provides an effective middle ground—combining the speed of heuristic methods with partition quality close to optimization-based approaches.

\begin{table}[ht]
\centering 
\begin{tabular}{cccccc}
\toprule
\multicolumn{1}{l}{$|V|$} & \multicolumn{1}{c}{Avg. $|E|$} & Part.      & $\frac{2|\widehat{E}|}{|\widehat{V}|}$ & $\frac{\min\{|\texttt{L}|,|\texttt{R}|\}}{\max\{|\texttt{L}|,|\texttt{R}|\}}$& Time(P) \\
\hline 
\multirow{7}{*}{50}  & \multirow{7}{*}{139}        
      & Basic      & 2.91	& 0.38	& 0.00  \\
 &    & BFS        & 3.73	& 0.83	& 0.00  \\
 &    & MILP(1\%)  & 4.16	& 0.89	& 16.42 \\
 &    & MILP(5\%)  & 4.14	& 0.89	& 12.87 \\
 &    & MILP(10\%) & 4.10	& 0.88	& 9.37  \\
 &    & MILP(20\%) & 4.02	& 0.88	& 3.68  \\ 
 &    & GNN        & 4.06	& 0.90	& 6.27  \\
\hline 
\multirow{7}{*}{100}  & \multirow{7}{*}{282}        
      & Basic      & 2.92	& 0.37	& 0.00  \\
 &    & BFS        & 3.77	& 0.74	& 0.00  \\
 &    & MILP(1\%)  & 4.21	& 0.93	& 54.57 \\
 &    & MILP(5\%)  & 4.20	& 0.90	& 45.52 \\
 &    & MILP(10\%) & 4.17	& 0.93	& 41.06 \\
 &    & MILP(20\%) & 4.12	& 0.92	& 22.26 \\
 &    & GNN        & 4.14	& 0.93	& 7.51  \\
\hline 
\multirow{7}{*}{200}  & \multirow{7}{*}{567}        
      & Basic      & 2.93	& 0.37	& 0.00  \\
 &    & BFS        & 3.79	& 0.71	& 0.00  \\
 &    & MILP(1\%)  & 4.17	& 0.94	& 62.39 \\
 &    & MILP(5\%)  & 4.17	& 0.95	& 59.65 \\
 &    & MILP(10\%) & 4.18	& 0.97	& 54.49 \\
 &    & MILP(20\%) & 4.15	& 0.91	& 48.27 \\
 &    & GNN        & 4.18	& 0.97	& 10.78 \\
\bottomrule
\end{tabular}
\caption{Average features of bipartite graphs generated by different methods.}\label{table: graph_feature}
\end{table}

\paragraph{Performance of ADMM.}
We run both the original ADMM and FLiP-ADMM with $\rho=10$\ on instances summarized in Table~\ref{table: graph_feature}, and report results (averaged over 5 seeds) in Fig.~\ref{figure: distributed_comparison_admm} and Fig.~\ref{figure: distributed_comparison_flip}. The BFS-, MILP-, and GNN-based bipartizations all yield nontrivial reductions in both iterations and total runtime for the original ADMM as well as FLiP-ADMM across all problem sizes, demonstrating the effectiveness of the proposed bipartization framework.

For the MILP-based approach, looser relative optimality gaps generally result in slightly worse reformulations, requiring more ADMM iterations. However, the partitioning overhead becomes the dominant factor in total runtime: larger gaps significantly reduce MILP solve time, and, despite the modest increase in ADMM iterations, often lead to faster overall convergence.

The GNN-based method typically requires fewer ADMM iterations than MILP with larger optimality gaps, indicating that it effectively emulates high-quality MILP partition decisions. In terms of total runtime, GNN is better than or comparable to the fastest MILP configuration across different gap settings, achieving near-MILP structural quality at substantially lower partition cost.

Finally, BFS incurs virtually no partition overhead and delivers surprisingly competitive performance on all generated instances. Although its structural metrics are generally inferior to MILP and GNN, it consistently provides meaningful speedups over the basic reformulation, highlighting its practicality as a lightweight heuristic.

\begin{figure}[H]
    \centering
    \begin{subfigure}{0.48\textwidth}
        \centering
        \includegraphics[width=\linewidth]{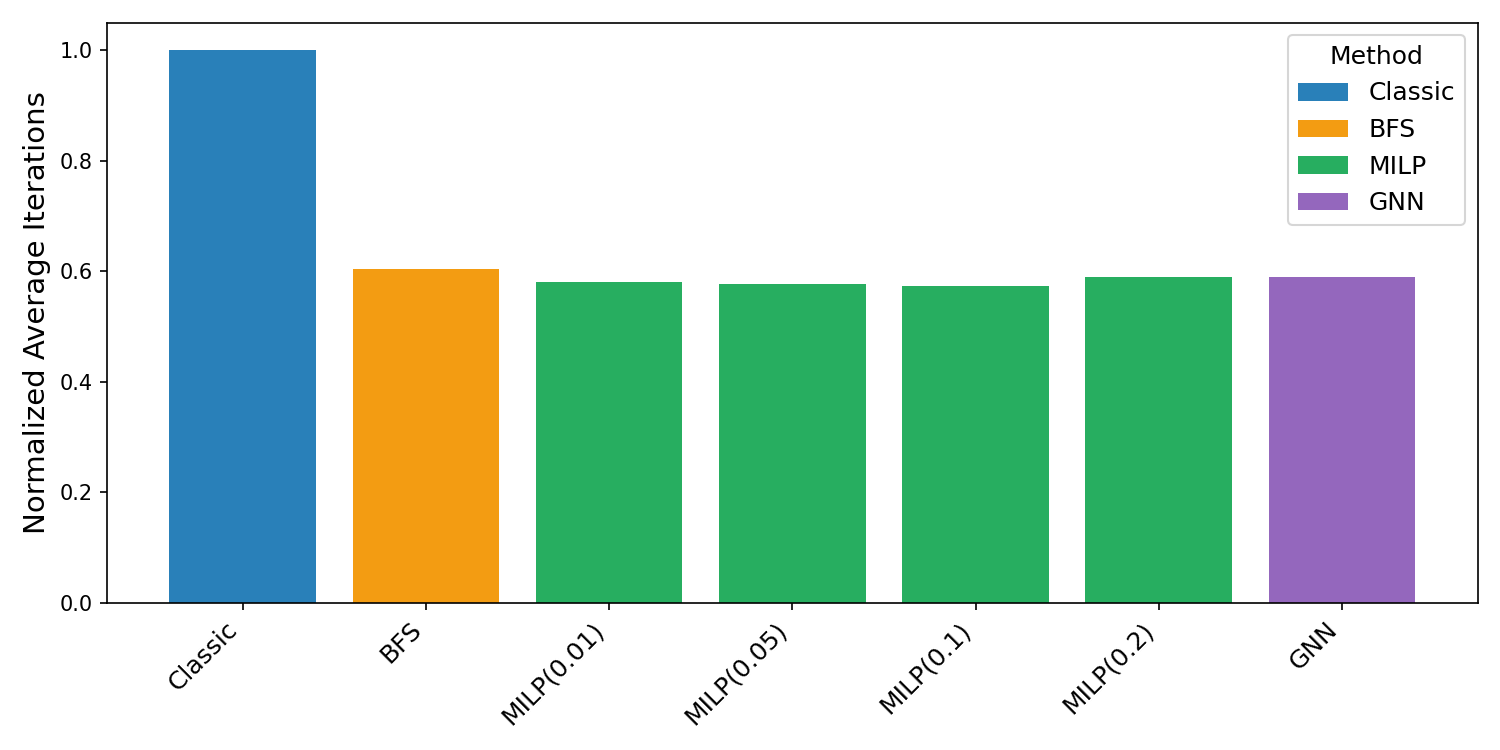}
        \caption{Original ADMM iteration, $|V|=50$}
    \end{subfigure}
    \hfill
    \begin{subfigure}{0.48\textwidth}
        \centering
        \includegraphics[width=\linewidth]{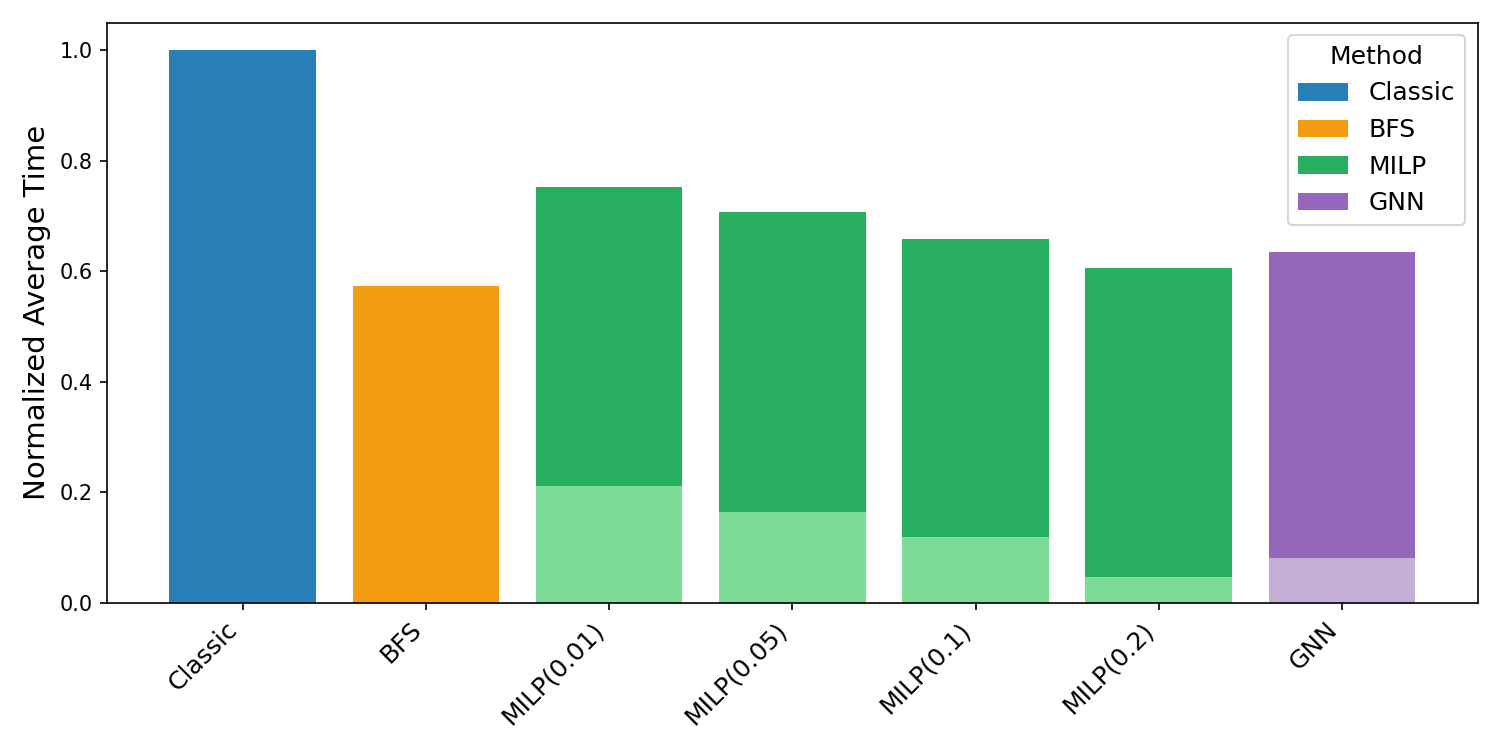}
        \caption{Original ADMM time, $|V|=50$}
    \end{subfigure}
    \begin{subfigure}{0.48\textwidth}
        \centering
        \includegraphics[width=\linewidth]{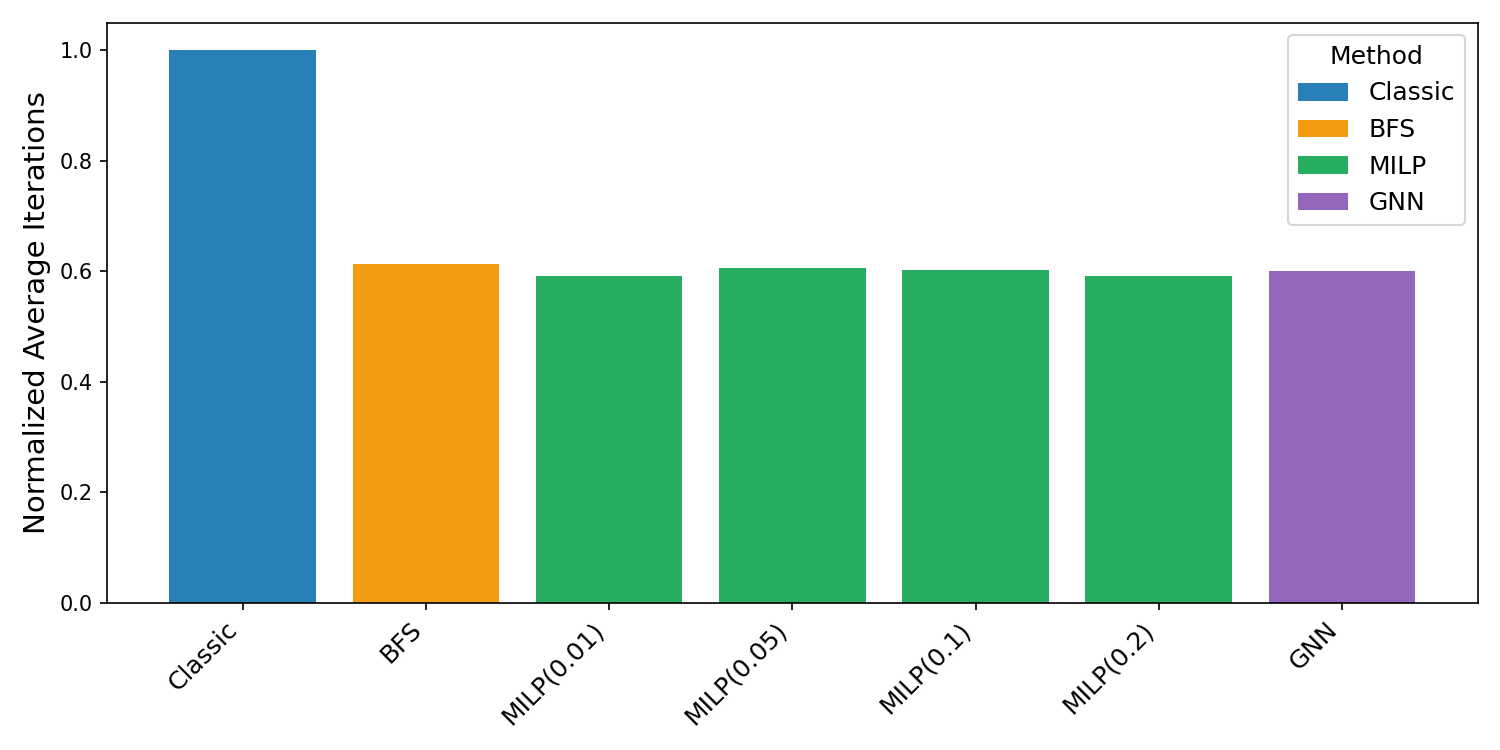}
        \caption{Original ADMM iteration, $|V|=100$}
    \end{subfigure}
    \hfill
    \begin{subfigure}{0.48\textwidth}
        \centering
        \includegraphics[width=\linewidth]{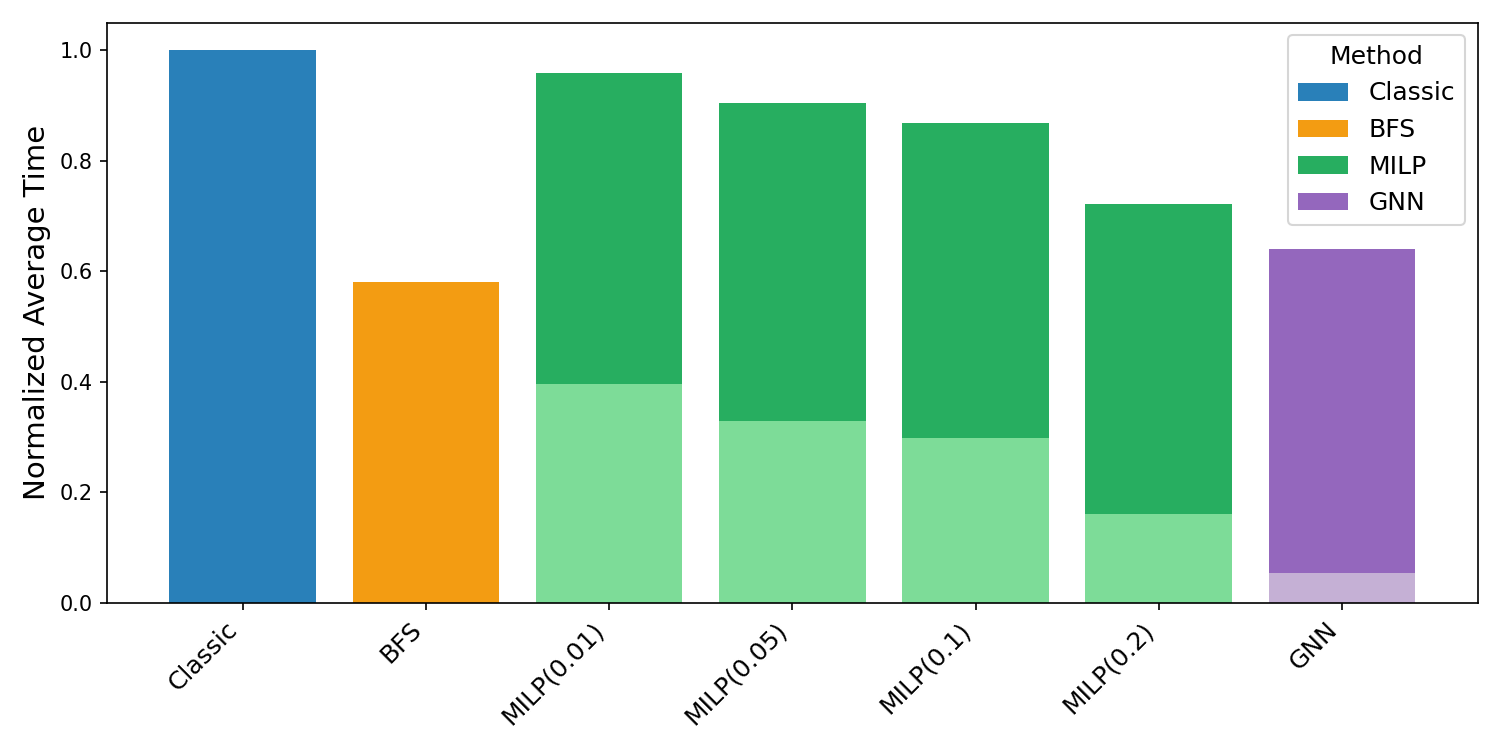}
        \caption{Original ADMM time, $|V|=100$}
    \end{subfigure}
    \begin{subfigure}{0.48\textwidth}
        \centering
        \includegraphics[width=\linewidth]{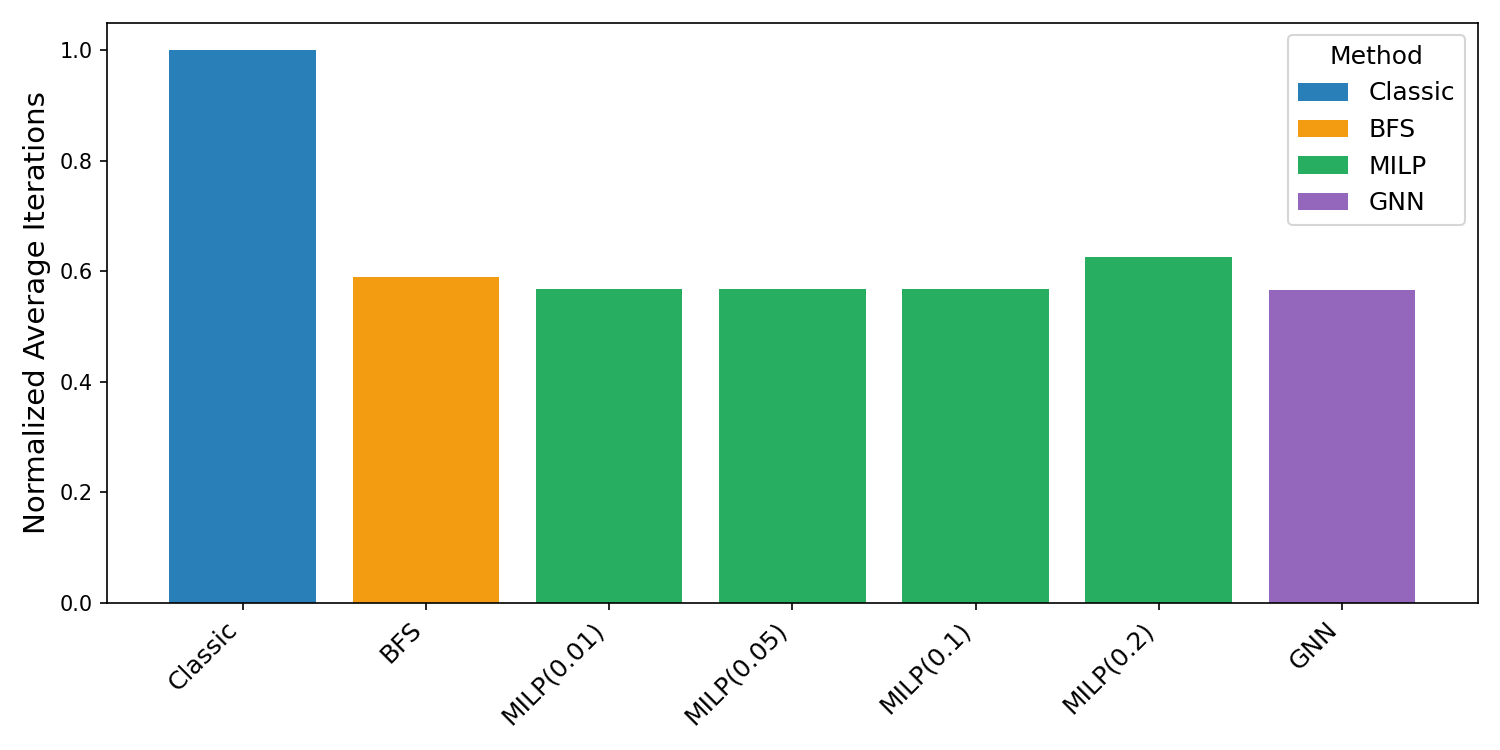}
        \caption{Original ADMM iteration, $|V|=200$}
    \end{subfigure}
    \hfill
    \begin{subfigure}{0.48\textwidth}
        \centering
        \includegraphics[width=\linewidth]{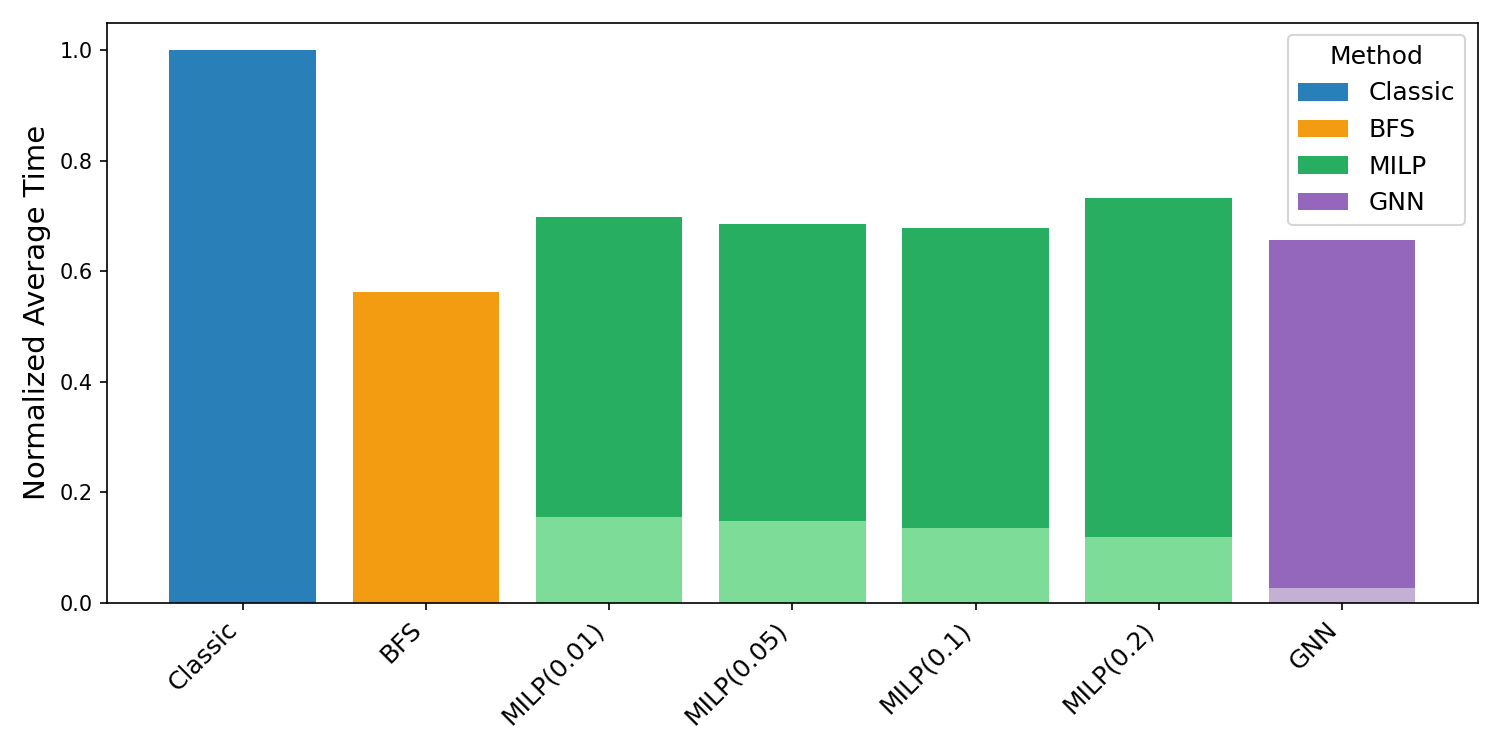}
        \caption{Original ADMM time, $|V|=200$}
    \end{subfigure}
     \caption{Averaged iterations and time comparison on different reformulations by ADMM. Partition time is shown in a lighter tone. The max time/iteration is normalized to 1.0.} 
    \label{figure: distributed_comparison_admm}
\end{figure}

\begin{figure}[ht]
    \centering
    \begin{subfigure}{0.48\textwidth}
        \centering
        \includegraphics[width=\linewidth]{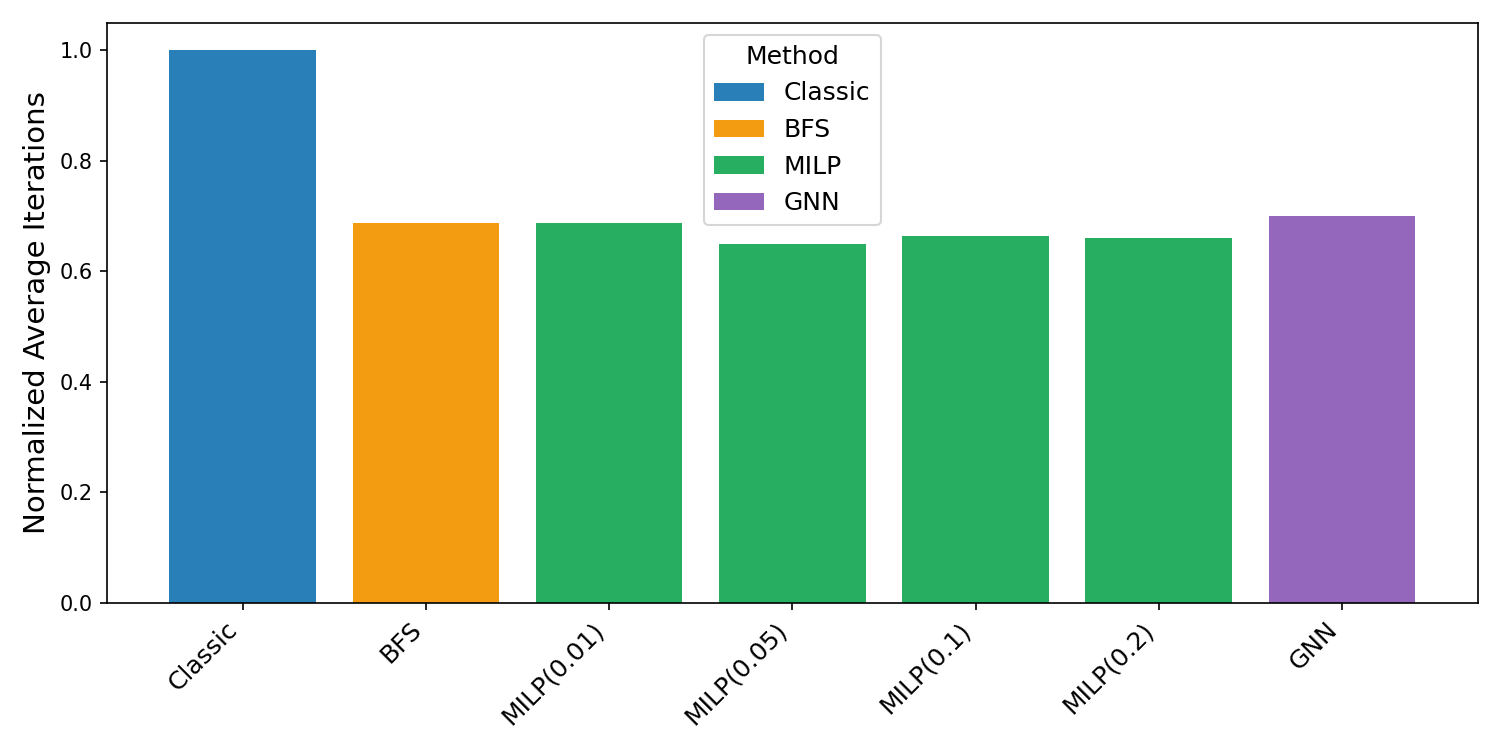}
        \caption{FLiP-ADMM iteration, $|V|=50$}
    \end{subfigure}
    \begin{subfigure}{0.48\textwidth}
        \centering
        \includegraphics[width=\linewidth]{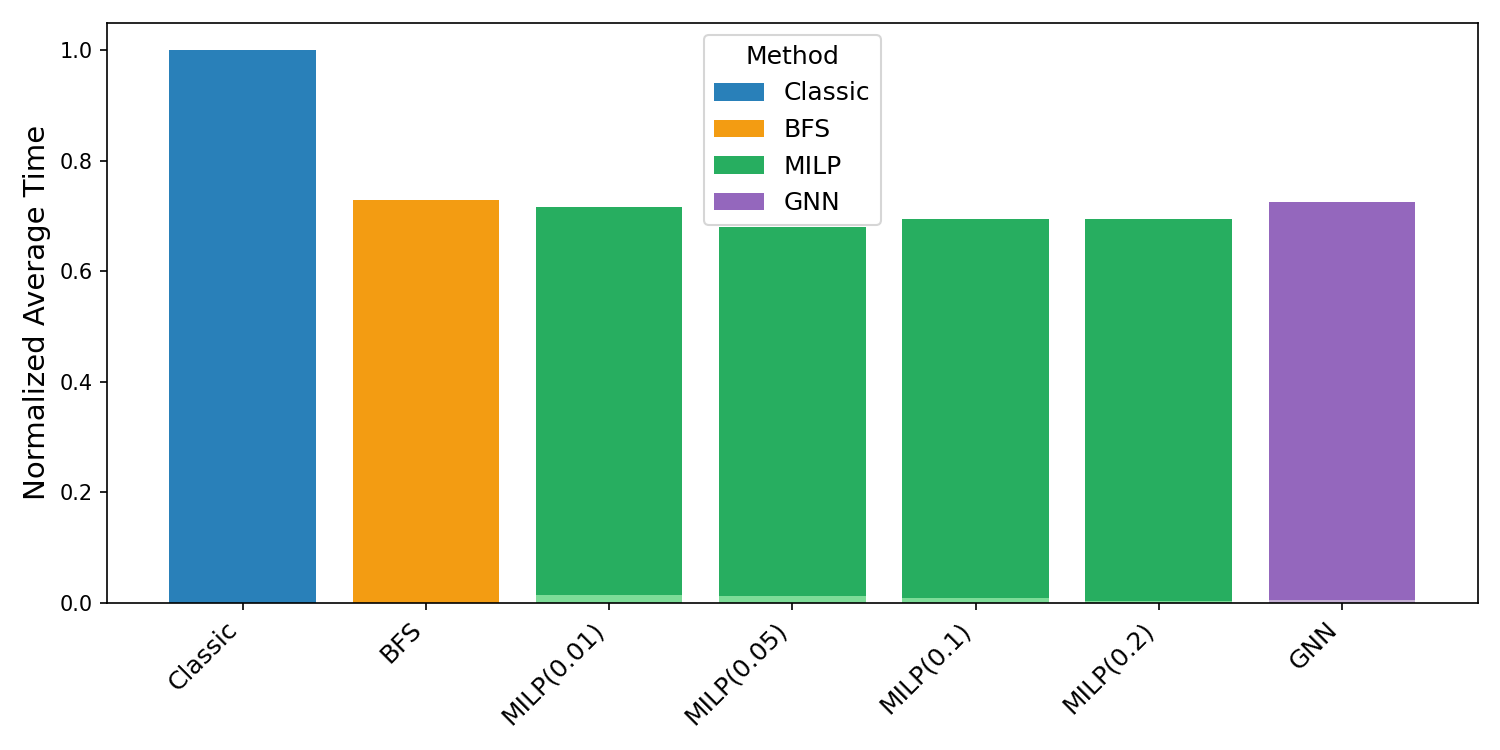}
        \caption{FLiP-ADMM time, $|V|=50$}
    \end{subfigure}
    \begin{subfigure}{0.48\textwidth}
        \centering
        \includegraphics[width=\linewidth]{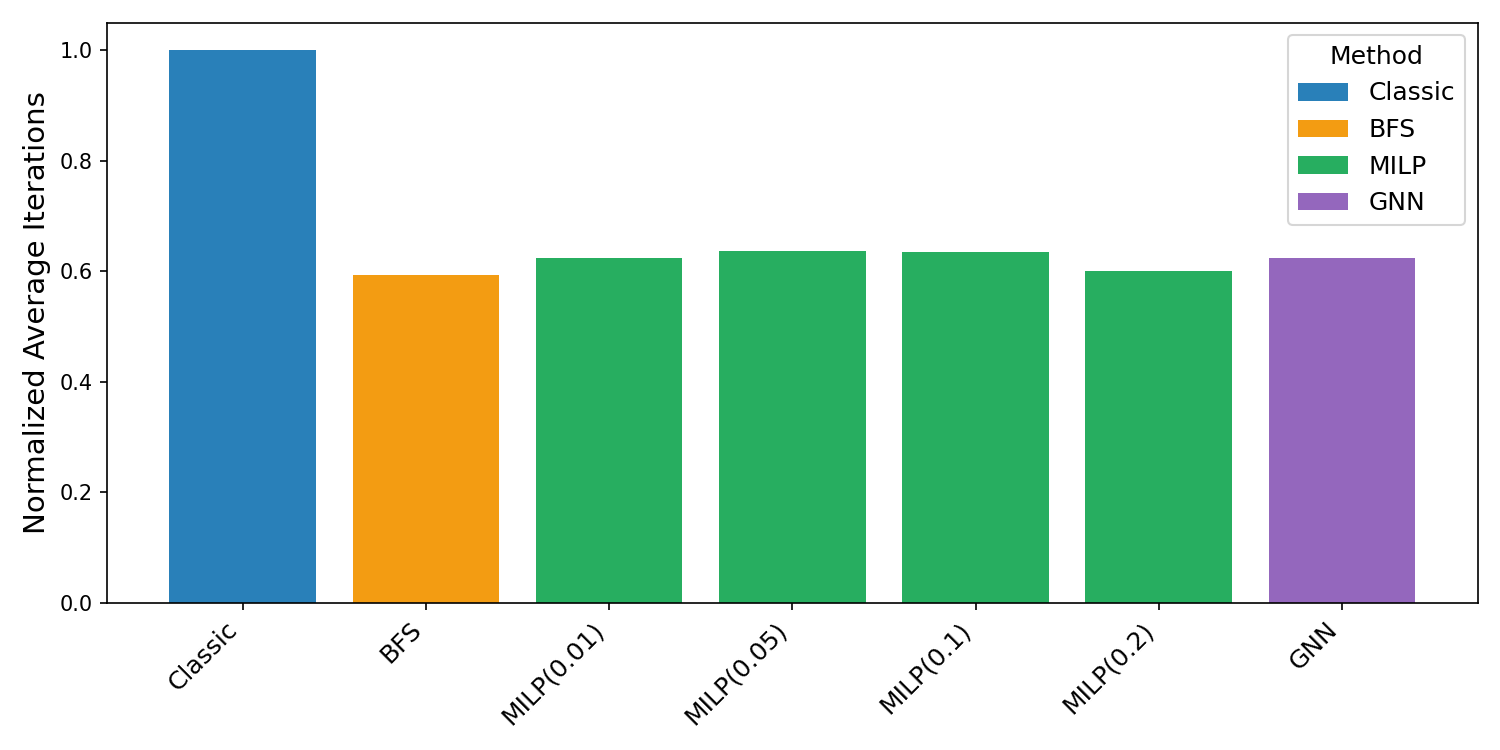}
        \caption{FLiP-ADMM iteration, $|V|=100$}
    \end{subfigure}
    \hfill
    \begin{subfigure}{0.48\textwidth}
        \centering
        \includegraphics[width=\linewidth]{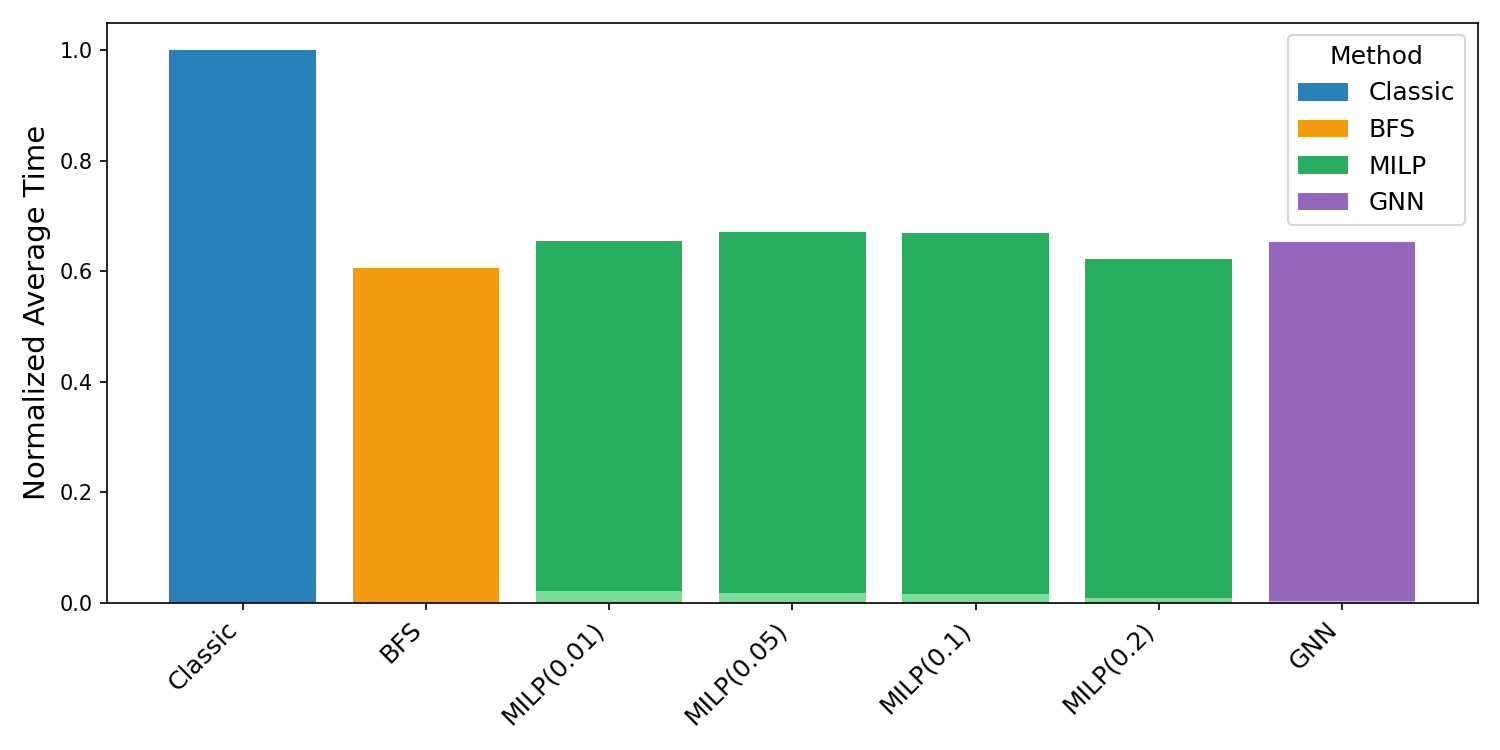}
        \caption{FLiP-ADMM time, $|V|=100$}
    \end{subfigure}
    \begin{subfigure}{0.48\textwidth}
        \centering
        \includegraphics[width=\linewidth]{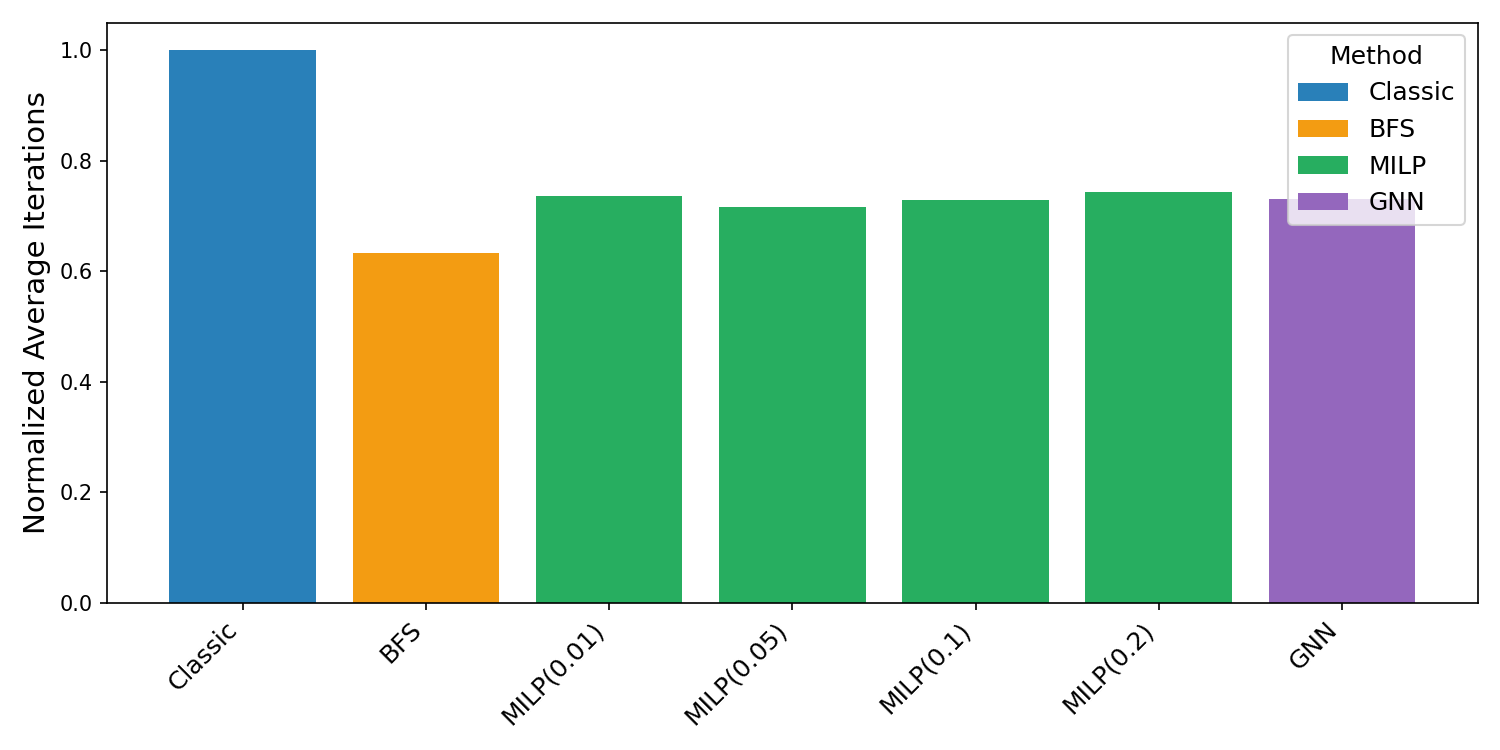}
        \caption{FLiP-ADMM iteration, $|V|=200$}
    \end{subfigure}
    \hfill
    \begin{subfigure}{0.48\textwidth}
        \centering
        \includegraphics[width=\linewidth]{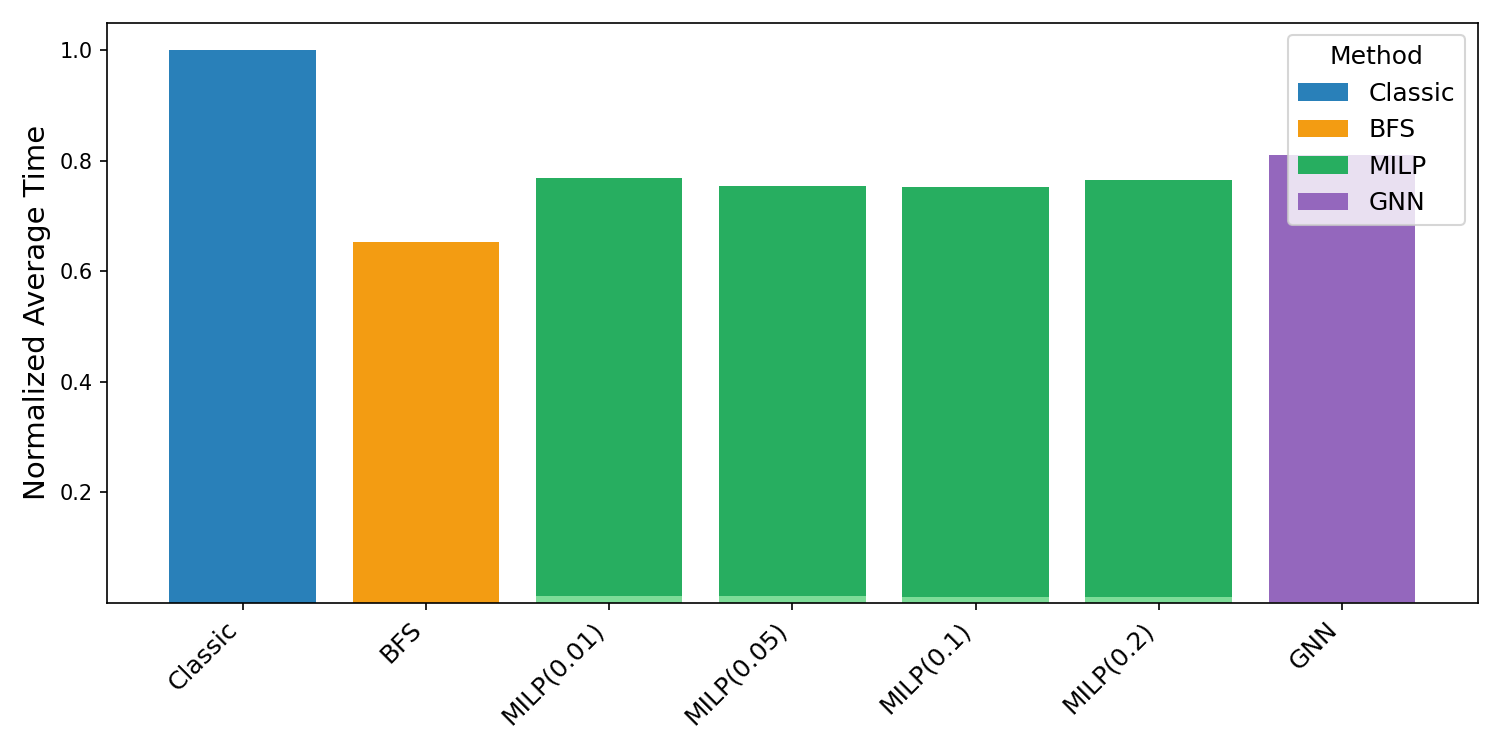}
        \caption{FLiP-ADMM time, $|V|=200$}
    \end{subfigure}
    \caption{Averaged iterations and time comparison on different reformulations by FLiP-ADMM. Partition time is shown in a lighter tone. The max time/iteration is normalized to 1.0.} 
    \label{figure: distributed_comparison_flip}
\end{figure}

\section{Conclusion}
This paper presents a fully automated pipeline that reformulates multiblock optimization models into a structure compatible with parallel ADMM. By representing the model as a coupling graph and enforcing bipartiteness through edge subdivision, the framework produces decompositions in which ADMM updates separate cleanly across blocks. Three bipartization strategies are studied: BFS as a lightweight heuristic, MILP as a performance-driven optimizer, and GNN as a student model trained to approximate MILP decisions.

Our experiments demonstrate that bipartization consistently improves scalability relative to the baseline formulation. However, several performance anomalies deserve careful interpretation. First, the MILP objective serves only as a surrogate metric for downstream ADMM performance; it is not a gold standard. While MILP-generated partitions are usually, but not always, associated with faster convergence, the surrogate objective does not perfectly align with the true ADMM contraction behavior. This mismatch explains cases where BFS performs competitively or where looser MILP optimality gaps yield lower total runtime. MILP can also incur substantial preprocessing cost, especially for small or one-off problems. In contrast, GNN achieves near-MILP partition quality with much lower preprocessing time, though its performance depends on how well the training data represent the target problem distribution.

These observations clarify appropriate usage scenarios. BFS provides a negligible-cost baseline suitable for quick experimentation or single-instance problems. MILP is most appropriate for offline analysis, fixed-structure models solved repeatedly with varying data, or challenging problems where decomposition quality is critical. GNN potentially offers an attractive trade-off for large-scale recurring problems whose structures follow a stable distribution, combining near-MILP quality with fast inference.

These observations also suggest several directions for future research. First, improved MILP surrogate objectives that more directly reflect ADMM convergence behavior are needed. The current structural metrics do not always capture the convergence properties of the ADMM iterations. Second, learning-based approaches could move beyond imitating MILP and instead directly optimize downstream performance metrics, such as predicted ADMM iteration counts or runtime. Integrating ADMM feedback signals into training, designing differentiable performance proxies, or developing hybrid optimization–learning frameworks are promising avenues. Finally, establishing theoretical connections between graph structure, bipartization quality, and ADMM convergence rates would provide a stronger foundation for both surrogate design and data-driven methods. The open-source implementation in \texttt{PDMO.jl} provides a practical platform for large-scale block-structured optimization and a foundation for these extensions.

\bibliographystyle{spmpsci}\bibliography{ref.bib}

\appendix 
\section{Implementation Details of the GNN Model and Training}\label{section: gnn-details}
\subsection{Graph Neural Network Architecture}\label{subsection: gnn-architecture}

The architecture of the proposed Graph Neural Network (GNN) is illustrated in Fig.~\ref{fig:GNN-Architecture}. The model is designed to learn effective node representations for bipartite graph partitioning by processing graph topology, node features, and edge features through a deep stacked encoder.
\begin{figure}[htbp] 
    \centering 
    \includegraphics[width=0.9\textwidth]{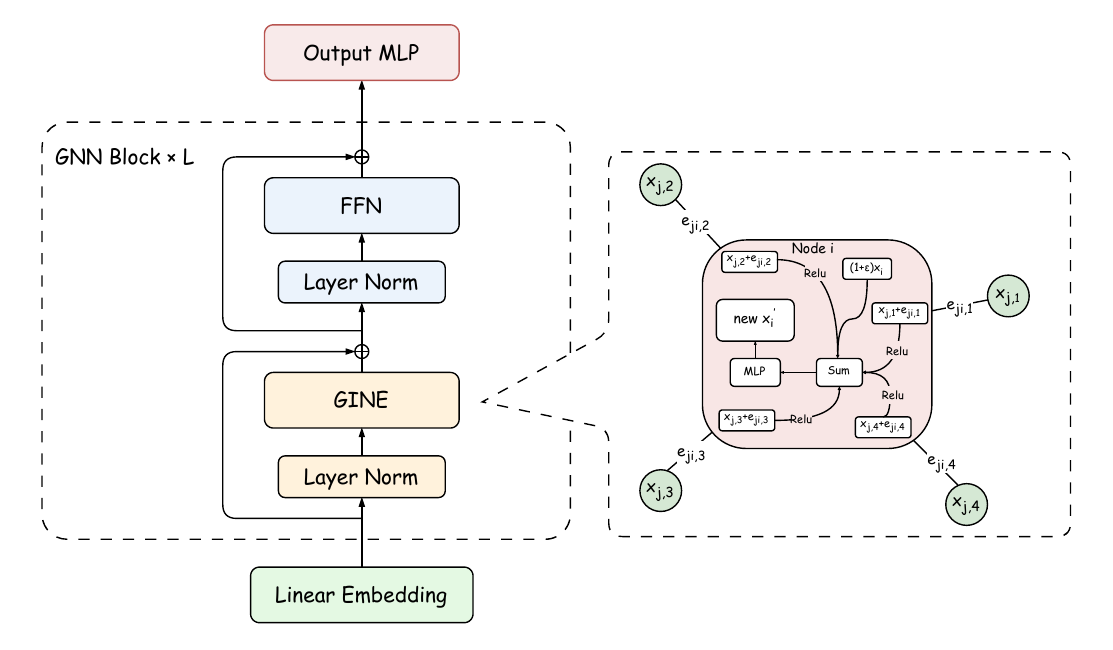}
    \caption{GNN architecture.} 
    \label{fig:GNN-Architecture} 
\end{figure}

\paragraph{\small Overall Pipeline}
The network consists of $L$ sequentially stacked GNN blocks, as indicated by ``GNN Block $\times L$'' in the figure. The input to the first block is the initial node feature matrix $\mathbf{H}^{(0)}$. Each block progressively refines these features, and the output of the final block, $\mathbf{H}^{(L)}$, is fed into a simple output MLP head to produce the final predictions (e.g., node assignment probabilities).

\paragraph{\small GNN Block Design}
Each GNN block is the core computational unit and follows a stabilized ``pre-normalization'' design inspired by modern transformer and graph network architectures. A single block executes the following sequence of operations:
\begin{enumerate}
    \item \textbf{Layer Normalization.} The input node features $\mathbf{H}^{(l-1)}$ are first normalized.
    
    \item \textbf{GINE Convolution Layer.} This is the central graph-aware operation. We employ the GINE ~\cite{hu2019strategies} (Graph Isomorphism Network with Edge features) convolution operator, chosen for its strong expressive power in capturing graph structures and its native support for edge features. The layer performs message passing, where each node aggregates information from its neighboring nodes and the features of the connecting edges.
    
    \item \textbf{Residual Connection.} The output of the GINE layer is added to the block input $\mathbf{H}^{(l-1)}$ via a skip connection, which helps mitigate vanishing gradient problems in deep networks.
    
    \item \textbf{Second Layer Normalization.} The result of the addition is normalized again.
    
    \item \textbf{Feed-Forward Network (FFN).} A small shared Multi-Layer Perceptron (MLP) is applied to each node's feature vector. This FFN, typically consisting of two linear layers with a nonlinearity (e.g., ELU) in between, provides additional nonlinear transformation capacity.
    
    \item \textbf{Second Residual Connection.} The output of the FFN is added to the input of the second normalization step.
\end{enumerate}

\paragraph{\small Detailed Message Passing in GINE Layer}
The right side of Fig.~17 zooms in on the update mechanism for a single node $i$ within a GINE\cite{hu2019strategies} layer. The update for node $i$ at layer $l$ is computed as follows:
\[
\mathbf{h}_i^{(l)} =
\text{MLP}^{(l)} \left(
(1 + \epsilon^{(l)}) \cdot \mathbf{h}_i^{(l-1)}
+ \sum_{j \in \mathcal{N}(i)} \text{ReLU} \left( \mathbf{h}_j^{(l-1)} + \mathbf{e}_{ji} \right)
\right)
\]

\noindent where:
\begin{itemize}
    \item $\mathbf{h}_i^{(l-1)}$ is the feature vector of node $i$ from the previous layer.
    \item $\mathcal{N}(i)$ is the set of neighboring nodes of $i$.
    \item $\mathbf{e}_{ji}$ is the feature vector of the edge connecting neighbor $j$ to node $i$.
    \item $\epsilon^{(l)}$ is a learnable scalar parameter.
    \item The summation $\sum$ represents the aggregation of messages from all neighbors.
    \item A shared MLP (distinct from the block FFN) transforms the aggregated information.
\end{itemize}

This formulation allows the model to learn a representation for node $i$ that captures both its own state and the relevant structural and feature-based information from its immediate graph neighborhood.

\paragraph{\small Output Head}
After $L$ blocks of iterative refinement, the final node embeddings $\mathbf{H}^{(L)}$ are passed through an output MLP head. This is a simple MLP that maps the high-dimensional node representations to the desired output space, generating, for example, a two-dimensional vector for binary node classification (left/right partition).



\subsection{Model Training Details}
This section provides a comprehensive description of the training protocol, hyperparameters, and implementation details of the GNN model introduced in Appendix~\ref{subsection: gnn-architecture}.

\paragraph{\small Optimization Objective}
The model was trained to perform node-level binary classification for graph bipartitioning. We employed the standard binary cross-entropy loss, averaged over all nodes in a batch:

\[
\mathcal{L} = -\frac{1}{N} \sum_{i=1}^{N} \left[ y_i \cdot \log(\hat{y}_i) + (1 - y_i) \cdot \log(1 - \hat{y}_i) \right]
\]

where  $y_i \in \{0, 1\}$  is the ground-truth partition label for node  $i$ , and  $\hat{y}_i$  is the model's predicted probability of the node  $i$.

\paragraph{\small{Data Preparation and Processing}}
Due to the scarcity of real-world training data in the optimization domain, we generate a synthetic dataset of Quadratic Programming (QP) problems with linear constraints to train and evaluate the GNN model. The dataset comprises 10{,}400 graph instances, split into 10{,}000 for training and 400 for testing.

\begin{itemize}
    \item \textbf{Graph Construction.} Each graph represents a QP instance, where nodes correspond to block variables. Variable dimensions are randomly sampled from $\{2,3,4,5\}$, and edges encode pairwise linear constraints with dimensions also sampled from $\{2,3,4,5\}$.

    \item \textbf{Parameter Generation.} All numerical parameters are independently sampled from standard normal distributions, producing diverse instances with varying curvature, scaling, and coupling.
    
    \item \textbf{Feasibility Guarantee.} After randomly generating the constraint parameters, we solve a small linear feasibility program to verify the existence of a solution. Any infeasible instance is discarded and regenerated. This procedure ensures that all training and test instances correspond to valid problems that admit at least one feasible solution.
\end{itemize}

\paragraph{\small Training Procedure and Hyperparameters.}
A comprehensive summary of the model architecture, optimization settings, and data configuration is provided in Table~\ref{table:hyperparameters}. Fig. \ref{fig:GNN Test Accuracy} illustrates the training performance (test accuracy versus epochs) of GNN under the configuration specified in Table \ref{table:hyperparameters}.

\begin{table}[h]
\centering
\begin{tabular}{@{}lll@{}}
\toprule
\textbf{Hyperparameter} & \textbf{Value/Setting} & \textbf{Description} \\
\midrule
\multicolumn{3}{c}{\textbf{Model Architecture}} \\
\cmidrule{1-3}
GNN Type & GINE & Graph Isomorphism Network with Edge features \\
Number of GNN Layers & 40 & Depth of the model (message passing steps) \\
Hidden Dimension & 64 & Dimensionality of node embeddings \\
Activation Function & ELU & Non-linear activation in MLPs \\
Normalization & LayerNorm & Applied before GNN layer and feed-forward network \\
Residual Connection & Yes & Skip connection in each GNN block \\
Dropout Rate & 0.5 & Applied after each feed-forward network \\
\midrule
\multicolumn{3}{c}{\textbf{Training Configuration}} \\
\cmidrule{1-3}
Optimizer & AdamW & Adam with decoupled weight decay \\
Learning Rate & 1e-5 & Initial learning rate \\
Weight Decay & 1e-5 & L2 regularization strength \\
Batch Size & 64 & Graphs per batch \\
Number of Epochs & 1000 & Maximum training epochs \\
Learning Rate Schedule & Cosine Annealing & Cosine decay \\
\midrule
\multicolumn{3}{c}{\textbf{Data \& Features}} \\
\cmidrule{1-3}
Node Feature Dimension & 20 & Initial node feature dimension (all-ones vector) \\
Edge Feature Dimension & 15 & Edge feature dimension (statistical, structural, geometric) \\
Graph Size (Training) & 20 nodes per graph & Fixed number of nodes in each generated graph \\
Dataset Size & 10,000 training, 400 test & Total 10,400 generated QP instances \\
\bottomrule
\end{tabular}
\caption{ Model hyperparameters and training configuration.}
\label{table:hyperparameters}
\end{table}

\begin{figure}[htbp] 
    \centering 
    \includegraphics[width=0.9\textwidth]{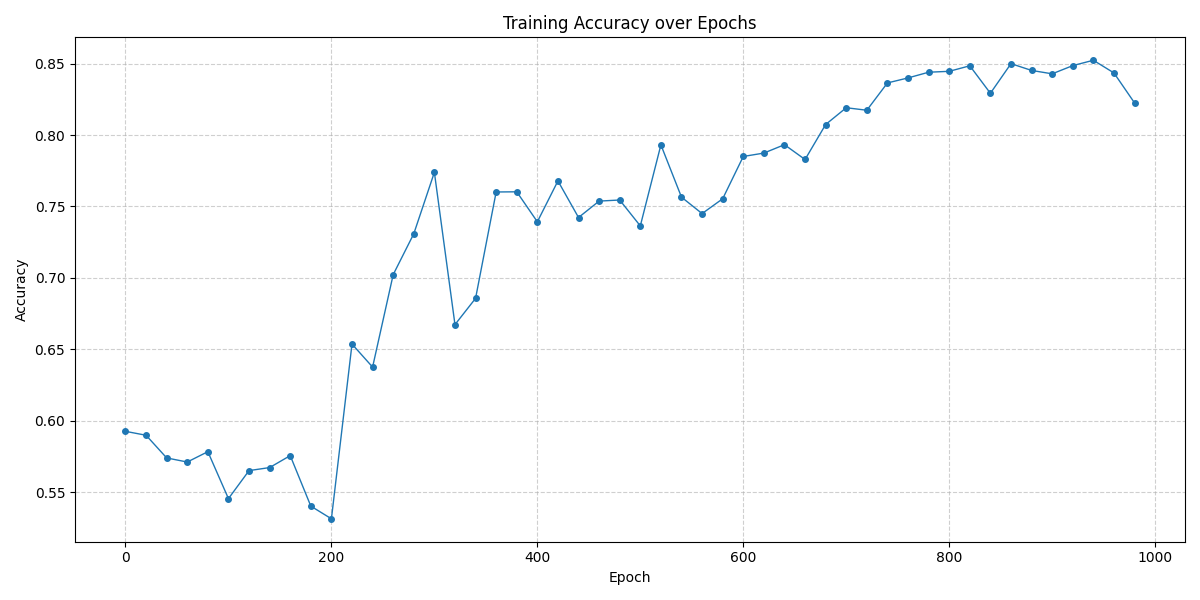}
    \caption{Training accuracy of GNN under hyperparameters in Table \ref{table:hyperparameters}.} 
    \label{fig:GNN Test Accuracy} 
\end{figure}

\end{document}